\documentclass[reqno]{amsart}
\usepackage{bm}
\usepackage{mathrsfs}
\usepackage{amsmath}
\usepackage{amsfonts}
\usepackage{mathtools}
\usepackage{amssymb}
\usepackage{color}
\usepackage{graphicx}
\usepackage{hyperref}
\usepackage{cleveref}
\usepackage{enumitem}
\setlist[enumerate,1]{label=\textup{(\arabic*)}}

\hypersetup{
	colorlinks=true,
	linkcolor=blue,
	filecolor=red,      
	urlcolor=cyan,
	citecolor=green,
}

\newtheorem{theorem}{Theorem}[section]
\newtheorem{corollary}[theorem]{Corollary}
\newtheorem{lemma}[theorem]{Lemma}
\newtheorem{proposition}[theorem]{Proposition}

\newtheorem{definition}[theorem]{Definition}
\newtheorem{question}[theorem]{Question}
\newtheorem{conjecture}[theorem]{Conjecture}

\newtheorem{example}[theorem]{Example}
\numberwithin{equation}{section}

\def\Val{\mathrm{Val}}
\def\val{\mathrm{val}}
\def\ZVal{\mathrm{ZVal}}
\def\QM{\mathrm{QM}}
\def\qm{\mathrm{qm}}
\def\lct{\mathrm{lct}}
\def\Arn{\mathrm{Arn}}
\def\ZV{\mathrm{ZV}}

\def\T{\mathcal{T}}

\def\wt{\widetilde}

\def\fk{\mathfrak}
\def\bf{\mathbf}
\def\cal{\mathcal}
\def\scr{\mathscr}

\def\bfQ{\mathbf{Q}}
\def\bfZ{\mathbf{Z}}
\def\bfR{\mathbf{R}}

\def\bfC{\mathbf{C}}
\def\bfB{\mathbf{B}}
\def\V{\mathcal{V}}

\DeclareMathOperator{\ord}{ord}
\DeclareMathOperator{\Spec}{Spec}
\DeclareMathOperator{\ddc}{dd^c}

\begin{document}

\title[Algebraic Zhou valuations]
 {Algebraic Zhou valuations}

\author{Shijie Bao}
\address{Shijie Bao: Institute of Mathematics, Academy of Mathematics and Systems Science, Chinese Academy of Sciences, Beijing 100190, China}
\email{bsjie@amss.ac.cn}

\author{Qi'an Guan}
\address{Qi'an Guan: School of
Mathematical Sciences, Peking University, Beijing 100871, China}
\email{guanqian@math.pku.edu.cn}

\author{Lin Zhou}
\address{Lin Zhou: Institute of Algebraic Geometry, Leibniz University Hannover, Hannover 30167, Germany}
\email{zhou@math.uni-hannover.de}

\thanks{}



\date{\today}

\dedicatory{}

\commby{}


\begin{abstract}
    In this paper, we generalize Zhou valuations, originally defined on complex domains, to the framework of general schemes. We demonstrate that an algebraic version of the Jonsson--Musta\c{t}\u{a} conjecture is equivalent to the statement that every Zhou valuation is quasi-monomial. By introducing a mixed version of jumping numbers and Tian functions associated with valuations, we obtain characterizations of a valuation being a Zhou valuation or computing some jumping number using the Tian functions. Furthermore, we establish the correspondence between Zhou valuations in algebraic settings and their counterparts in analytic settings.
\end{abstract}

\maketitle

\setcounter{tocdepth}{1}

\tableofcontents
\section{Introduction}

In the present paper, we work on \emph{separated regular connected Noetherian excellent schemes} over $k\coloneqq\bfQ$. We extend the definition of Zhou valuations, originally introduced in \cite{BGMY23}, from the ring of germs of holomorphic functions on $\bfC^n$ to such schemes, while all the settings are completely algebraic. First, we recall some concepts that are closely related to the definition and properties of algebraic Zhou valuations.

\subsection{Background}
\subsubsection{Jumping number and graded sequence of ideals}
Let $X$ be a regular algebraic variety defined over a field with characteristic zero, and let $\fk{a}$ and $\fk{q}$ be nonzero ideals on $X$. The \emph{jumping number} (associated to $\fk{q}$), denoted by $\lct^\fk{q}(\fk{a})$, is defined as
\[\lct^\fk{q}(\fk{a})\coloneqq\inf_E\frac{A(\ord_E)+\ord_E(\fk{q})}{\ord_E(\fk{a})} =\min\{\lambda\colon \fk{q}\not\subseteq\cal{J}(\fk{a}^\lambda)\},\]
where $E$ runs through all prime divisors over $X$, $\cal{J}(\fk{a})$ is the \emph{multiplier ideal} associated to $\fk{a}$ (cf. \cite[Part Three]{LarII04}) and the number $A(\ord_E)-1$ is the coefficient of the divisor $E$ on $Y$ in the relative canonical class $K_{Y /X}$. When $q=\cal{O}_X$, the jumping number specializes to the \emph{log canonical threshold} (see \cite[Section 8]{Kol97}), denoted by $\lct(\fk{a})$. Both jumping numbers and log canonical thresholds are important birational invariants in algebraic geometry, which are algebro-geometric analogues of jumping numbers (cf. \cite{ELSV04}) and complex singularity exponents (cf. \cite{Tian87,Kol97,DK01}) in complex analysis.

Given a nonzero \emph{graded sequence of ideals} $\fk{a}_{\bullet}$, one can define an \emph{asymptotic jumping number} $\lct^{\fk{q}}(\fk{a}_{\bullet})$ as 
\[0<\lct^{\fk{q}}(\fk{a}_{\bullet})\coloneqq  \sup_m m\cdot \lct^{\fk{q}}(\fk{a}_m)=\lim_{m\to \infty}m\cdot\lct^{\fk{q}}(\fk{a}_m).\]
In \cite{JM12} Jonsson--Musta\c{t}\u{a} showed that if we set $v(\fk{a}_{\bullet})\coloneqq \lim_{m\to \infty}v(\fk{a}_m)/m$, then
\[\lct^{\fk{q}}(\fk{a}_{\bullet})=\inf_{v}\frac{A(v)+v(\fk{q})}{v(\fk{a}_{\bullet})},\]
where $v$ runs all the nontrivial valuations on $X$ and $A(v)$ is the \emph{log discrepancy} of $v$ (see \cite[Section 5]{JM12} or \Cref{subsetion-Log discrepancy and jumping numbers}). Using the compactness arguments, they proved that there must exist a valuation $v$ computing $\lct^{\fk{q}}(\fk{a}_{\bullet})$ if $\lct^{\fk{q}}(\fk{a}_{\bullet})\neq \infty$, i.e. $\lct^{\fk{q}}(\fk{a}_{\bullet})=\big(A(v)+v(\fk{q})\big)/v(\fk{a}_{\bullet})$.

\subsubsection{Jonsson--Musta\c{t}\u{a}'s conjectures and Demailly's strong openness conjecture}

In \cite{JM12}, Jonsson--Musta\c{t}\u{a} proposed several conjectures regarding the valuations computing $\lct^{\fk{q}}(\fk{a}_{\bullet})$. We present here a conjecture closely related to our work; for further versions, see \cite{JM12,JM14}.

\begin{conjecture}[Algebraic version of Jonsson--Musta\c{t}\u{a}'s conjecture]\label{conj-weak version JM}
    Let $\fk{a}_{\bullet}$ be a graded sequence of ideals on $X$ and $\fk{q}$ a nonzero ideal on $X$ with $\lct^{\fk{q}}(\fk{a}_{\bullet})<\infty$. There exists a quasi-monomial valuation $v\in \Val_X^*$ that computes $\lct^{\fk{q}}(\fk{a}_{\bullet})$.
\end{conjecture}

While \Cref{conj-weak version JM} was being formulated, the cases for $\dim X\le 2$ were already established in \cite{JM12}. Moreover, Xu \cite{Xu20} made a breakthrough on \Cref{conj-weak version JM}, that is, proving the case $\fk{q}=\cal{O}_X$ of \Cref{conj-weak version JM}.

One very important purpose of Jonsson--Musta\c{t}\u{a} in proposing these conjectures was to study the openness conjecture and the strong openness conjecture through algebraic approaches. In particular, as stated in \cite[Theorem D and D']{JM14}, the validity of \Cref{conj-weak version JM} would imply the strong openness conjecture, and when $\fk{q}=\cal{O}_X$, the \Cref{conj-weak version JM} regarding the log canonical threshold would imply the openness conjecture. 

We recall \emph{Demailly--Koll\'{a}r's openness conjecture} and \emph{Demailly's strong openness conjecture}. Let $\varphi$ be a plurisubharmonic (psh for short) function on a complex manifold $M$, and let $\cal{J}(\varphi)$ be the \emph{multiplier ideal sheaf} of $\varphi$ (\cite{Nad89,Nad90}); see (\ref{equ-MIS}) for the definition. Set $\cal{J}_+(\varphi)\coloneqq \bigcup_{\varepsilon>0}\cal{J}\big((1+\varepsilon)\varphi\big)$. Then Demailly--Koll\'{a}r's openness conjecture (\cite{DK01}) can be formulated as:

\vspace{0.3cm}

\paragraph{\emph{Openness Conjecture}}  If $\cal{J}(\varphi)=\cal{O}_M$, then $\cal{J}_+(\varphi)=\cal{J}(\varphi)$.

\vspace{0.3cm}

When the restriction of $\cal{J}(\varphi)$ is removed, this is Demailly's strong openness conjecture (see \cite{Dem00,AMAG}):

\vspace{0.3cm}

\paragraph{\emph{Strong Openness Conjecture}} $\cal{J}_+(\varphi)=\cal{J}(\varphi)$.

\vspace{0.3cm}

Breakthroughs in the openness conjecture and strong openness conjecture initially emerged from algebraic pathways. Using the valuative tree theory (cf. \cite{FJ04}), Favre--Jonsson \cite{FJ05a} (see also \cite{FJ05b}) proved the $2$-dimensional case of the openness conjecture, and Jonsson--Musta\c{t}\u{a} \cite{JM12} established the $2$-dimensional case of the strong openness conjecture by proving the $2$-dimensional case of \Cref{conj-weak version JM}. Berndtsson \cite{Bern13} proved the openness conjecture by using the complex Brunn-Minkowski inequality (see also \cite{Bern15}). After that, Guan-Zhou \cite{GZ15soc} proved the strong openness conjecture by movably using the Ohsawa-Takegoshi $L^2$ extension theorem (\cite{OT87}). Meanwhile, Xu’s breakthrough (\cite{Xu20}) on \Cref{conj-weak version JM} for the log canonical threshold case provided an algebraic proof of the openness conjecture without dimensional restrictions.

More precisely, the (algebraic version of) Jonsson-Musta\c{t}\u{a}'s conjecture (\Cref{conj-weak version JM}) can imply the strong openness conjeture through implying the following conjecture (\cite[Conjecture B']{JM14}):

\vspace{0.3cm}

\paragraph{\emph{(Analytic version of) Jonsson-Musta\c{t}\u{a}'s conjecture}} If $c_o^{\fk{q}}(\varphi)<+\infty$, then
\[\frac{1}{r^2}\mu\big(\{c_o^{\fk{q}}(\varphi)\varphi-\log|\fk{q}|<\log r\}\big)\]
has a positive lower bound independent of $r\in (0,1]$, where $\varphi$ is a psh germ at $o\in\bfC^n$, $\fk{q}\subseteq\cal{O}_o$ is a nonzero ideal, $\mu$ is the Lebesgue measure on $\bfC^n$, and
\[c_o^{\fk{q}}(\varphi)\coloneqq\sup\big\{c\ge 0\colon |\fk{q}|^2e^{-2c\varphi} \text{ is locally $L^1$ integrable near } o\big\}\]
is the \emph{jumping number} (see \cite{DK01}). 

\vspace{0.3cm}

This (analytic version of) Jonsson-Musta\c{t}\u{a}'s conjecture was proved by Guan--Zhou based on the truth of the strong openness conjecture (\cite{GZ15eff}), and a proof without using the strong openness property can be found in \cite{BGY25}, with a sharp effectiveness result related to this conjecture obtained (see also \cite{BG24a}). The effectiveness results of the strong openness conjecture can also be referred to \cite{GZ15eff} (by giving the $L^2$ estimates to the solutions of $\bar{\partial}$-equations), \cite{Guan19} (by the concavity of minimal $L^2$ integrals), and \cite{BG22,BG24b} (by the log-plurisubharmonicity of fiberwise $\xi$-Bergman kernels, a variant of Berndtsson's log-plurisubharmonicity of fiberwise Bergman kernels \cite{Bern06}).

Although the strong openness conjecture and the analytic version of Jonsson-Musta\c{t}\u{a}'s conjecture have been proved, the algebraic version of Jonsson--Musta\c{t}\u{a}'s conjecture still remains open in the case of nontrivial ideal $\fk{q}$ (i.e. $\fk{q}\neq\cal{O}_X$ in \Cref{conj-weak version JM}).

\subsubsection{Zhou weight, Zhou number and Zhou valuation}

Recall the definitions of \emph{Zhou weight}, \emph{Zhou number} and \emph{Zhou valuation} introduced in \cite{BGMY23}. Let $f_0=(f_{0,1},\ldots,f_{0,m})$ be a vector of holomorphic functions near the origin $o$ in $\bfC^n$, and let $\varphi_0$ be a psh function near $o$ such that $|f_0|^2e^{-2\varphi_0}$ is integrable near $o$.

\begin{definition}[see {\cite[Definition 1.2]{BGMY23}}]\label{def-Zhou.weight.number.val}
We call that a psh function $\Phi$ near $o$ is a (local) \emph{Zhou weight} related to $|f_0|^2e^{-2\varphi_0}$ near $o$, if the following three statements hold:
\begin{enumerate}
\item$|f_0|^2e^{-2\varphi_0}|z|^{2N_0}e^{-2\Phi}$ is integrable near $o$ for large enough $N_0\gg 0$;
\item $|f_0|^2e^{-2\varphi_0}e^{-2\Phi}$ is not integrable near $o$;
\item For any psh function $\varphi'\geq\Phi+O(1)$ near $o$ such that
$|f_0|^2e^{-2\varphi_0}e^{-2\varphi'}$ is not integrable near $o$, $\varphi'=\Phi+O(1)$ holds.
\end{enumerate}
The relative type of a psh germ $u$ at $o$ to a Zhou weight $\Phi$:
\[\sigma(u,\Phi)\coloneqq\sup\{c\ge 0\colon u\le c\Phi +O(1) \text{ near } o\}\]
is called the \emph{Zhou number} of $u$ to $\Phi$.

Call a valuation $v$ on $\cal{O}_o$ a \emph{Zhou valuation} related to $|f_0|^2e^{-2\varphi_0}$, if there exists a Zhou weight $\Phi$ related $|f_0|^2e^{-2\varphi_0}$ near $o$ such that $\sigma(\log |f|,\Phi)=v(f)$ for all $(f,o)\in\cal{O}_o$.
\end{definition}

The existence of Zhou valuations is based on the truth of the strong openness conjecture. By establishing an integral expression of Zhou numbers (see \cite[Theorem 1.6]{BGMY23}), it was proved in \cite{BGMY23} that the Zhou numbers satisfying both \emph{tropical additivity} and \emph{tropical multiplicativity}, inducing the valuation structure, which makes that every Zhou weight can be associated with a Zhou valuation (see \cite[Corollary 1.9]{BGMY23}).

The notions of Zhou weights and Zhou numbers were used to analytically reprove and generalize the characterization of singularities of plurisubharmonic germs established by Boucksom--Favre--Jonsson in \cite{BFJ08}, while Zhou valuations were also used to characterize if one holomorphic germ can be divided by another (see \cite[Theorem 1.11 and Corollary 1.14]{BGMY23}). In this context, \emph{Tian functions} were introduced as technical instruments to facilitate the construction and analysis of these structures. These constructions revealed connections between the singularities in several complex variables and valuation theory.

Furthermore, \cite{BGY23} demonstrated some connections between Zhou valuations and Jonsson--Musta\c{t}\u{a}'s conjectures in complex analytic framework. For example, it was showed that for any Zhou valuation $\nu$, there exists a graded sequence of ideals $\fk{a}_{\bullet}$ and a nonzero ideal $\fk{q}$ such that $\nu$ $\scr{A}$-computes (see \cite[Definition 1.5]{BGY23}) the jumping number $\lct^{\fk{q}}(\fk{a}_{\bullet})$. The results in \cite{BGY23} reveal not only the possibility but also the necessity of extending the notion of Zhou valuations beyond the analytic setting into a broader algebro-geometric framework.

\subsection{Algebraic Zhou valuation and main results}
Motivated by these observations, the present paper undertakes a systematic algebraic study of Zhou valuations. The idea of generalizing Zhou valuations originates from the answer in \cite{BGY23} to a question asked by Jonsson, which concerns how to characterize a valuation on $\cal{O}_o$ appearing as a Zhou valuation (see \cite[Question 1.10 and Proposition 1.12]{BGY23}). Based on this characterization, we realize that it is reasonable to define \emph{algebraic Zhou valuations} related to a given nonzero ideal \(\fk{q}\) on more general scheme $X$, as follows:

\begin{definition}[=\Cref{def-Zhou.val}]\label{def-intro Zhou val}
    A valuation on $X$ is called an algebraic \emph{Zhou valuation} related to a nonzero ideal $\fk{q}$ if it is a maximal element among the valuations $v$ satisfying $\lct^{\fk{q}}(\fk{a}_{\bullet}^v)\le 1$, where $\fk{a}_{\bullet}^v$ is the graded sequence of ideals associated to the valuation $v$ (see (\ref{equ-graded sequence ass v})).
\end{definition}

 The algebraic Zhou valuations can be seen as the analogue of the original Zhou valuations and we will prove in fact there is direct correspondence in \Cref{section11} (see also \Cref{thm-main thm C}).
 
 We establish a direct connection between algebraic Zhou valuations and the Jonsson--Musta\c{ț}\u{a} conjecture (\Cref{conj-weak version JM}), and we show that Zhou valuations satisfy the following properties in the most general setting considered in article \cite{JM12}.

\begin{theorem}\label{thm-main thmA}
    Let $\fk{q}$ be a nonzero ideal on $X$, and $\fk{a}_{\bullet}$ be a graded sequence of ideal on $X$ with $\lct^{\fk{q}}(\fk{a}_{\bullet})<+\infty$. Then the following statements hold.
    \begin{enumerate}
        \item (\Cref{cor-ZV.computes}) There exists a Zhou valuation related to $\fk{q}$ computing $\lct^{\fk{q}}(\fk{a}_{\bullet})$.
        \item (\Cref{thm-A(v)=1-v(q)}) Any Zhou valuation $v$ related to $\fk{q}$ computes $\lct^{\fk{q}}(\fk{a}_{\bullet}^v)$, and especially $A(v)=1-v(\fk{q})$.
        \item (\Cref{thm-equiv.weak.ZVal=QM}) The \Cref{conj-weak version JM} is equivalent to that all Zhou valuations are quasi-monomial.
        \item (\Cref{cor-ZV.OX.qm}+\Cref{cor-dim.le.2}) If $\fk{q}=\cal{O}_X$ or $\dim X\le 2$, then all Zhou valuations related to $\fk{q}$ are quasi-monomial.
    \end{enumerate}
\end{theorem}


By \Cref{thm-main thmA}, regarding as \Cref{conj-weak version JM}, it is natural to investigate the question that how to determine whether a given valuation is a Zhou valuation. To address this question, we develop algebraic analogues of the objects studied in \cite{BGMY23}, including \emph{mixed jumping numbers} and \emph{(algebraic) Tian functions}. For given nonzero ideals $\fk{q}$, $\fk{q}'$ and $\fk{a}$, we define the mixed version of jumping number as 
\begin{equation*}
 0<\lct(\fk{q}, t\cdot\fk{q}';\fk{a})\coloneqq \inf_E\frac{A(\ord_E)+\ord_E(\fk{q})+t\cdot\ord_E(\fk{q}')}{\ord_E(\fk{a})},
\end{equation*}
where the infimum is over all divisors $E$ over $X$, and $t\in(-\varepsilon_0,+\infty)$ with small enough $\varepsilon_0>0$ (to allow $t<0$ is very important in the latter discussions). Similarly, for a nonzero graded sequence of ideals $\fk{a}_{\bullet}$, we can then define an \emph{asymptotic mixed version jumping number} as
\begin{align*}
 \lct(\fk{q}, t\cdot\fk{q}';\fk{a}_{\bullet})\coloneqq\lim_{m\to\infty}m\cdot\lct(\fk{q}, t\cdot\fk{q}';\fk{a}_m)=\sup_{m} m\cdot\lct(\fk{q}, t\cdot\fk{q}';\fk{a}_m)\in (0,+\infty].
\end{align*}
For any nonzero ideals $\fk{q}$ and $\fk{q}'$ and for any nonzero graded sequence of ideals $\fk{a}_{\bullet}$, the function
\[t\longmapsto \lct(\fk{q}, t\cdot \fk{q}'; \fk{a}_{\bullet})\]
is called as the (algebraic) Tian function, which is the analogue of the Tian function introduced in \cite{BGMY23}.

We will establish a criterion, formulated via algebraic Tian functions, for determining whether a given valuation is a Zhou valuation associated to some nonzero ideal. This result can be seen as parallel with \cite[Proposition 4.8]{BGMY23}. In addition, a characterization on the valuations computing some jumping number will also be demonstrated, while the proof of this result is much farther away from the analytic framework. We are going to prove

\begin{theorem}\label{thm-main thmB}
    Let $\fk{q}$ be a nonzero ideal on $X$ and $v\in\Val^*_X$ with $A(v)<+\infty$. Assume $\lct^{\fk{q}}(\fk{a}_{\bullet}^v)<+\infty$. We have the followings:
    \begin{enumerate}
        \item (\Cref{prop-compute.linear}+\Cref{prop-linear.imply.compute}) The valuation $v$ computes the jumping number $\lct^{\fk{q}}(\fk{a}_{\bullet}^v)$ if and only if the function 
        \[t\longmapsto \lct(\fk{q}, t\cdot \fk{q}'; \fk{a}_{\bullet}^v)\]
        is linear on $[0,+\infty)$ for every nonzero ideal $\fk{q}'$ on $X$.
         \item (\Cref{thm-ZV.lct.differential}+\Cref{prop-converse.differential.ZVal}) The valuation $v$ is a Zhou valuation related to $\fk{q}$ if and only if $\lct^{\fk{q}}(\fk{a}_{\bullet}^v)=1$ and the function
        \[(-\varepsilon_{\fk{q}'},+\infty)\ni t\longmapsto \lct(\fk{q}, t\cdot \fk{q}'; \fk{a}_{\bullet}^v)\]
        is linear on $[0,+\infty)$ and differentiable at $t=0$ for every nonzero ideal $\fk{q}'$ on $X$, where $\varepsilon_{\fk{q}'}>0$ are sufficiently small. 
    \end{enumerate}
\end{theorem}


At last, it is natural to consider if one can go back to the (analytic) Zhou valuations (see \Cref{def-analytic Zhou valuation}) introduced in \cite{BGMY23} from the algebraic Zhou valuations defined in the present paper. In fact, if we let $\mathcal{O}_o$ be the local ring of germs of holomorphic functions at the origin $o$ in $\bfC^n$, then restricted to $X = \mathrm{Spec}\, \mathcal{O}_o$, the algebraic and analytic Zhou valuations coincide, that is

\begin{theorem}[=\Cref{thm-Zhou.valuation.coincide}]\label{thm-main thm C}
    Let $\fk{q}\subseteq\cal{O}_o$ be a nonzero ideal, and let $v$ be a valuation on $\cal{O}_o$ with $v(\fk{m})>0$. Then $v$ is an algebraic Zhou valuation related to $\fk{q}$ if and only if $v$ is an analytic Zhou valuation related to $|\fk{q}|^2$.
\end{theorem}

Actually, Theorem \ref{thm-Zhou.valuation.coincide} can be proved combining with the definition of algebraic Zhou valuation and the results in \cite{BGMY23,BGY23}, but for the sake of completeness, we provide a brief proof in \Cref{section11}, which involves much more algebraic methods, and is far different with the methods in \cite{BGMY23,BGY23}. We do not assume the existence of Zhou weights in the proof, and the truth of the strong openness property of multiplier ideal sheaves, which is the foundation of the proofs in \cite{BGMY23}, will only be used in the last step.

To establish the equivalence between these two notions, we start with an (algebraic) Zhou valuation $v$ on $\cal{O}_o$, and construct the corresponding Zhou weight $\Phi_v^{\star}$. Then we show that the Zhou number $\sigma(\log|f|,\Phi_v^{\star})$ coincides with $v(f)$. We also use this correspondence to provide algebraic reproofs of several results concerning Zhou weights from \cite{BGMY23} (see \Cref{cor-Zhou.weight.to.Zhou.val,cor-Zhou.weight.tame,ex-Zhou.weight}).

\subsection*{Organization}

The structure of the paper is as follows. In \Cref{section2}, we recall some definitions and properties of valuations and jumping numbers that will be frequently used in subsequent sections. In \Cref{section3}, we introduce the Zhou valuations, and in \Cref{section4}, we prove \Cref{thm-main thmA}. To prepare for the proof of \Cref{thm-main thmB}, we define a mixed version of jumping numbers in \Cref{section5} and extend certain properties of jumping numbers to this setting. \Cref{section6} is devoted to the proof of \Cref{thm-main thmB} (1). \Cref{section7} develops technical tools necessary for the proof of \Cref{thm-main thmB} (2), via the enlargement of graded sequences, and we prove \Cref{thm-main thmB} (2) in \Cref{section8}. In \Cref{section9}, we show that the cone of Zhou valuations is dense in the space of valuations. In \Cref{section10}, we provide characterizations of singularities of graded sequences of ideals via Zhou valuations.  Finally, in \Cref{section11}, we describe the correspondence between algebraic and analytic Zhou valuations, and complete the proof of \Cref{thm-main thm C}.

\subsection*{Notations and Conventions}
\begin{itemize}[leftmargin=1em]
    \item We always assume that $X$ is a \emph{separated}, \emph{regular}, \emph{connected}, and \emph{Noetherian excellent scheme} over $k\coloneqq\bfQ$. A Noetherian scheme $X$ is excellent if $X$ can be covered by spectra of excellent rings. 
    \item A \emph{log-smooth pair} over $X$ is a pair $(Y,D)$ with $Y$ regular and $D$ a reduced effective simple normal crossing (SNC for short) divisor, together with a proper birational morphism $\pi\colon Y\to X$ which is an isomorphism outside $\mathrm{Supp}(D)$.  Moreover, we say $(Y,D)\preceq (Y',D')$ if there is a birational morphism $\phi\colon Y'\to Y$ over $X$ such that we have  $\mathrm{Supp}(\phi^*(D))\subseteq\mathrm{Supp}(D')$. 
    \item Let $\fk{a}\subseteq \mathcal{O}_X$ be a nonzero ideal sheaf on $X$. A \emph{log resolution of $\fk{a}$} is a proper birational morphism $\pi\colon Y\to X$ such that $Y$ is regular and 
\begin{equation*}
    \pi^{-1}\fk{a}\coloneqq \fk{a}\cdot\mathcal{O}_Y=\mathcal{O}_Y(-D),
\end{equation*}
where $D$ is an effective divisor on $Y$ such that $D+\mathrm{except}(\pi)$ has simple normal
crossing support.
\item All the ``Zhou valuation" appears in \Cref{section2} to \Cref{section10} are ``algebraic Zhou valuation", and since both the algebraic Zhou valuation and the analytic Zhou valuation will appear in \Cref{section11}, we will make a clear distinction in terminology between the two in this section.
\end{itemize}

\section{Preliminaries}\label{section2}
In this section, for the convenience of readers, we recall some basic knowledge about valuations and jumping numbers. Moreover, we have compiled the main results and definitions from \cite{JM12} that will appear repeatedly in our subsequent sections.

\subsection{The space of valuations}
\subsubsection{Valuations}
A \emph{real valuation} of the function field of $X$ with center $\xi$ on $X$ is a  map $v\colon K(X)^*\to \bfR$ satisfying the following:
\begin{enumerate}
    \item $v(fg)=v(f)+v(g)$ for all $f,g$;
    \item $v(f+g)\geq \min\{v(f),v(g)\}$ for all $f,g$;
    \item $v(f)=0$ for all $f\in \bfQ^*$;
    \item there is an inclusion of local rings $\mathcal{O}_{X,\xi}\hookrightarrow \mathcal{O}_v$, where $\mathcal{O}_{X,\xi}$ is the local ring of $\xi\in X$ and  $\mathcal{O}_v\coloneqq \{f\in K(X)\colon v(f)\geq 0\}$ is the valuation ring associated to the valuation $v$.
\end{enumerate}
We set $v(0)=+\infty$. The \emph{trivial valuation} is the valuation whose restriction to $K(X)^*$ is zero. Note that the center of a valuation, if it exists, is unique as $X$ is separated, and we denote the center of the valuation $v$ by $c_X(v)$. Let $\Val_{X,\xi}$ be the set of nontrivial real valuations of $K(X)$ with the center $\xi\in X$. Let  $\Val_X$ be the set of all real valuations of the function field $K(X)$ of $X$ that admits a center on $X$, and we denote $\Val_X^*\subseteq\Val_X$ the subset of nontrivial valuations with a center on $X$. 

For each valuation $v$ with a center $\xi$, if $\fk{a}$ is an ideal sheaf on $X$, then we write
\begin{equation*}
    v(\fk{a})\coloneqq \min_{f\in \fk{a}\cdot \mathcal{O}_{X,\xi}}v(f).
\end{equation*}
Then this function satisfies the equations
\begin{equation*}
    v(\fk{a}\cdot\fk{b})=v(\fk{a})+v(\fk{b})\quad \text{and} \quad v(\fk{a}+\fk{b})=\min\{v(\fk{a}),v(\fk{b})\}.
\end{equation*}
Moreover, $v(\fk{a})>0$ if and only if $\xi\in V(\fk{a})$. For any $v,w\in\Val_X$, we say $v\le w$ if $v(\fk{a})\le w(\fk{a})$ for all nonzero ideals on $X$. 
\subsubsection{Divisorial valuations}
 We say $E$ is a divisor over $X$, if $E$ is a prime divisor on a normal scheme $Y$ with a proper birational morphism $\pi\colon Y\to X$. For a divisor $E$ over $X$, we define the \emph{divisorial valuation} $\ord_E$ that sends each rational function $f$ in $K(X)^*=K(Y)^*$ to its order of vanishing along $E$. By resolution of singularities (see \cite{Hi64, Tem08, Tem18}), we may always assume that $E$ is a prime divisor on a regular scheme $Y$. A valuation $v\in \Val_X$ is a divisorial valuation if there is a divisor $E$ over $X$ and $\lambda\in \bfR_{>0}$ such that $v(f)=\lambda\cdot \ord_E(f)$ for all $f\in K(X)^*$. 

\subsubsection{Quasi-monomial valuations}
Let $\pi\colon Y\to X$ be a proper birational morphism with regular and connected $Y$, and let $y=(y_1,\cdots,y_r)\in \mathcal{O}_{Y,\eta}^{\oplus r}$ be a system of parameters at a point $\eta\in Y$. For $\alpha,\beta\in \bfR^r$, we set $\langle\alpha,\beta\rangle\coloneqq \sum_{i=1}^r\alpha_i\beta_i$.
\begin{definition}
Set $Y$ and $y$ as above. The \emph{quasi-monomial valuation} associated with $y$ and $\alpha=(\alpha_1,\cdots,\alpha_r)\in \bfR^r_{\geq 0}$ is a valuation $\val_{\alpha}\colon K(X)\to \bfR\cup \{+\infty\}$ such that 
\begin{equation*}
    \val_{\alpha}(f)=\min\{\langle \alpha,\beta\rangle\colon c_{\beta}\neq 0\}
\end{equation*}
for all $f\in \mathcal{O}_{Y,\eta}$, where  $f=\sum_{\beta\in \bfZ^r_{\geq 0}}c_{\beta}y^{\beta}$ under the completion of a local ring $\mathcal{O}_{Y,\eta}\to \widehat{\mathcal{O}_{Y,\eta}}$.
\end{definition}
Given a quasi-monomial valuation $v=\val_{\alpha}$, the center of $v$ on $Y$ is the generic point $\eta'$ of $\bigcap_{i}V(y_i)$ where $i$ runs through all $i$ such that $\alpha_i\neq 0$, hence the center of $v$ on $X$ is $\pi(\eta')$; see \cite[Proposition 3.1]{JM12}.

Let $v\in \Val_X^*$ with a center $\xi\in X$. The \emph{rational rank} of $v$ is defined as 
\begin{equation*}
    \mathrm{rat.rk}(v)\coloneqq\dim_{\bfQ}(v(K(X)^*)\otimes _{\bfZ}\bfQ);
\end{equation*}
and the \emph{transcendence degree} of $v$ is defined as
\begin{equation*}
    \mathrm{tr.deg}_X(v)\coloneqq \mathrm{tr.deg}(\kappa_v/\kappa(\xi)),
\end{equation*}
where $\kappa_v$ and $\kappa(\xi)$ are the \emph{residue fields} of the valuation ring $\mathcal{O}_v$ and $\mathcal{O}_{X,\xi}$ respectively. Write $\dim(\mathcal{O}_{X,\xi})$ as the \emph{Krull dimension} of the local ring $\mathcal{O}_{X,\xi}$. The \emph{Abhyankar inequality} says we have the following inequality for all $v\in \Val_{X,\xi}$ (see \cite{Va97}):
\begin{equation*}
   \mathrm{rat.rk}(v)+ \mathrm{tr.deg}_X(v)\leq \dim(\mathcal{O}_{X,\xi}),
\end{equation*}
and a valuation satisfying the equality is called an \emph{Abhyankar valuation}.
\begin{proposition}[cf.{\cite[Proposition 3.7, Remark 3.9]{JM12}, \cite{ELS03}}]
 A valuation $v\in \Val_X$ is quasi-monomial if and only if $v$ is Abhyankar.  A valuation $v\in \Val_X^*$ is divisorial if and only if $v$ is quasi-monomial with $\mathrm{rat.rk}(v)=1$.
\end{proposition}
Given a log-smooth pair $(Y,D)$ over $X$, let $\QM_{\eta}(Y,D)$ be the set of all quasi-monomial valuation $v$ that can be described at the point $\eta\in Y$ with respect to coordinates $y_1,\cdots,y_r$ such that each $y_i$ defines at $\eta$ an irreducible component of~$D$. We put $\QM(Y,D)\coloneqq \bigcup_{\eta}\QM_{\eta}(Y,D)$ where $\eta$ runs through all generic point of a connected component of the intersection of some of $D_i$. For every quasi-monomial valuation $v$, there exists a log-smooth pair $(Y,D)$ over $X$ such that $v\in \QM(Y,D)$ by \cite[Remark 3.4]{JM12}.
\subsubsection{Topology of the space of valuations}
Let $\tau$ be \emph{the weakest topology} on $\Val_X$ such that the evaluation map 
\begin{equation*}
    \varphi_{\fk{a}}\colon \Val_X\to \bfR \qquad v\mapsto \varphi_{\fk{a}}(v)\coloneqq v(\fk{a})
\end{equation*}
is continuous for all nonzero ideals $\fk{a}$ on $X$. Under this topology, the map
\begin{equation*}
    c_X\colon \Val_X\to X \qquad\qquad v\mapsto c_X(v)
\end{equation*}
is \emph{anti-continuous}, i.e., the inverse image of any open subset is closed. Indeed, if $U\subseteq X$ is an open subset of $X$, let $J$ be the ideal sheaf of $X\setminus U$ with reduced scheme structure. Then
\begin{equation*}
    c_X^{-1}(U)=\Val_U\coloneqq \{v\in \Val_X\colon v(J)=0\},
\end{equation*}
because $v(J)\geq 0$ for all real valuation $v$ with a center on $X$ and $v(J)>0$ if and only if $c_X(v)\in V(J)=(X\setminus U)_{\mathrm{red}}$.

Applying this topology, we provide more descriptions of $\QM(Y,D)\subseteq \Val_X$. 
\begin{proposition}[{\cite[Section 4.2]{JM12}}]\label{prop-topology of ValX}
Let $(Y,D)$ be a log-smooth pair over $X$. Then we have the followings:
\begin{enumerate}
    \item If $\eta$ is the generic point of a connected component of the intersection of $r$ irreducible components $D_1,\cdots,D_r$ of $D$, then the map
    \begin{equation*}
        \QM_{\eta}(Y,D)\to \bfR^r\qquad v\mapsto (v(D_1),\cdots,v(D_r))
    \end{equation*}
    gives a homeomorphism onto the cone $\bfR^r_{\geq 0}$.
    \item The space $\QM(Y,D)$ is the union of finitely many closed simplicial cones $\QM_{\eta}(Y,D)$.
\end{enumerate}
\end{proposition}
\subsubsection{Retraction maps and structure theorem}
We define the \emph{retraction map} with respect to a log-smooth pair $(Y,D)$ over $X$ as the continuous map 
\begin{equation*}
  r_{Y,D}\colon\Val_X\to \QM(Y,D)  
\end{equation*}
 that maps a valuation $v$ to the quasi-monomial valuation $w\coloneqq r_{Y,D}(v)$ satisfying $w(D_i)=v(D_i)$ for every irreducible component $D_i$ of $D$ by \Cref{prop-topology of ValX} (1). 
 
 For every valuation $v\in \Val_X$ and any ideal $\fk{a}$ on $X$, we have $r_{Y,D}(v)(\fk{a})\leq v(\fk{a})$ with equality if $(Y,D)$ gives a log resolution of $\fk{a}$ (see \cite[Corollary 4.8]{JM12}). Applying retraction maps, the space of valuations could be achieved as a projective limit of simplicial cone complexes.
 \begin{proposition}[cf. {\cite[Section 4.4]{JM12}}]\label{prop-structure thm of ValX}
     The retraction maps induce a homeomorphism
     \begin{equation*}
         r\colon \Val_X\to \varprojlim_{(Y,D)}\QM(Y,D)
     \end{equation*}
such that 
\begin{enumerate}
    \item the set of quasi-monomial valuations is dense in $\Val_X$;
    \item the set of divisorial valuations is dense in $\Val_X$.
\end{enumerate}
 \end{proposition}
\subsection{Graded sequences and subadditive systems of ideals}
\subsubsection{Graded sequences}
A \emph{graded sequence} of ideals $\fk{a}_{\bullet}=(\fk{a}_m)_{m\in \bfZ_{>0}}$ is a sequence of ideals on $X$ satisfying $\fk{a}_p\cdot\fk{a}_q\subseteq \fk{a}_{p+q}$ for all $p,q\geq 1$ (\cite{ELS01}). By convention, we put $\fk{a}_0=\mathcal{O}_X$, and we always assume that such a sequence is \emph{nonzero}, i.e., $\fk{a}_m\neq 0$ for some $m>0$. Let $\fk{a}_{\bullet}$ be a nonzero graded sequence and let $v\in \Val_X$. Then we can define 
\begin{equation*}
    v(\fk{a}_{\bullet})\coloneqq \inf_{m}\frac{v(\fk{a}_m)}{m}=\lim_{\substack{m\to \infty\\ \fk{a}_m\neq 0}}\frac{v(\fk{a}_m)}{m}\in \bfR_{\geq 0},
\end{equation*}
which is well-defined as we have $v(\fk{a}_{p+q})\leq v(\fk{a}_p\cdot\fk{a}_q)=v(\fk{a}_p)+v(\fk{a}_q)$.

For any $v\in\Val_X^*$ with center $\xi=c_X(v)\in X$, we can associate $v$ with a {graded sequence} $\fk{a}^v_{\bullet}=\{\fk{a}^v_m\}$ of ideals such that for an affine open subset $U\subset X$
\begin{equation}\label{equ-graded sequence ass v}
\Gamma(U,\fk{a}_m^v)\coloneqq\left\{\begin{array}{ll}
      \{f\in \Gamma(U,\mathcal{O}_X)\colon v(f)\geq m\} & \text{if } \xi\in U \\
      \Gamma(U,\mathcal{O}_X) & \text{if } \xi\notin U
  \end{array} \right.
\end{equation}
for any $m\in\bf{Z}_{\ge 0}$. Below we present a lemma that is repeatedly used in this article.
\begin{lemma}[{\cite[Lemma 2.4]{JM12}}]\label{lem-w(a.bullet.v)=inf}
    Let $\fk{a}_{\bullet}^v$ be the graded sequence of ideals associated to a nontrivial valuation $v\in \Val_X^*$. Then for any $w\in \Val_X$, we have
    \begin{equation*}
        w(\fk{a}_{\bullet}^v)=\inf_{\fk{b}}\frac{w(\fk{b})}{v(\fk{b})},
    \end{equation*}
    where $\fk{b}$ ranges over ideals on $X$ for which $v(\fk{b})>0$. In particular, $v(\fk{a}_{\bullet}^v)=1$, and $w\geq v$ if and only if $w(\fk{a}_{\bullet}^v)\geq 1$.
\end{lemma}
\subsubsection{Subadditive systems}
A \emph{subadditive system} of ideals $\fk{b}_{\bullet}$ is a one-parameter family $(\fk{b}_t)_{t\in \bfR_{>0}}$ of nonzero ideals satisfying $\fk{b}_{s+t}\subseteq \fk{b}_s\cdot \fk{b}_t$ for all $s,t>0$. Again by convention we put $\fk{b}_0=\mathcal{O}_X$. Similar to graded sequences, let $\fk{b}_{\bullet}$ be a nonzero graded sequence and let $v\in \Val_X$. Then we can define 
\begin{equation*}
    v(\fk{b}_{\bullet})\coloneqq \sup_{t}\frac{v(\fk{b}_t)}{t}=\lim_{t\to \infty}\frac{v(\fk{b}_t)}{t}\in \bfR_{\geq 0}\cup \{+\infty\}.
\end{equation*}

Given a graded sequence of ideals  on $X$, denoted by $\fk{a}_{\bullet}$, the \emph{asymptotic multiplier ideals} of $\fk{a}_{\bullet}$ are defined as
\begin{equation*}
    \fk{b}_t\coloneqq \mathcal{J}(\fk{a}_{\bullet}^t)\coloneqq \mathcal{J}(\fk{a}_m^{t/m}),
\end{equation*}
where $m$ is divisible enough and $\mathcal{J}(\fk{a})$ is the \emph{multiplier ideal associated to an ideal sheaf} $\fk{a}$ (see \cite[Definition 1.4]{ELS01} or \cite[Definition 9.2.3]{LarII04}). 

By the definition of $\fk{b}_{\bullet}\coloneqq (\fk{b}_t)_{t>0}$ associated to a graded sequence of ideals $\fk{a}_{\bullet}$, $\fk{b}_{\bullet}$ is a {subadditive system} of ideals and $\fk{a}_m\subseteq \fk{b}_m$ for all $m\in \bfZ_{>0}$. Moreover, there are many interesting relations between $\fk{a}_{\bullet}$ and $\fk{b}_{\bullet}$ (e.g. \Cref{lem-relation between a and b}), and we will introduce them one by one in the other subsections of this section.
\subsection{Log discrepancy and jumping numbers}\label{subsetion-Log discrepancy and jumping numbers}
We now introduce the main objects that we wish to study.
\subsubsection{Log discrepancy} The \emph{log discrepancy} $A(v)$ for every valuation $v\in \Val_X$ is defined as follows, by applying the topological structure of $\Val_X$ (e.g., \Cref{prop-topology of ValX} and \Cref{prop-structure thm of ValX}).

Let $E$ be a divisor over $X$ with a proper birational morphism $\pi\colon Y\to X$. The log discrepancy associated with the divisorial valuation $\ord_E$ is defined as 
\begin{equation*}
    A(\ord_E)\coloneqq \ord_E(K_{Y/X})+1,
\end{equation*}
where $K_{Y/X}$ is the relative canonical divisor of $\pi$. The log discrepancy $A(\ord_E)$ depends on the scheme $X$ but does not depend on the choice of $Y$ (cf. \cite[pp.~40]{Kol13}). We sometimes denote it by $A_X(\ord_E)$ if there is some ambiguity. Then for a divisorial valuation $v=\lambda\cdot \ord_E$ for a rational number $\lambda>0$ and a divisor $E$ over $X$, we define $A(v)\coloneqq\lambda\cdot A(\ord_E)$.

For every quasi-monomial valuation $v\in \QM(Y,D)$ associated with a log-smooth pair $(Y,D)$, one can define (see \cite[(5.1)]{JM12})
\begin{equation}\label{equ-A(v) for QM(Y,D)}
    A_X(v)\coloneqq\sum_{i=1}^N v(D_i)\cdot A_X(\ord_{D_i}),
\end{equation}
where $D_1,\cdots,D_N$ are the irreducible components of $D$. It does not depend on any choices made by \cite[Proposition 5.1]{JM12}. Now for an arbitrary valuation $v\in \Val_X$, we set
\begin{equation*}
   A_X(v)\coloneqq \sup_{(Y,D)}A_X(r_{Y,D}(v))\in \bfR_{\geq 0}\cup \{\infty\},
\end{equation*}
where the supremum is over all log-smooth pairs $(Y,D)$ over $X$. When $v\in \Val_X^*$, then $A(v)>0$. Moreover, $A(v)$ might be $\infty$ for some $v\in \Val_X^*$; see \cite[Remark~5.12]{JM12}.

Below are some basic properties of the log discrepancy function.
\begin{proposition}[cf. {\cite{JM12}}]\label{prop-log dis}
 The log discrepancy function $A\colon \Val_X\to \bfR_{\geq 0}\cup \{\infty\}$  for every $v\in \Val_X$ is well-defined and satisfies the following:
 \begin{enumerate}
     \item $A(r_{Y,D}(v))\leq A(v)$ for each log-smooth pair $(Y,D)$ over $X$, with equality if and only if $v\in \QM(Y,D)$;
     \item if $(Y,D)\preceq (Y',D')$ for log-smooth pairs over $X$, then for all $v\in \Val_X$ we have $A(r_{Y,D}(v))\leq A(r_{Y',D'}(v))$, with equality if and only if $r_{Y',D'}(v)\in \QM(Y,D)$;
     \item the function $A$ is continuous on each $\QM(Y,D)$ and lower semi-continuous on $\Val_X$;
     \item if $\mu\colon X'\to X$ is a proper birational morphism with $X'$ regular, then $A_X(v)=A_{X'}(v)+v(K_{X'/X})$ for all $v\in \Val_X$.
 \end{enumerate}
\end{proposition}
Now we can introduce a relation between a graded sequence of ideals $\fk{a}_{\bullet}$ on $X$ and its associated asymptotic multiplier ideals $\fk{b}_{\bullet}$.
\begin{lemma}[cf. {\cite[Proposition 2.13, Proposition 6.2]{JM12}}]\label{lem-relation between a and b}
Let $\fk{a}_{\bullet}$ be a graded sequence of ideals and let $\fk{b}_{\bullet}$ be the corresponding system of asymptotic multiplier ideals. Then
\begin{enumerate}
\item the system $\fk{b}_{\bullet}$ has \emph{controlled growth}, that is
\begin{equation*}
    \frac{\ord_E(\fk{b}_t)}{t}>\ord_E(\fk{b}_{\bullet})-\frac{A(\ord_E)}{t}
\end{equation*}
for every divisor $E$ over $X$ and every $t>0$. Moreover, for every $v\in \Val_X$, we have
\begin{equation*}
    v(\fk{b}_{\bullet})-\frac{A(v)}{t}\leq \frac{v(\fk{b}_t)}{t}\leq v(\fk{b}_{\bullet}).
\end{equation*}
\item $\ord_E(\fk{a}_{\bullet})=\ord_E(\fk{b}_{\bullet})$ for every divisor $E$ over $X$;
\item $v(\fk{a}_{\bullet})\geq v(\fk{b}_{\bullet})$ for all $v\in \Val_X$, with equality if $A(v)<\infty$.
\end{enumerate}
\end{lemma}
\subsubsection{Jumping numbers}
For every ideal sheaf $\fk{a}$ and every nonzero ideal sheaf $\fk{q}$ on $X$, the \emph{jumping number} of $\fk{a}$ with respect to $\fk{q}$ is defined by
\begin{equation*}
    \lct^{\fk{q}}(\fk{a})\coloneqq\min\{\lambda\geq 0\colon \fk{q}\not\subseteq \mathcal{J}(\fk{a}^{\lambda})\}.
\end{equation*}
By convention, we put $\lct^{\fk{q}}(\mathcal{O}_X)=\infty$, and when $\fk{q}=\mathcal{O}_X$ we write it as $\lct(\fk{a})$. Moreover, the \emph{Arnold multiplicity} of $\fk{a}$ with respect to $\fk{q}$ is defined as 
\begin{equation*}
    \Arn^{\fk{q}}(\fk{a})\coloneqq \lct^{\fk{q}}(\fk{a})^{-1},
\end{equation*}
and again when $\fk{q}=\mathcal{O}_X$ we write it as $\Arn(\fk{a})$.
\begin{lemma}[cf. {\cite[Lemma 1.7]{JM12}}]\label{lem-JM12 lem1.7}
 If $\pi\colon Y\to X$ is a log resolution of the nonzero ideal $\fk{a}$, and if $\fk{a}\cdot \mathcal{O}_Y=\mathcal{O}_Y(-\sum_i\alpha_iE_i)$ and $K_{Y/X}=\sum_i\kappa_iE_i$, then for every nonzero ideal $\fk{q}$
 \begin{equation}\label{equ-compute lctq(a)}
     \lct^{\fk{q}}(\fk{a})=\min_{\alpha_i>0}\frac{A(\ord_{E_i})+\ord_{E_i}(\fk{q})}{\ord_{E_i}(\fk{a})}=\min_{\alpha_i>0}\frac{\kappa_i+1+\ord_{E_i}(\fk{q})}{\alpha_i}.
 \end{equation}
 Moreover, for any ideals $\fk{a}$, $\fk{b}$, $\fk{q}$, $\fk{q}_1$ and $\fk{q}_2$, we have
 \begin{enumerate}
     \item if $\fk{a}\subseteq \fk{b}$, then $\lct^{\fk{q}}(\fk{a})\leq \lct^{\fk{q}}(\fk{b})$;
     \item $\lct^{\fk{q}}(\fk{a}^m)=\lct^{\fk{q}}(\fk{a})/m$ for every $m\geq 1$;
     \item $\lct^{\fk{q}_1+\fk{q}_2}(\fk{a})=\min_{i=1,2}\lct^{\fk{q}_i}(\fk{a})$.
 \end{enumerate}
\end{lemma}
If the minimum in (\ref{equ-compute lctq(a)}) is achieved for $E_i$, we say that \emph{$\ord_{E_i}$ computes $\lct^{\fk{q}}(\fk{a})$}. Moreover, $\lct^{\fk{q}}(\fk{a})$ is a birational invariant, that is, if $\varphi\colon X'\to X$ is a proper birational morphism with $X$ and $X'$ regular, then we have 
\begin{equation*}
    \lct^{\fk{q}}(\fk{a})=\lct^{\fk{q'}}(\fk{a'}),
\end{equation*}
where $\fk{a}'\coloneqq \fk{a}\cdot \mathcal{O}_{X'}$ and $\fk{q'}\coloneqq \fk{q}\cdot \mathcal{O}_{X'}(-K_{X'/X})$.

For a nonzero graded sequence of ideals $\fk{a}_{\bullet}$ and for a subadditive system of ideals $\fk{b}_{\bullet}$, we can define the jumping numbers of $\fk{a}_{\bullet}$ and $\fk{b}_{\bullet}$ with respect to $\fk{q}$, similar to the definitions of $v(\fk{a}_{\bullet})$ and $v(\fk{b}_{\bullet})$ for any $v\in \Val_X$,
\begin{equation*}
\begin{split}
\lct^{\fk{q}}(\fk{a}_{\bullet})\coloneqq  \sup_m m\cdot \lct^{\fk{q}}(\fk{a}_m)=\lim_{m\to \infty}m\cdot\lct^{\fk{q}}(\fk{a}_m)\in (0,+\infty];\\
\lct^{\fk{q}}(\fk{b}_{\bullet})\coloneqq  \inf_t t\cdot \lct^{\fk{q}}(\fk{b}_t)=\lim_{t\to \infty}t\cdot\lct^{\fk{q}}(\fk{b}_t)\in [0,+\infty).
\end{split}   
\end{equation*}
Furthermore, one can prove the following lemma likely to \Cref{lem-JM12 lem1.7}.

\begin{lemma}[cf. {\cite{JM12}}]\label{lem-JM12 lct.equals.min}
Let $\fk{a}_{\bullet}$ be a graded sequence of ideals and $\fk{b}_{\bullet}$ be a subadditive system of ideals $\fk{b}_{\bullet}$. If $\fk{q}$ is a nonzero ideal, then
\begin{equation*}
    \begin{split}
        \lct^{\fk{q}}(\fk{a}_{\bullet})=\inf_E\frac{A(\ord_E)+\ord_E(\fk{q})}{\ord_E(\fk{a}_{\bullet})};\\
         \lct^{\fk{q}}(\fk{b}_{\bullet})=\inf_E\frac{A(\ord_E)+\ord_E(\fk{q})}{\ord_E(\fk{b}_{\bullet})},
    \end{split}
\end{equation*}
where $E$ runs through all divisors over $X$ satisfying $\ord_E(\fk{a}_{\bullet})\neq 0$ and $\ord_E(\fk{b}_{\bullet})\neq 0$, respectively. Moreover, if $\fk{b}_{\bullet}$ is the asymptotic multiplier ideals of $\fk{a}_{\bullet}$, then 
\begin{equation*}
    \lct^{\fk{q}}(\fk{a}_{\bullet})=\lct^{\fk{q}}(\fk{b}_{\bullet})=\min\{\lambda\geq 0\colon\fk{q}\not\subseteq\fk{b}_{\lambda}\}.
\end{equation*}
\end{lemma}

Here we remind the readers to note that there may not be a divisor $E$ over $X$ that allows the above equations to take the infimum (see \cite[Example 8.5]{JM12}).

We also have the following equations to compute $ \lct^{\fk{q}}(\fk{a}_{\bullet})$ and $ \lct^{\fk{q}}(\fk{b}_{\bullet})$.
\begin{lemma}[{\cite[Section 6.2]{JM12}}]\label{lem-JM12 section6.2}
  Let $\fk{a}_{\bullet}$ be a graded sequence of ideals and $\fk{b}_{\bullet}$ be a subadditive system of ideals $\fk{b}_{\bullet}$, respectively. If $\fk{q}$ is a nonzero ideal, then
\begin{equation*}
    \begin{split}
        \lct^{\fk{q}}(\fk{a}_{\bullet})=\inf_{v\in \Val_X^*}\frac{A(v)+v(\fk{q})}{v(\fk{a}_{\bullet})};\\
         \lct^{\fk{q}}(\fk{b}_{\bullet})=\inf_{\substack{v\in \Val_X^*\\ A(v)<\infty}}\frac{A(v)+v(\fk{q})}{v(\fk{b}_{\bullet})}.
    \end{split}
\end{equation*} 
\end{lemma}

We say that a valuation $v\in\Val_X^*$ computes $\lct^{\fk{q}}(\fk{a}_{\bullet})$ for a nonzero ideal $\fk{q}$ and a graded sequence of ideals $\fk{a}_{\bullet}$ if $\lct^{\fk{q}}(\fk{a}_{\bullet})=\frac{A(v)+v(\fk{q})}{v(\fk{a}_{\bullet})}$. If $\lct^{\fk{q}}(\fk{a}_{\bullet})=+\infty$, then every valuation $v\in \Val_X^*$ with $A(v)<+\infty$ computes $\lct^{\fk{q}}(\fk{a}_{\bullet})$. For general cases, using the compactness arguments, Jonsson--Musta\c{t}\u{a} obtained the existence of valuation computing the jumping numbers:

\begin{theorem}[{\cite[Theorem 7.3]{JM12}}]\label{thm-exist.val.comp}
Let $\{a_{\bullet}\}$ be a graded sequence of ideals on $X$, and $\fk{q}$ be a
nonzero ideal. Then there exists a valuation $v\in\Val_X^*$ that computes $\lct^{\fk{q}}(\fk{a}_{\bullet})$.
\end{theorem}

The following lemma is from \cite{JM12}, but there are some typos in the statements there, so we demonstrate a proof below.

\begin{lemma}[{\cite[Proposition 7.10]{JM12}}]\label{lem-JM.compute}
    Let $\fk{a}_{\bullet}$ be a graded sequence of ideals on $X$. Assume that $v\in\Val_X^*$ computes $\lct^{\fk{q}}(\fk{a}_{\bullet})<+\infty$. Then 
    \begin{enumerate}
    \item $v$ also computes $\lct^{\fk{q}}(\fk{a}_{\bullet}^v)$;
    \item  $\lct^{\fk{q}}(\fk{a}_{\bullet}^v)=v(\fk{a}_{\bullet})\cdot\lct^{\fk{q}}(\fk{a}_{\bullet})$;
    \item any $v'\in\Val_X^*$ that computes $\lct^{\fk{q}}(\fk{a}_{\bullet}^v)$ also computes $\lct^{\fk{q}}(\fk{a}_{\bullet})$.
    \end{enumerate}    
\end{lemma}

\begin{proof}
    For every $w\in\Val_X^*$, since $v$ computes $\lct^{\fk{q}}(\fk{a}_{\bullet})$, we have
    \begin{equation}\label{eq-v.comp.lct}
        \frac{A(v)+v(\fk{q})}{v(\fk{a}_{\bullet})}=\lct^{\fk{q}}(\fk{a}_{\bullet})\le\frac{A(w)+w(\fk{q})}{w(\fk{a}_{\bullet})}.
    \end{equation}

    To prove (1), we need to prove that for every $w\in\Val_X^*$, it holds that
    \begin{equation}\label{eq-v.comp.lct.2}
        \frac{A(v)+v(\fk{q})}{v(\fk{a}_{\bullet}^v)}\le\frac{A(w)+w(\fk{q})}{w(\fk{a}_{\bullet}^v)}.
    \end{equation}
    We can assume $A(v)+v(\fk{q})>0$. Since $v(\fk{a}_{\bullet}^v)=1$ and it follows from (\ref{eq-v.comp.lct}) that
    \[\frac{A(w)+w(\fk{q})}{A(v)+v(\fk{q})}\ge\frac{w(\fk{a}_{\bullet})}{v(\fk{a}_{\bullet})}\ge w(\fk{a}_{\bullet}^v).\]
    Thus, the inequality (\ref{eq-v.comp.lct.2}) is verified for every $w\in\Val_X^*$, which implies (1).
    
    By (1), we have
    \[\lct^{\fk{q}}(\fk{a}_{\bullet})=\frac{A(v)+v(\fk{q})}{v(\fk{a}_{\bullet})}=\frac{A(v)+v(\fk{q})}{v(\fk{a}_{\bullet}^v)}\cdot \frac{1}{v(\fk{a}_{\bullet})}=\frac{\lct^{\fk{q}}(\fk{a}_{\bullet}^v)}{v(\fk{a}_{\bullet})},\]
    which yields (2).

    Assume that $v'\in\Val_X^*$ computes $\lct^{\fk{q}}(\fk{a}_{\bullet}^v)$. Then we have
    \[\lct^{\fk{q}}(\fk{a}_{\bullet}^v)=\frac{A(v')+v'(\fk{q})}{v'(\fk{a}_{\bullet}^v)}.\]
    The above equality and (2) induce
    \[\lct^{\fk{q}}(\fk{a}_{\bullet})=\frac{\lct^{\fk{q}}(\fk{a}_{\bullet}^v)}{v(\fk{a}_{\bullet})}=\frac{A(v')+v'(\fk{q})}{v'(\fk{a}_{\bullet})}\cdot \frac{v'(\fk{a}_{\bullet})}{v(\fk{a}_{\bullet})\cdot v'(\fk{a}_{\bullet}^v)}\ge \frac{A(v')+v'(\fk{q})}{v'(\fk{a}_{\bullet})}.\]
    This indicates that $v'$ computes $\lct^{\fk{q}}(\fk{a}_{\bullet})$.
\end{proof}

\subsection{Jonsson--Musta\c{t}\u{a}'s conjectures}
In this section, we recall the so-called Jonsson--Musta\c{t}\u{a}'s conjectures posed in \cite{JM12} (see also \cite{JM14}), with some reformulations for simplicity.

\begin{conjecture}[{\cite[Conjecture B]{JM12}}, {\cite[Conjecture C]{JM14}}]\label{conj-JM12.lct.version}
  Let $\fk{a}_{\bullet}$ be a graded sequence of ideals on $X$ with $\lct(\fk{a}_{\bullet})<\infty$.
  \begin{itemize}[leftmargin=1em]
      \item \textbf{Weak version}: there exists a quasi-monomial valuation $v\in \Val_X^*$ that computes $\lct(\fk{a}_{\bullet})$;
      \item \textbf{Strong version}: any valuation $v\in \Val_X^*$ computing $\lct(\fk{a}_{\bullet})$ must be quasi-monomial.
  \end{itemize}
\end{conjecture}
\begin{conjecture}[{\cite[Conjecture C, Conjecture 7.4]{JM12}}, {\cite[Conjecture C']{JM14}}]\label{conj-JM12.jn.version}
  Let $\fk{a}_{\bullet}$ be a graded sequence of ideals on $X$ and $\fk{q}$ a nonzero ideal on $X$ with $\lct^{\fk{q}}(\fk{a}_{\bullet})<\infty$.
  \begin{itemize}[leftmargin=1em]
      \item \textbf{Weak version}: there exists a quasi-monomial valuation $v\in \Val_X^*$ that computes $\lct^{\fk{q}}(\fk{a}_{\bullet})$;
      \item \textbf{Strong version}: any valuation $v\in \Val_X^*$ computing $\lct^{\fk{q}}(\fk{a}_{\bullet})$ must be quasi-monomial.
  \end{itemize}
\end{conjecture}

The following theorem shows that, to prove the above Jonsson--Musta\c{t}\u{a}'s conjectures, it suffices to consider the graded sequences associated to valuations in $\Val_X^*$.

\begin{theorem}[{\cite[Theorem 7.7, Theorem 7.8]{JM12}}]\label{thm-TFAEinJM12}
If $v\in \Val_X^*$ is a nontrivial valuation with $A(v)<+\infty$ and $\fk{q}$ is a nonzero ideal on $X$, the following are equivalent:
 \begin{enumerate}
     \item there is a graded sequence of ideals $\fk{a}_{\bullet}$ on $X$ such that $v$ computes the jumping number $\lct^{\fk{q}}(\fk{a}_{\bullet})<\infty$;
     \item there is a subadditive system of ideals $\fk{b}_{\bullet}$ of controlled growth such that $v$ computes $\lct^{\fk{q}}(\fk{b}_{\bullet})<\infty$;
     \item for every $w\in \Val_X$ with $w\geq v$, we have $A(w)+w(\fk{q})\geq A(v)+v(\fk{q})$;
     \item $v$ computes $\lct^{\fk{q}}(\fk{a}_{\bullet}^v)$.
 \end{enumerate}
\end{theorem}

The $2$-dimensional cases of the conjectures are known to be true, and the conjectures are trivial for dimension $1$.

\begin{theorem}[{\cite[Section 9]{JM12}}]\label{thm-strong version of An}
The strong versions of \Cref{conj-JM12.lct.version} and \Cref{conj-JM12.jn.version} are valid for all $X$ with $\dim(X)\leq 2$.
\end{theorem}

In \cite{Xu20}, Xu made a breakthrough on the weak version of \Cref{conj-JM12.lct.version}.
\begin{theorem}[{\cite[Theorem 1.1]{Xu20}}]\label{thm-Xu}
The weak version of \Cref{conj-JM12.lct.version} holds for all $X$. 
\end{theorem}

\section{Definition and existence of Zhou valuations}\label{section3}

This section is devoted to the introduction of Zhou valuations and their associated cone, along with several elementary propositions related to them.

Let $\fk{q}$ be a nonzero ideal on $X$, and denote
\begin{equation}\label{equ-Val(X;q)}
    \Val(X;\fk{q})\coloneqq \{v\in\Val_X\colon \fk{q}\not\subseteq\cal{J}(\fk{a}_{\bullet}^v)\},
\end{equation}
which is a subset of $\Val_X$, where $\cal{J}(\fk{a}_{\bullet}^v)$ is the asymptotic multiplier ideal of the graded sequence $\fk{a}^v_{\bullet}$. In case that $\fk{q}$ is nonzero, clearly the set $\Val(X;\fk{q})$ is nonempty. By the definition of jumping numbers, one obtains the following description of \(\Val(X; \mathfrak{q})\).

\begin{lemma}\label{lem-Val(X;q)}
For any nonzero ideal $\fk{q}$ on $X$,
   \[\Val(X;\fk{q})=\{v\in\Val_X\colon \lct^{\fk{q}}(\fk{a}_{\bullet}^v)\le 1\}.\]
   Moreover, $v(\fk{b})\le \Arn^{\fk{q}}(\fk{b})$ for every $v\in\Val(X;\fk{q})$ and every nonzero ideal $\fk{b}$ on $X$.
\end{lemma}
\begin{proof}
By \Cref{lem-JM12 lct.equals.min}, we have
\begin{equation}\label{equ-lemma3.1}
    \lct^{\fk{q}}(\fk{a}^v_{\bullet})=\min\{\lambda\ge 0 \colon \fk{q}\not\subseteq\cal{J}(\lambda\cdot\fk{a}^v_{\bullet})\},
\end{equation}
    where $\cal{J}(\lambda\cdot\fk{a}^v_{\bullet})$ is the asymptotic multiplier ideal with coefficient $\lambda$ (see \cite{LarII04}). Thus, it is clear that $\Val(X;\fk{q})$ consists of these $v\in\Val_X$ with $\lct^{\fk{q}}(\fk{a}_{\bullet}^v)\le 1$.

    If $v\in\Val(X;\fk{q})$, then by \Cref{thm-exist.val.comp} we can pick $w\in\Val_X^*$ which computes $\lct^{\fk{q}}(\fk{a}^v_{\bullet})$, that is
    \[1\ge \lct^{\fk{q}}(\fk{a}^v_{\bullet})=\frac{A(w)+w(\fk{q})}{w(\fk{a}^v_{\bullet})}.\]
    After rescaling $w$, we can let $w(\fk{a}^v_{\bullet})=1$, which implies $w\ge v$ and $A(w)+w(\fk{q})\le 1$. The valuative formula for the jumping number gives
    \[v(\fk{b})\le w(\fk{b})\le \big(A(w)+w(\fk{q})\big)\Arn^{\fk{q}}(\fk{b})\le \Arn^{\fk{q}}(\fk{b}).\qedhere\]
\end{proof}

Now we give the definition of Zhou valuations on $X$. 

\begin{definition}\label{def-Zhou.val}
    A valuation $v\in \Val_X$ is called a \textbf{Zhou valuation related to $\fk{q}$} if and only if $v$ is a maximal element in the set $\Val(X;\fk{q})$, i.e., there is no $w\in\Val(X;\fk{q})$ such that $w\ge v$ and $w\neq v$.
\end{definition}

 According to Zorn's lemma, the Zhou valuation related to $\fk{q}$ always exists for nonzero $\fk{q}$ (but it can be not unique). More precisely, we have the following theorem.

\begin{theorem}\label{thm-existence.gr.sys}
    Fix a nonzero ideal $\fk{q}$ on $X$. For any graded sequence $\fk{a}_{\bullet}=\{\fk{a}_m\}$ satisfying $\lct^{\fk{q}}(\fk{a}_{\bullet})=1$, there exists a Zhou valuation $v$ related to $\fk{q}$ such that $\fk{a}_m\subseteq\fk{a}_m^{v}$ for each $m\in\bfZ_+$.
\end{theorem}

\begin{proof}
    Denote
    \[\cal{S}(\fk{a}_{\bullet})\coloneqq \big\{w\in\Val(X;\fk{q})\colon \fk{a}_m\subseteq\fk{a}_{m}^w, \ \forall\, m\big\}.\]
    
    First, we prove $\cal{S}(\fk{a}_{\bullet})\neq\emptyset$. By \Cref{thm-exist.val.comp}, we can pick $w\in\Val_X^*$ that computes $\lct^{\fk{q}}(\fk{a}_{\bullet})$, i.e.,
    \[1=\lct^{\fk{q}}(\fk{a}_{\bullet})=\frac{A(w)+w(\fk{q})}{w(\fk{a}_{\bullet})},\]
    and after rescaling $w$, we can let $w(\fk{a}_{\bullet})=1$. This yields $w(\fk{a}_m)\ge m$ by the definition of $w(\fk{a}_{\bullet})$ and then $\fk{a}_m\subseteq \fk{a}_m^w$ for each $m$. In addition, Lemma \ref{lem-JM.compute} (2) shows that $\lct^{\fk{q}}(\fk{a}_{\bullet}^w)=w(\fk{a}_{\bullet})\cdot\lct^{\fk{q}}(\fk{a}_{\bullet})=1$. Now we get $w\in\cal{S}(\fk{a}_{\bullet})$, which shows that the set $\cal{S}(\fk{a}_{\bullet})$ is not empty.
    
    Next, by Zorn's lemma and by noting that for any $w\in \Val(X;\fk{q})$ with $w\geq v$ for some $v\in \cal{S}(\fk{a}_{\bullet})$ we must have $w\in \cal{S}(\fk{a}_{\bullet})$, to prove the existence of the Zhou valuation, we only need to prove that for any increasing chain in $\cal{S}(\fk{a}_{\bullet})$ there is an upper bound belonging to $\cal{S}(\fk{a}_{\bullet})$. Let $\{v_j\}_{j\in I}$ be such a chain. Set
    \[\tilde{v}(\fk{q}')\coloneqq \sup_{j}v_j(\fk{q}'), \quad \forall \text{ ideal } \fk{q}'\subseteq\cal{O}_X.\]
    The closures of the centers form a decreasing chain by \cite[Lemma 4.4]{JM12}. Since $X$ is Noetherian, we may pass to a tail with common center $\xi$. For $0\ne f\in\cal{O}_{X,\xi}$, applying \Cref{lem-Val(X;q)} to an ideal with stalk $(f)$ at $\xi$ gives $\sup_j v_j(f)<\infty$. The valuation axioms pass to this increasing limit on $\cal{O}_{X,\xi}$, so $\tilde{v}\in\Val_X$, and clearly $\fk{a}_m\subseteq\fk{a}_m^{\tilde{v}}$ for every $m$.
    
    It is left to prove $\fk{q}\not\subseteq \cal{J}(\fk{a}^{\tilde{v}}_{\bullet})$, which is equivalent to 
    \[\fk{q}\not\subseteq \cal{J}\left(\frac{1}{m}\cdot \fk{c}^{\tilde{v}}_{m}\right), \quad \forall\, m\in\bfZ_+,\]
    where $\fk{c}^{v}_{m}\coloneqq\{v > m\}$. Indeed, $\fk{a}^{v}_{m+1}\subseteq \fk{c}^{v}_{m} \subseteq \fk{a}_{m}^{v}$ gives $\lct^{\fk{q}}(\fk{c}^{v}_{\bullet})=\lct^{\fk{q}}(\fk{a}^{v}_{\bullet})$ by the limit formula for jumping numbers. For every fixed $m\in\bfZ_{+}$, the ideals $\fk{c}^{v_j}_m$ form an increasing chain with union $\fk{c}^{\tilde v}_m$. Hence, by the Noetherian property, there exists some $j_m$ such that $\fk{c}^{\tilde{v}}_{m} = \fk{c}^{v_{j_{m}}}_{m}$. Meanwhile, it follows from $v_{j_m}\in\Val(X;\fk{q})$ that 
    \[\fk{q}\not\subseteq \cal{J}\left(\frac{1}{m}\cdot \fk{c}^{v_{j_{m}}}_{m}\right)=\cal{J}\left(\frac{1}{m}\cdot \fk{c}^{\tilde{v}}_{m}\right).\]
    Thus, $\tilde{v}\in\Val(X;\fk{q})$ is an upper bound of the increasing chain $\{v_j\}$. Then the proof is done by Zorn's lemma.
\end{proof}

Particularly, when the graded sequence in Theorem \ref{thm-existence.gr.sys} is associated to a valuation, we immediately get the following corollary.

\begin{corollary}\label{cor-existence.valuation}
    For any nonzero ideal $\fk{q}$ on $X$ and any $w\in\Val_X$ with $\lct^{\fk{q}}(\fk{a}_{\bullet}^w)=1$, there exists a Zhou valuation $v$ related to $\fk{q}$ such that $v\ge w$.
\end{corollary}

\begin{proof}
   Note that $\fk{a}_{m}^w\subseteq\fk{a}_{m}^v$ for each $m$ induces $v(\fk{a}_{\bullet}^w)\ge 1$ and then $v\ge w$. Thus, the needed Zhou valuation comes from Theorem \ref{thm-existence.gr.sys}.
\end{proof}

It is easy to see
\begin{proposition}\label{prop-lct(Zhou val)=1}
    Every Zhou valuation $v$ related to $\fk{q}$ must satisfy $\lct^{\fk{q}}(\fk{a}_{\bullet}^{v})=1$.
\end{proposition}

\begin{proof}
    Suppose $\lct^{\fk{q}}(\fk{a}_{\bullet}^{v})=c<1$ for a Zhou valuation $v$ related to $\fk{q}$. Then we have $v/c\in\Val(X;\fk{q})$ and $v/c\ge v$. This contradicts to the definition of Zhou valuation.
\end{proof}

By multiplying by a positive scalar product, we extend the definition of Zhou valuation to a cone space.

\begin{definition}
    For every nonzero ideal $\fk{q}$ on $X$, we denote by $\ZVal_X(\fk{q})$ the set of all Zhou valuations on $X$ related to $\fk{q}$, and denote \textbf{the cone of Zhou valuations} on $X$ by
\[\ZVal_X\coloneqq \{v=cv'\in \Val_X\colon c\in \bfR_{+}, \ v'\in\ZVal_X(\fk{q})\}.\]
\end{definition}

Then \Cref{cor-existence.valuation} directly implies

\begin{proposition}\label{prop-lctqav'=lctqav}
    Let $\fk{q}$ be a nonzero ideal on $X$, and $v\in\Val_X^*$ satisfying $\lct^{\fk{q}}(\fk{a}_{\bullet}^v)<+\infty$. Then there exists $v'\in\ZVal_X$ such that $v'\ge v$ and $\lct^{\fk{q}}(\fk{a}_{\bullet}^{v'})=\lct^{\fk{q}}(\fk{a}_{\bullet}^v)$.
\end{proposition}

\begin{proof}
    Set $c=\lct^{\fk{q}}(\fk{a}_{\bullet}^v)$. Then $\lct^{\fk{q}}(\fk{a}_{\bullet}^{v/c})=1$. By Corollary \ref{cor-existence.valuation}, there exists a Zhou valuation $w$ related to $\fk{q}$ such that $w\ge v/c$. Let $v'=cw$. Then $v'\in\ZVal_X$ satisfying $v'\ge v$ and $\lct^{\fk{q}}(\fk{a}_{\bullet}^{v'})=\lct^{\fk{q}}(\fk{a}_{\bullet}^v)$.
\end{proof}

\section{Zhou valuations computing jumping numbers}\label{section4}
The goal of this section is to establish the proof of \Cref{thm-main thmA}. First, we compute the log discrepancy of Zhou valuation.

\begin{theorem}\label{thm-A(v)=1-v(q)}
    Let $v$ be a Zhou valuation related to $\fk{q}$. Then the log discrepancy \[A(v)=1-v(\fk{q}).\]
    Moreover, $v$ computes $\lct^{\fk{q}}(\fk{a}_{\bullet}^v)$, and every valuation in $\Val^*_X$ computing $\lct^{\fk{q}}(\fk{a}_{\bullet}^v)$ is a positive scalar product of $v$.
\end{theorem}

\begin{proof}
    Pick any $w\in\Val^*_X$ that computes $\lct^{\fk{q}}(\fk{a}_{\bullet}^v)$. Then
    \begin{equation}\label{eq-1=lctq.A(w)+w(q)}
        1=\lct^{\fk{q}}(\fk{a}_{\bullet}^v)=\frac{A(w)+w(\fk{q})}{w(\fk{a}_{\bullet}^v)}.
    \end{equation}
    After rescaling, we may assume $w(\fk{a}_{\bullet}^v)=1$. Then Lemma \ref{lem-JM.compute} implies $\lct^{\fk{q}}(\fk{a}_{\bullet}^w)=1$, and thus $w\in\Val(X;\fk{q})$. It follows from $w(\fk{a}_{\bullet}^v)=1$ that $w\ge v$. Since $v$ is maximal in $\Val(X;\fk{q})$, we get $w=v$. Consequently, $v$ computes $\lct^{\fk{q}}(\fk{a}_{\bullet}^v)$, and $A(v)+v(\fk{q})=1$ by (\ref{eq-1=lctq.A(w)+w(q)}). In addition, from the above we can know that every valuation in $\Val^*_X$ computing $\lct^{\fk{q}}(\fk{a}_{\bullet}^v)$ is a positive scalar product of $v$. The proof is done.
\end{proof}

By \Cref{thm-A(v)=1-v(q)}, for every $v\in \ZVal_X$ we have $A(v)<+\infty$, and we repeatedly use this fact in the paper. 

\begin{corollary}\label{cor-ZV.computes}
    For every nonzero ideal $\fk{q}$ and graded sequence of ideals $\fk{a}_{\bullet}$ on $X$ with $\lct^{\fk{q}}(\fk{a}_{\bullet})<+\infty$, there exists a Zhou valuation $v$ related to $\fk{q}$ computing $\lct^{\fk{q}}(\fk{a}_{\bullet})$.
\end{corollary}

\begin{proof}
    Pick any $w\in\Val^*_X$ that computes $\lct^{\fk{q}}(\fk{a}_{\bullet})$. We may assume $w(\fk{a}_{\bullet})=1$. By Lemma \ref{lem-JM.compute}, $\lct^{\fk{q}}(\fk{a}_{\bullet}^w)=\lct^{\fk{q}}(\fk{a}_{\bullet})$, and $w$ computes $\lct^{\fk{q}}(\fk{a}_{\bullet}^w)$. Then $w/\lct^{\fk{q}}(\fk{a}_{\bullet})\in\Val(X;\fk{q})$, and we can take a Zhou valuation related to $\fk{q}$ such that $v\ge w/\lct^{\fk{q}}(\fk{a}_{\bullet})$. The above and Theorem \ref{thm-A(v)=1-v(q)} indicate
    \[\lct^{\fk{q}}(\fk{a}_{\bullet}^w)=\lct^{\fk{q}}(\fk{a}_{\bullet})=\frac{A(v)+v(\fk{q})}{1/\lct^{\fk{q}}(\fk{a}_{\bullet})}\ge\frac{A(v)+v(\fk{q})}{v(\fk{a}_{\bullet}^w)}.\]
    Thus, $v$ computes $\lct^{\fk{q}}(\fk{a}_{\bullet}^w)$, which is followed by that $v$ computes $\lct^{\fk{q}}(\fk{a}_{\bullet})$; see Lemma \ref{lem-JM.compute} again.
\end{proof}

Notably, if $v\in\Val_X$ satisfies $\lct^{\fk{q}}(\fk{a}_{\bullet}^v)=1$ and only positive scalars of $v$ compute $\lct^{\fk{q}}(\fk{a}_{\bullet}^v)$, then $v$ is a Zhou valuation related to $\fk{q}$. Indeed, if $v$ were not such a valuation, then by \Cref{cor-ZV.computes}, there would exist a Zhou valuation $\tilde{v}$ related to $\mathfrak{q}$ that computes $\lct^{\mathfrak{q}}(\mathfrak{a}_{\bullet}^v)$. We would then have $\tilde{v} = v$, contradicting our assumption.

\begin{corollary}
    If $\fk{a}_{\bullet}$ is a graded sequence of ideals, then for every nonzero ideal $\fk{q}$ with $\lct^{\fk{q}}(\fk{a}_{\bullet})<+\infty$ we have
    \[\lct^{\fk{q}}(\fk{a}_{\bullet})=\min_{v\in\ZVal_X}\frac{A(v)+v(\fk{q})}{v(\fk{a}_{\bullet})}.\]
\end{corollary}

\begin{proof}
    This is a direct consequence of Corollary \ref{cor-ZV.computes} and \Cref{lem-JM12 section6.2}.
\end{proof}

We present a simple and special example to motivate the relationship between Zhou valuations and quasi-monomial valuations.
\begin{example}\label{ex-Zhou.val}
    Let $\fk{a}$ and $\fk{q}$ be nonzero ideals on $X$, where $\fk{a}\neq \cal{O}_X$. Assume $\lct^{\fk{q}}(\fk{a})=1$. Set $\fk{a}_{\bullet}=\{\fk{a}_m\}$ with $\fk{a}_m=\fk{a}^m$. Then $v(\fk{a}_{\bullet})=v(\fk{a})$ for each $v\in\Val_X$.
    
    Pick a log-smooth pair $(Y,E)$ over $X$ which gives a log resolution of $\fk{a}\cdot\fk{q}$. It is known (see \cite[Lemma 6.7]{JM12}) that 
    \[\lct^{\fk{q}}(\fk{a})=\min_{v\in\Val_X^*}\frac{A(v)+v(\fk{q})}{v(\fk{a})},\]
    where the equality is achieved for $v$ if and only if $v \in \QM(Y,E)$ and $\ord_{E_i}$ computes $\lct^{\fk{q}}(\fk{a})$ for every irreducible component $E_i$ of $E$ for which $v(E_i)>0$.

    Note that any Zhou valuation $w$ related to $\fk{q}$ with $\fk{a}_m\subseteq \fk{a}_m^w$ for each $m$ (equivalent to $w(\fk{a})\ge 1$) must compute $\lct^{\fk{q}}(\fk{a})$, since (by Theorem \ref{thm-A(v)=1-v(q)})
    \[\lct^{\fk{q}}(\fk{a})=1=A(w)+w(\fk{q})\ge\frac{A(w)+w(\fk{q})}{w(\fk{a})}.\]
    It follows that $w\in\QM(Y,E)$ and satisfies the above conditions. More precisely, if we let $w=\val_{\alpha}\in\QM(Y,E)$, with $\alpha=(\alpha_1,\ldots,\alpha_r)\in\bfR_{\ge 0}$ and $\alpha_i=\val_{\alpha}(E_i)$, then for each $i$ with $\alpha_i>0$, $\ord_{E_i}$ must compute $\lct^{\fk{q}}(\fk{a})$. Also, we have $\val_{\alpha}(\fk{a})=1$ since $\lct^{\fk{q}}(\fk{a}_{\bullet}^{w})=1$ (Lemma \ref{lem-JM.compute}). 
    
    Conversely, suppose that the log resolution is the identity on $X$. Let $v\in\QM(X,E)$ compute $\lct^{\fk{q}}(\fk{a})$, with $v(\fk{a})=1$. By \Cref{lem-JM.compute,cor-existence.valuation}, there is a Zhou valuation $w\ge v$ related to $\fk{q}$. Since the $E_i$ are divisors on $X$, we have $w(E_i)\ge v(E_i)$ for all $i$. Hence
    \[1=A(w)+w(\fk{q})\ge\sum_i w(E_i)+w(\fk{q})\ge\sum_i v(E_i)+v(\fk{q})=1.\]
    Thus $w(E_i)=v(E_i)$ for all $i$, and $A(w)=A(r_{X,E}(w))$, so $w\in\QM(X,E)$ by \Cref{prop-log dis}(1). The centers of $v$ and $w$ are generic points of components of the same intersection of the $E_i$. Since $c_X(w)\in\overline{\{c_X(v)\}}$ by \cite[Lemma 4.4]{JM12}, their centers coincide. The common weights then give $w=v$.
    
\end{example}

Noting that in case $\fk{q}=\cal{O}_X$, we have the following result, which is based Xu's solution (see \Cref{thm-Xu}) to the weak version of Jonsson--Musta\c{t}\u{a}'s conjecture.

\begin{corollary}\label{cor-ZV.OX.qm}
    Any Zhou valuation related to $\cal{O}_X$ is quasi-monomial.
\end{corollary}

\begin{proof}
  Let $v$ be a Zhou valuation related to $\cal{O}_X$. By \cite[Theorem 1.1]{Xu20}, there exists a quasi-monomial valuation $w\in\Val_X^*$ computing $\lct(\fk{a}_{\bullet}^v)$. Thus, by \Cref{thm-A(v)=1-v(q)}, there exists $\lambda>0$ such that $v=\lambda\cdot w$, which implies that $v$ is also quasi-monomial.
\end{proof}

For the general cases, we can see that the weak version of Jonsson--Musta\c{t}\u{a}'s conjecture is equivalent to that every Zhou valuation is quasi-monomial, where we denote by $\QM_X$ the set of quasi-monomial valuations on $X$ in the following.

\begin{theorem}\label{thm-equiv.weak.ZVal=QM}
    The following statements are equivalent:
\begin{enumerate}
    \item For every nonzero ideal $\fk{q}$ and graded sequence of ideals $\fk{a}_{\bullet}$ on $X$ with $\lct^{\fk{q}}(\fk{a}_{\bullet})<+\infty$, there exists a quasi-monomial valuation $v$ which computes $\lct^{\fk{q}}(\fk{a}_{\bullet})$;
    \item $\ZVal_X\subseteq\QM_X$.
\end{enumerate}
\end{theorem}

\begin{proof}
    $(1) \Rightarrow (2)$: Pick any Zhou valuation $v$ related to some nonzero $\fk{q}$. Then $\lct^{\fk{q}}(\fk{a}_{\bullet}^v)=1<+\infty$, and there exists a quasi-monomial valuation $w$ computing $\lct^{\fk{q}}(\fk{a}_{\bullet}^v)$ by the assumption. Theorem \ref{thm-A(v)=1-v(q)} indicates $w=cv$ for some $c\in\bfR_+$, which shows that $v\in\QM_X$ is a quasi-monomial valuation.

    $(2) \Rightarrow (1)$: For every nonzero ideal $\fk{q}$ and graded sequence of ideals $\fk{a}_{\bullet}$ on $X$ with $\lct^{\fk{q}}(\fk{a}_{\bullet})<+\infty$, by Corollary \ref{cor-ZV.computes}, there exists a Zhou valuation $v$ related to $\fk{q}$ that computes $\lct^{\fk{q}}(\fk{a}_{\bullet})$. If $\ZVal_X\subseteq\QM_X$, then $v$ is also quasi-monomial.
\end{proof}

Especially, by the solution of the 2-dimensional case of Jonsson--Musta\c{t}\u{a}'s conjecture (see \Cref{thm-strong version of An}), Theorem \ref{thm-equiv.weak.ZVal=QM} induces the following result when $\dim X\le 2$.

\begin{corollary}\label{cor-dim.le.2}
    If $\dim X\le 2$, then $\ZVal_X\subseteq\QM_X$.
\end{corollary}

\section{A mixed version of jumping numbers}\label{section5}

In this section, we introduce a mixed version of jumping numbers, which is different from the usual so-called mixed jumping numbers (cf. \cite{LarII04}). We also extend the results on jumping numbers from \Cref{subsetion-Log discrepancy and jumping numbers}, such as \Cref{lem-JM12 lem1.7}, \Cref{lem-JM12 lct.equals.min}, and \Cref{lem-JM.compute}, to their mixed counterparts.

Let $\fk{q}, \fk{q}'$ be ideals and $\fk{a}$ be a nonzero ideal on $X$. Let $\lambda\in (-\varepsilon_0, +\infty)$, where $\varepsilon_0>0$ so small, and the reason to take the $\varepsilon_0$ is that the cases $\lambda<0$ are very important in the later discussions of the present paper. Denote
\begin{equation}\label{equ-lct(q,labdamq',a)}
 \lct(\fk{q}, \lambda\cdot\fk{q}';\fk{a})=\lct^{{\fk{q}'}^{\lambda}}(\fk{q}, \fk{a})\coloneqq \inf_E\frac{A(\ord_E)+\ord_E(\fk{q})+\lambda\cdot\ord_E(\fk{q}')}{\ord_E(\fk{a})},   
\end{equation}
where the infimum is over all divisors $E$ over $X$ with $\ord_E(\fk{a})>0$, and we assume that $\lct(\fk{q}, \lambda\cdot\fk{q}';\fk{a})$ is positive for $\lambda\in (-\varepsilon_0, +\infty)$. Clearly, this notation coincides with $\lct^{\fk{q}\cdot(\fk{q}')^k}(\fk{a})$ when $\lambda=k\in\bfZ_{\ge 0}$. Moreover, similar to \Cref{lem-JM12 lem1.7}, one can get the following lemma.

\begin{lemma}\label{lem-log.resolution.lct}
    If $\pi\colon Y\to X$ is a log resolution of the nonzero ideal $\fk{a}\cdot \fk{q}\cdot \fk{q}'$, and if $\fk{a}\cdot\cal{O}_Y=\cal{O}_Y(-\sum_{i}\alpha_iE_i)$ and $K_{Y/X}=\sum_i\kappa_iE_i$, then for the ideals $\fk{q}$, $\fk{q}'$ and $\lambda\in (-\varepsilon_0,+\infty)$,
    \begin{flalign*}
     \begin{split}
         \lct(\fk{q},\lambda\cdot\fk{q}';\fk{a})&=\min_{\alpha_i>0}\frac{\kappa_i+1+\ord_{E_i}(\fk{q})+\lambda\cdot\ord_{E_i}(\fk{q}')}{\alpha_i}\\
         &=\min_{\ord_{E_i}(\fk{a})>0} \frac{A(\ord_{E_i})+\ord_{E_i}(\fk{q})+\lambda\cdot\ord_{E_i}(\fk{q}')}{\ord_{E_i}(\fk{a})}.
     \end{split}
    \end{flalign*} 
 \end{lemma}

 \begin{proof}
Let $E_i$ be an irreducible component of $V(\fk{a}\cdot \mathcal{O}_Y)$. Then we have the following two equations:
  \begin{enumerate}
    \item $\alpha_i=\ord_{E_i}(\fk{a})$;
    \item $A(\ord_{E_i})\coloneqq 1+\ord_{E_i}(K_{Y/X})=1+\kappa_i$.
\end{enumerate}
Therefore, we have the following inequality due to (\ref{equ-lct(q,labdamq',a)})
\[\lct(\fk{q},\lambda\cdot\fk{q}';\fk{a})\leq \min_{\alpha_i>0}\frac{\kappa_i+1+\ord_{E_i}(\fk{q})+\lambda\cdot\ord_{E_i}(\fk{q}')}{\alpha_i}.\]

In the following, we prove the inverse inequality. Note that for each divisor $E$, $\ord_E(\fk{a})> 0$ if and only if $c_X(E)\in V(\fk{a})$. For any divisorial valuation $\nu$, we can find a divisor $E$ such that they define the same divisorial valuation $\nu=\nu(E,Y')$ up to a scaling, and there is a proper birational morphism $f\colon Y'\to Y$ over $X$. So $\ord_E(\fk{a})> 0$ if and only if $c_Y(E)\in V(\fk{a}\cdot \mathcal{O}_Y)$. Suppose $c_Y(E)\in E_1 \cap \cdots \cap E_j$ and $c_Y(E)\notin E_1\cap \cdots \cap E_j\cap E_{j+1}$ with $\alpha_i\neq 0$ for all $1\leq i\leq j+1$ for convenience. Suppose further $\fk{q}\cdot \mathcal{O}_Y=\cal{O}_Y(-\sum_{i}t_iE_i)$ and $\fk{q}'\cdot \mathcal{O}_Y=\cal{O}_Y(-\sum_{i}s_iE_i)$. Then we have 
\begin{equation*}
\begin{split}
A(\ord_{E})&\coloneqq 1+\ord_{E}(K_{Y'/X})=1+\ord_{E}(K_{Y'/Y})+\ord_{E}(f^*(K_{Y/X}))  \\ 
&\geq 1+\ord_E(f^*E_1)+\cdots +\ord_E(f^*E_j)-1+\ord_{E}\big(f^*(\sum_i\kappa_iE_i)\big)\\
&= \sum_{1\leq i\leq j}(1+\kappa_i)\ord_E(f^*E_i)\\
&=\sum_{1\leq i\leq j} A(\ord_{E_i})\ord_E(f^*E_i).
\end{split}
\end{equation*}
In addition, we have the following equations:
\begin{equation*}
    \begin{split}
        \ord_E(\fk{a})&=\ord_E\Big(f^*(\sum_i \alpha_iE_i)\Big)=\sum_{1\leq i\leq j}\alpha_i\ord_E(f^*E_i);\\
    \ord_E(\fk{q})&=\ord_E\Big(f^*(\sum_i t_iE_i)\Big)=\sum_{1\leq i\leq j}t_i\ord_E(f^*E_i);\\
    \ord_E(\fk{q}')&=\ord_E\Big(f^*(\sum_i s_iE_i)\Big)=\sum_{1\leq i\leq j}s_i\ord_E(f^*E_i).
    \end{split}
\end{equation*}
Combining these relationships, we get
\begin{equation*}
\begin{split}
    \frac{A(\ord_E)+\ord_E(\fk{q})+\lambda\cdot\ord_E(\fk{q}')}{\ord_E(\fk{a})}&\geq \min_{1\leq i\leq j}\frac{A(\ord_{E_i})+\ord_{E_i}(\fk{q})+\lambda\cdot\ord_{E_i}(\fk{q}')}{\ord_{E_i}(\fk{a})}\\ 
    &\geq \min_{\alpha_i>0}\frac{\kappa_i+1+\ord_{E_i}(\fk{q})+\lambda\cdot\ord_{E_i}(\fk{q}')}{\alpha_i}.
\end{split}
\end{equation*}
The proof is done.
\end{proof}

Then we can obtain the following corollary which is similar to \cite[Lemma 6.7]{JM12}.

\begin{corollary}[cf. {\cite[Lemma 6.7]{JM12}}]\label{cor-log.resolution.lct}
    For each fixed $\lambda\in(-\varepsilon_0,+\infty)$, we have
    \begin{equation}\label{eq-lct.a.single.ideal.equal.min}
        \lct(\fk{q}, \lambda\cdot\fk{q}';\fk{a})=\min_{v\in\Val_X^*}\frac{A(v)+v(\fk{q})+\lambda\cdot v(\fk{q}')}{v(\fk{a})}.
    \end{equation}
    Suppose that $\fk{a}\neq \cal{O}_X$ and $(Y,E)$ is a log-smooth pair over $X$ giving a log resolution of $\fk{a}\cdot\fk{q}\cdot\fk{q}'$. Then equality in (\ref{eq-lct.a.single.ideal.equal.min}) is achieved if and only if $v\in\QM(Y,E)$ and $\ord_{E_i}$ computes $\lct(\fk{q},\lambda\cdot\fk{q}';\fk{a})$ for every irreducible component $E_i$ of $E$ for which $v(E_i)>0$.
\end{corollary}

\begin{proof}
By \Cref{lem-log.resolution.lct}, it suffices to show that 
\begin{equation}\label{equ-cor 5.2}
  \lct(\fk{q}, \lambda\cdot\fk{q}';\fk{a})\leq \min_{v\in\Val_X^*}\frac{A(v)+v(\fk{q})+\lambda\cdot v(\fk{q}')}{v(\fk{a})}.  
\end{equation}
Let $\pi\colon Y\to X$ be a log resolution of $\fk{a}\cdot\fk{q}\cdot\fk{q}'$ with some exceptional effective simple normal crossing divisor $E$. Consider the {retraction map} $r_{Y,E}\colon \Val_X\to \QM(Y,E)$. Then we have $v(*)=r_{Y,E}v(*)$ for $*\in \{\fk{a},\fk{q},\fk{q}'\}$ by \cite[Corollary 4.8]{JM12}, and we have $A(v)\geq A(r_{Y,E}(v))$ by \Cref{prop-log dis}. Therefore,
  \begin{equation*}
  \min_{v\in\QM(Y,E)}\frac{A(v)+v(\fk{q})+\lambda\cdot v(\fk{q}')}{v(\fk{a})}\leq \min_{v\in\Val_X^*}\frac{A(v)+v(\fk{q})+\lambda\cdot v(\fk{q}')}{v(\fk{a})}.  
\end{equation*}  
Let $\{E_i\}$ be the set of smooth irreducible components of $E$. Then for every $v\in \QM(Y,E)$, we have $A(v)=\sum_iv(E_i)\cdot A(\ord_{E_i})$ by (\ref{equ-A(v) for QM(Y,D)}), and for all $*\in \{\fk{a},\fk{q},\fk{q}'\}$ we have $v(*)=\sum_iv(E_i)\cdot \ord_{E_i}(*)$ by \Cref{prop-topology of ValX}(1). Thus, 
\begin{equation*}
  \min_{v\in\QM(Y,E)}\frac{A(v)+v(\fk{q})+\lambda\cdot v(\fk{q}')}{v(\fk{a})}\geq   \min_{i} \frac{A(\ord_{E_i})+\ord_{E_i}(\fk{q})+\lambda\cdot\ord_{E_i}(\fk{q}')}{\ord_{E_i}(\fk{a})}.
\end{equation*}
We get (\ref{equ-cor 5.2}) by applying \Cref{lem-log.resolution.lct}. 

Moreover, if $v$ achieves the equality in (\ref{equ-cor 5.2}), we must have $A(v) =A(r_{Y,E}(v))$ (that is $v\in \QM(Y,E)$ by \Cref{prop-log dis} again) and every $E_i$ satisfying $v(E_i)>0$ computes $ \lct(\fk{q}, \lambda\cdot\fk{q}';\fk{a})$. 
\end{proof}

Now we can define the \textbf{mixed version of jumping numbers} associated to a graded sequence of ideals in this article as follows. Let $\fk{q}, \fk{q}'$ be nonzero ideals on $X$. For any nonzero graded sequence $\fk{a}_{\bullet}=\{\fk{a}_m\}$ of ideals, it can be verified that
\[\lct(\fk{q}, \lambda\cdot\fk{q}';\fk{a}_{\bullet})\coloneqq\lim_{m\to\infty}m\cdot\lct(\fk{q}, \lambda\cdot\fk{q}';\fk{a}_m)=\sup_{m} m\cdot\lct(\fk{q}, \lambda\cdot\fk{q}';\fk{a}_m)\in (0,+\infty]\]
is well-defined for any $\lambda\in (-\varepsilon_0,+\infty)$ with $\varepsilon_0$ small enough. Similarly, for any subadditive system of ideals $\fk{b}_{\bullet}=\{\fk{b}_t\}$ satisfying $\lct^{\fk{q}}(\fk{b}_{\bullet})>0$, we can also define the mixed version of jumping numbers associated to this subadditive sequence by
\[\lct(\fk{q}, \lambda\cdot\fk{q}';\fk{b}_{\bullet})\coloneqq\lim_{t\to\infty}t\cdot\lct(\fk{q}, \lambda\cdot\fk{q}';\fk{b}_t)=\inf_{t} t\cdot\lct(\fk{q}, \lambda\cdot\fk{q}';\fk{b}_t)\in [0,+\infty)\]
for any $\lambda\in (-\varepsilon_0,+\infty)$. The following results are in the spirit of \Cref{lem-JM12 lct.equals.min}.

\begin{proposition}[{\cite[Proposition 2.8]{JM12}}]\label{prop-lct.equals.inf.Val.mixed.verb}
If $\fk{b}_{\bullet}$ is a subadditive system of ideals, then 
\begin{equation*}
\begin{split}
   \lct(\fk{q}, \lambda\cdot\fk{q}';\fk{b}_{\bullet})&=\inf_{E}\frac{A(\ord_E)+\ord_E(\fk{q})+\lambda\cdot \ord_E(\fk{q}')}{\ord_E(\fk{b}_{\bullet})}\\
&=\inf_{v\in\Val_X^*}\frac{A(v)+v(\fk{q})+\lambda\cdot v(\fk{q}')}{v(\fk{b}_{\bullet})},  
\end{split}   
\end{equation*}
where $E$ is over all divisors over $X$.
\end{proposition}
\begin{proof}
Let us first show 
\begin{equation}\label{equ1-lct(b.)}
\lct(\fk{q}, \lambda\cdot\fk{q}';\fk{b}_{\bullet})=\inf_{E}\frac{A(\ord_E)+\ord_E(\fk{q})+\lambda\cdot \ord_E(\fk{q}')}{\ord_E(\fk{b}_{\bullet})}.    
\end{equation}
On the one hand, we have 
\[m\cdot \lct(\fk{q},\lambda\cdot\fk{q}';\fk{b}_m)\leq \frac{A(\ord_E)+\ord_E(\fk{q})+\lambda\cdot \ord_E(\fk{q}')}{\ord_E(\fk{b}_m)/m}\]
for every divisor $E$ over $X$ by \Cref{lem-log.resolution.lct}. Taking direct limits of $m$, we get
\[\lct(\fk{q},\lambda\cdot\fk{q}';\fk{b}_{\bullet})\coloneqq\lim_{m\to\infty}m\cdot\lct(\fk{q}, \lambda\cdot\fk{q}';\fk{b}_m)\leq \frac{A(\ord_E)+\ord_E(\fk{q})+\lambda\cdot \ord_E(\fk{q}')}{\ord_E(\fk{b}_{\bullet})}\]
for every $E$, and further gives the inequality
\[\lct(\fk{q}, \lambda\cdot\fk{q}';\fk{b}_{\bullet})\leq \inf_{E}\frac{A(\ord_E)+\ord_E(\fk{q})+\lambda\cdot \ord_E(\fk{q}')}{\ord_E(\fk{b}_{\bullet})}.\]
On the other hand, for any given $m$, we can find a divisor $E$ over $X$ such that
\[\lct(\fk{q},\lambda\cdot\fk{q}';\fk{b}_m)=\frac{A(\ord_E)+\ord_E(\fk{q})+\lambda\cdot \ord_E(\fk{q}')}{\ord_E(\fk{b}_{m})}\]
by \Cref{lem-log.resolution.lct} again. Then 
\begin{equation*}
\begin{split}
m\cdot\lct(\fk{q},\lambda\cdot\fk{q}';\fk{b}_m)&=\frac{A(\ord_E)+\ord_E(\fk{q})+\lambda\cdot \ord_E(\fk{q}')}{\ord_E(\fk{b}_{m})/m}\\    
&\geq \frac{A(\ord_E)+\ord_E(\fk{q})+\lambda\cdot \ord_E(\fk{q}')}{\ord_E(\fk{b}_{\bullet})}\\
&\geq \inf_{E}\frac{A(\ord_E)+\ord_E(\fk{q})+\lambda\cdot \ord_E(\fk{q}')}{\ord_E(\fk{b}_{\bullet})}
\end{split}
\end{equation*}
by applying the equations $\ord_E(\fk{b}_{\bullet})=\sup_m\ord_E(\fk{b}_{m})/m$. Thus, we have proven (\ref{equ1-lct(b.)}). Similarly, one can prove 
\[\lct(\fk{q}, \lambda\cdot\fk{q}';\fk{b}_{\bullet})=\inf_{v\in\Val_X^*}\frac{A(v)+v(\fk{q})+\lambda\cdot v(\fk{q}')}{v(\fk{b}_{\bullet})}\]
by \Cref{cor-log.resolution.lct}.
\end{proof}
\begin{proposition}[{\cite[Proposition 2.14]{JM12}}]\label{prop-lct(a.)=lct(b.)}
Let $\fk{a}_{\bullet}$ be a graded sequence of ideals and $\fk{q}, \fk{q}'$ be ideals on $X$. 
Then $\lct(\fk{q},\lambda\cdot\fk{q}';\fk{a}_{\bullet})=\lct(\fk{q},\lambda\cdot\fk{q}';\fk{b}_{\bullet})$, where $\fk{b}_{\bullet}$ is the subadditive system given by the asymptotic multiplier ideals of $\fk{a}_{\bullet}$.    
\end{proposition}
\begin{proof}
First we have  $\lct(\fk{q}, \lambda\cdot\fk{q}';\fk{a}_m)\leq \lct(\fk{q}, \lambda\cdot\fk{q}';\fk{b}_m)$ since $\fk{a}_m\subseteq \fk{b}_m$. Multiplying by $m$ and letting $m$ go to infinity, it gives $\lct(\fk{q},\lambda\cdot\fk{q}';\fk{a}_{\bullet})\leq \lct(\fk{q},\lambda\cdot\fk{q}';\fk{b}_{\bullet})$. Moreover, if $\lct(\fk{q},\lambda\cdot\fk{q}';\fk{a}_{\bullet})=+\infty$, then we have $\lct(\fk{q},\lambda\cdot\fk{q}';\fk{a}_{\bullet})=\lct(\fk{q},\lambda\cdot\fk{q}';\fk{b}_{\bullet})=+\infty$ already.

Now, we may assume that $\lct(\fk{q},\lambda\cdot\fk{q}';\fk{a}_{\bullet})<+\infty$. For the reverse inequality, for a given $t>0$, let us choose $m$ divisible enough such that we have $\fk{b}_t=\mathcal{J}(\fk{a}_m^{t/m})$. Then by \cite[Lemma 2.15]{JM12}, 
\begin{equation}\label{inequ-JM12-Lem2.15}
\frac{\ord_E(\fk{a}_m)}{m}<\frac{\ord_E(\fk{b}_t)}{t}+\frac{A(\ord_E)}{t}   
\end{equation}
for any divisor $E$ over $X$. Pick a divisor $E$ computing $\lct(\fk{q},\lambda\cdot\fk{q}';\fk{a}_m)>0$. Then we have $A(\ord_E)+\ord_E(\fk{q})+\lambda\cdot\ord_E(\fk{q}')>0$ and $\ord_E(\fk{a}_m)>0$. Multiplying both sides of formula (\ref{inequ-JM12-Lem2.15}) by $A(\ord_E)+\ord_E(\fk{q})+\lambda\cdot\ord_E(\fk{q}')>0$. We get the following inequality
\begin{equation*}
\begin{split}
  &\frac{\ord_ E(\fk{a}_m)}{m} (A(\ord_E)+\ord_E(\fk{q})+\lambda\cdot\ord_E(\fk{q}'))\\
  <&\frac{\ord_E(\fk{b}_t)}{t} (A(\ord_E)+\ord_E(\fk{q})+\lambda\cdot\ord_E(\fk{q}'))\\
  &+\frac{A(\ord_E)}{t} (A(\ord_E)+\ord_E(\fk{q})+\lambda\cdot\ord_E(\fk{q}')).   
\end{split}   
\end{equation*}
Note that    
\begin{equation*}
\begin{split}
 &\frac{A(\ord_E)}{t}\cdot(A(\ord_E)+\ord_E(\fk{q})+\lambda\cdot\ord_E(\fk{q}'))\\
 =&m\cdot\lct (\fk{q},\lambda\cdot \fk{q}';\fk{a}_m)\cdot \frac{A(\ord_E)\ord_E(\fk{a}_m)}{mt}\\   
 \leq& \lct(\fk{q},\lambda\cdot\fk{q}';\fk{a}_{\bullet})\cdot \frac{A(\ord_E)\ord_E(\fk{a}_m)}{mt}.
\end{split}
\end{equation*}
Combining the above two inequalities, we get 
\begin{equation}\label{inequ-lcta=lctb}
\begin{split}
 &t\cdot{\ord_E(\fk{a}_m)}(A(\ord_E)+\ord_E(\fk{q})+\lambda\cdot\ord_E(\fk{q}'))\\
<&m\cdot\ord_E(\fk{b}_t)\big(A(\ord_E)+\ord_E(\fk{q})+\lambda\cdot\ord_E(\fk{q}')\big)\\
&+\lct(\fk{q},\lambda\cdot\fk{q}';\fk{a}_{\bullet})\cdot {A(\ord_E)\cdot\ord_E(\fk{a}_m)}.  
\end{split}    
\end{equation}

Put $L=\lct(\fk{q},\lambda\cdot\fk{q}';\fk{a}_{\bullet})$, $A_E=A(\ord_E)$, and $D_E=A_E+\ord_E(\fk{q})+\lambda\ord_E(\fk{q}')$. Take $\varepsilon_0$ small enough that $2\varepsilon_0<\lct(\fk{q}')$. By the definition of $\lct(\fk{q}')$, we have $\varepsilon_0\ord_E(\fk{q}')\le A_E/2$, so $D_E\ge A_E/2$ for every divisor $E$. For the divisor $E$ chosen above, $mD_E/\ord_E(\fk{a}_m)=m\lct(\fk{q},\lambda\cdot\fk{q}';\fk{a}_m)\le L$. Hence (\ref{inequ-JM12-Lem2.15}) gives
\begin{equation*}
    \ord_E(\fk{b}_t)>t\frac{\ord_E(\fk{a}_m)}{m}-A_E
    \geq \frac{tD_E}{L}-A_E\geq \left(\frac{t}{2L}-1\right)A_E.
\end{equation*}
In particular, $\ord_E(\fk{b}_t)>0$ for $t>2L$. Dividing (\ref{inequ-lcta=lctb}) by $\ord_E(\fk{a}_m)\cdot\ord_E(\fk{b}_t)$, we obtain
\begin{equation}\label{inequ-lcta=lctbII}
\begin{split}
     &t\cdot\lct(\fk{q},\lambda\cdot\fk{q}';\fk{b}_t)\\
     \leq & \frac{A(\ord_E)+\ord_E(\fk{q})+\lambda\cdot \ord_E(\fk{q}')}{\ord_E(\fk{b}_{t})/t}\\
     \leq & \frac{A(\ord_E)+\ord_E(\fk{q})+\lambda\cdot \ord_E(\fk{q}')}{\ord_E(\fk{a}_{m})/m}+\lct(\fk{q},\lambda\cdot\fk{q}';\fk{a}_{\bullet})\cdot \frac{A(\ord_E)}{\ord_E(\fk{b}_t)}\\
     = & m\cdot\lct(\fk{q},\lambda\cdot\fk{q}';\fk{a}_m)+\lct(\fk{q},\lambda\cdot\fk{q}';\fk{a}_{\bullet})\cdot \frac{A(\ord_E)}{\ord_E(\fk{b}_t)}\\
     \leq & \lct(\fk{q},\lambda\cdot\fk{q}';\fk{a}_{\bullet})+\lct(\fk{q},\lambda\cdot\fk{q}';\fk{a}_{\bullet})\cdot \frac{A(\ord_E)}{\ord_E(\fk{b}_t)}.
\end{split}
\end{equation}    
For $t>2L$, the same lower bound gives
\begin{equation*}
    0<\frac{A(\ord_E)}{\ord_E(\fk{b}_t)}<\frac{2L}{t-2L}.
\end{equation*}
The right-hand side tends to zero independently of $E$. Letting $t$ go to infinity in (\ref{inequ-lcta=lctbII}), we get
\begin{equation*}
    \lct(\fk{q},\lambda\cdot\fk{q}';\fk{b}_{\bullet})\leq \lct(\fk{q},\lambda\cdot\fk{q}';\fk{a}_{\bullet})\qedhere.
\end{equation*}
\end{proof}

\begin{corollary}[cf. {\cite[Corollary 6.8]{JM12}}]\label{cor-lct.equals.inf.Val.mixed.ver}
Let $\fk{a}_{\bullet}$ be a graded sequence of ideals and $\fk{q}, \fk{q}'$ be ideals on $X$. Then
    \[\lct(\fk{q}, \lambda\cdot\fk{q}';\fk{a}_{\bullet})=\inf_{v\in\Val_X^*}\frac{A(v)+v(\fk{q})+\lambda\cdot v(\fk{q}')}{v(\fk{a}_{\bullet})},\]
    for each $\lambda\in (-\varepsilon_0,+\infty)$ such that $\lct(\fk{q}, \lambda\cdot\fk{q}';\fk{a}_{\bullet})>0$.    
\end{corollary}
   
\begin{proof}
By \Cref{prop-lct.equals.inf.Val.mixed.verb} and \Cref{prop-lct(a.)=lct(b.)}, we have 
\begin{equation*}
   \lct(\fk{q}, \lambda\cdot\fk{q}';\fk{a}_{\bullet})=\lct(\fk{q}, \lambda\cdot\fk{q}';\fk{b}_{\bullet})=\inf_{v\in\Val_X^*}\frac{A(v)+v(\fk{q})+\lambda\cdot v(\fk{q}')}{v(\fk{b}_{\bullet})}.
\end{equation*}
We finish our proof by noting that $v(\fk{b}_{\bullet})=v(\fk{a}_{\bullet})$ whenever $A(v)<\infty$, which is given by \cite[Proposition 6.2]{JM12}.
\end{proof}

\begin{lemma}\label{lem-JM.compute.mixed.ver}
    Assume that $v\in\Val_X^*$ computes $\lct(\fk{q}, \lambda\cdot\fk{q}';\fk{a}_{\bullet})<+\infty$. Then 
\begin{enumerate}
    \item $v$ also computes $\lct(\fk{q}, \lambda\cdot\fk{q}';\fk{a}_{\bullet}^v)$;
    \item $\lct(\fk{q}, \lambda\cdot\fk{q}';\fk{a}_{\bullet}^v)=v(\fk{a}_{\bullet})\cdot\lct(\fk{q}, \lambda\cdot\fk{q}';\fk{a}_{\bullet})$;
    \item any $v'\in\Val_X^*$ that computes $\lct(\fk{q}, \lambda\cdot\fk{q}';\fk{a}_{\bullet}^v)$ also computes $\lct(\fk{q}, \lambda\cdot\fk{q}';\fk{a}_{\bullet})$.
\end{enumerate}
\end{lemma}

\begin{proof}
    With the help of \Cref{cor-lct.equals.inf.Val.mixed.ver}, this lemma can be proved by the same way of proving Lemma \ref{lem-JM.compute}.
\end{proof}

\section{Tian functions of valuations}\label{section6}
This section is devoted to the proof of \Cref{thm-main thmB} (1), following the introduction of the Tian function. Let $\fk{q}, \fk{q}'$ be ideals on $X$, and $\fk{a}_{\bullet}=\{\fk{a}_m\}$ a nonzero graded sequence of ideals on $X$. We have the following lemma.

\begin{lemma}\label{lem-lct.concave}
    If $\lct(\fk{q},\fk{q}';\fk{a}_{\bullet})<+\infty$, then the function
    \[(-\varepsilon_0,+\infty)\ni t \longmapsto \lct(\fk{q}, t\cdot\fk{q}'; \fk{a}_{\bullet})\]
    is concave and increasing, where $\varepsilon_0$ is a sufficiently small positive real number.
\end{lemma}

\begin{proof}
    In fact, for each $t$ and $m\in\bfZ_{+}$,
    \begin{flalign*}\label{eq-lctqkamv}
        \begin{split}
            \lct(\fk{q}, t\cdot\fk{q}';\fk{a}_{m})=\inf_E\frac{A(\ord_E)+\ord_E(\fk{q})+\ord_E(\fk{q}')\cdot t}{\ord_E(\fk{a}_{m})},
        \end{split}
    \end{flalign*}
    where the infimum is over all divisors $E$ over $X$ with $\ord_E(\fk{a}_m)>0$. Then the function $\lct(\fk{q}, t\cdot\fk{q}';\fk{a}_{m})$ is concave and increasing in $t$ as the infimum of a family of linear functions, for each fixed $m\in\bfZ_+$. It follows that the function $\lct(\fk{q}, t\cdot\fk{q}';\fk{a}_{\bullet})=\lim\limits_{m\to\infty} m\cdot \lct(\fk{q}, t\cdot\fk{q}';\fk{a}_{m})$ is also concave and increasing.
\end{proof}

The function $ t \longmapsto \lct(\fk{q}, t\cdot\fk{q}'; \fk{a}_{\bullet})$ is an algebro-geometric analogue of the so-called \emph{Tian function} named in \cite{BGMY23}, so we will also call the function by \textbf{Tian function}. When the graded sequence is associated to a valuation of finite log discrepancy, we have

\begin{proposition}\label{prop-lim.equals.vq'}
    Let $v\in\Val_X^*$ be a valuation with $A(v)<+\infty$. Then for any nonzero ideals $\fk{q}$, $\fk{q}'$, the limit
    \[\lim_{t\to +\infty}\frac{\lct(\fk{q},t\cdot\fk{q}';\fk{a}_{\bullet}^v)}{t}=v(\fk{q}').\]
\end{proposition}

\begin{proof}
    Lemma \ref{lem-lct.concave} shows the limit
    \[\mu\coloneqq\lim_{t\to +\infty}\frac{\lct(\fk{q},t\cdot\fk{q}';\fk{a}_{\bullet}^v)}{t}\]
    exists. First, we prove $\mu\le v(\fk{q}')$. For each $k\in\bfZ_+$, it holds that
    \[\lct(\fk{q},k\cdot\fk{q}';\fk{a}^v_{\bullet})=\lct^{\fk{q}\cdot{\fk{q}'}^k}(\fk{a}^v_{\bullet})\le\frac{A(v)+v(\fk{q})+k v(\fk{q}')}{v(\fk{a}^v_{\bullet})}=A(v)+v(\fk{q})+kv(\fk{q}'),\]
    which implies $\mu\le v(\fk{q}')$ since $A(v)<+\infty$.

    Next, we verify $\mu\ge v(\fk{q}')$. We may assume $v(\fk{q}')>0$. Set
    \[\fk{c}_{\bullet}=\{\fk{c}_m\}\coloneqq\left\{\fk{q}'^{\lceil m/v(\fk{q}')\rceil}\right\}_{m=1}^{\infty},\]
    which is a graded sequence of ideals. Then we have $\fk{c}_m\subseteq\fk{a}_{m}^{v}$ for each $m$. Thus, $\lct(\fk{q}, k\cdot\fk{q}';\fk{c}_{\bullet})\leq \lct(\fk{q}, k\cdot\fk{q}'; \fk{a}^v_{\bullet})<+\infty$ for every nonnegative integer $k$. By the definition of mixed jumping numbers and \Cref{cor-lct.equals.inf.Val.mixed.ver}, we get
    \begin{flalign*}
        \begin{split}
            \mathrm{lct}(\fk{q}, k\cdot\fk{q}';\fk{c}_{\bullet})&=\inf_{w\in \Val_X^*}\frac{A(w)+w(\fk{q})+k\cdot w(\fk{q}')}{w(\fk{c}_{\bullet})}\\
            &=\inf_{w\in \Val_X^*}\frac{A(w)+w(\fk{q})+k\cdot w(\fk{q}')}{w(\fk{q}')/v(\fk{q}')}\ge k\cdot v(\fk{q}').
        \end{split}
    \end{flalign*}
    It follows that
    \[\mu=\lim_{k\to\infty}\frac{\mathrm{lct}(\fk{q}, k\cdot\fk{q}';\fk{a}^{v}_{\bullet})}{k}\ge \liminf_{k\to\infty}\frac{\mathrm{lct}(\fk{q}, k\cdot\fk{q}';\fk{c}_{\bullet})}{k}\ge v(\fk{q}').\]

    The proof is complete.
\end{proof}

If additionally the valuation computes the jumping number, we can verify that for any nonzero ideal $\fk{q}'$ the Tian function is linear on the positive half-axis.

\begin{proposition}\label{prop-compute.linear}
    Let $\fk{q}$ and $\fk{q}'$ be two nonzero ideals on $X$, and $v\in\Val_X^*$ a valuation computing $\lct^{\fk{q}}(\fk{a}^v_{\bullet})<+\infty$. Then the Tian function
    \[t\longmapsto\lct(\fk{q}, t\cdot\fk{q}';\fk{a}_{\bullet}^v)\]
    is linear on $[0,+\infty)$.
\end{proposition}

\begin{proof}
    Denote $\T(t)=\lct(\fk{q}, t\cdot\fk{q}';\fk{a}_{\bullet}^v)$ for simplicity, where $t\in\bfR_{\ge 0}$. Due to the concavity of $\T(t)$ given by \Cref{lem-lct.concave}, we only need to prove that for each $k\in\bfZ_{\ge 0}$, $\T(k)$ is linear in $k$.

    Since $v$ computes $\lct^{\fk{q}}(\fk{a}^v_{\bullet})<+\infty$, we have $A(v)<+\infty$, and for each $w\in\Val_X$ with $w\ge v$, it holds that $A(w)+w(\fk{q})\ge A(v)+v(\fk{q})$ by \Cref{thm-TFAEinJM12}. Thus, for such $w$, it also holds that $A(w)+w(\fk{q}\cdot{\fk{q}'}^k)\ge A(v)+v(\fk{q}\cdot{\fk{q}'}^k)$ for each $k$. This shows that $v$ also computes $\lct^{\fk{q}\cdot{\fk{q}'}^k}(\fk{a}_{\bullet}^v)$ due to \Cref{thm-TFAEinJM12}, which implies
    \[\T(k)=\lct^{\fk{q}\cdot{\fk{q}'}^k}(\fk{a}_{\bullet}^v)=A(v)+v(\fk{q}\cdot{\fk{q}'}^k)=A(v)+v(\fk{q})+kv(\fk{q}'), \quad \forall k\in\bfZ_{\ge 0},\]
    which is linear in $k$. The proof is done.
\end{proof}

In fact, the linearity of Tian functions conversely implies that the valuation computes the jumping number. To prove this fact, we first present a technical lemma.

\begin{lemma}\label{lem-f(t).linear.on.T.+infty}
    Let $\fk{q}$ be a nonzero ideal on $X$, and let $v\in\Val_X^*$ be a valuation with $A(v)<+\infty$. If there exists a nonzero ideal $\fk{q}'$ such that $\T(t)=\lct(\fk{q},t\cdot\fk{q}';\fk{a}_{\bullet}^v)$ is linear in $t\in[t_0,+\infty)$ for some $t_0\ge 0$, then there exists $w\in\ZVal_X$ such that the following statements hold:
    \begin{enumerate}
        \item $w\ge v$;
        \item $w(\fk{q}')=v(\fk{q}')$;
        \item $\lct(\fk{q}, t_0\cdot\fk{q}'; \fk{a}_{\bullet}^v)=\lct(\fk{q}, t_0\cdot\fk{q}'; \fk{a}_{\bullet}^w)=A(w)+w(\fk{q})+t_0\cdot w(\fk{q}')$.
    \end{enumerate}
\end{lemma}

\begin{proof}
    Pick $N\in\bfZ_{+}$ with $N>t_0$. Then by \Cref{prop-lctqav'=lctqav} there exists a Zhou valuation $\tilde{w}$ related to $\fk{q}\cdot{\fk{q}'}^N$ such that $w\coloneqq \T(N) \tilde{w}\ge v$ and
    \[\lct(\fk{q}, N\cdot\fk{q}'; \fk{a}_{\bullet}^w)=\lct^{\fk{q}\cdot{\fk{q}'}^N}(\fk{a}_{\bullet}^w)=\lct^{\fk{q}\cdot{\fk{q}'}^N}(\fk{a}_{\bullet}^v)=\lct(\fk{q},N\cdot \fk{q}';\fk{a}_{\bullet}^v)=\T(N).\]
    Denote $\wt{\T}(t)=\lct(\fk{q}, t\cdot\fk{q}'; \fk{a}_{\bullet}^w)$, where $t\in [0,+\infty)$. Note $\wt{\T}(t)\ge \T(t)$ since $w\ge v$. Since $\wt{\T}$ is concave on $(0,+\infty)$ by \Cref{lem-lct.concave}, $\wt{\T}(N)=\T(N)$, and $\T$ is linear on $[t_0,+\infty)$, it follows that $\wt{\T}\equiv \T$ on $[t_0,+\infty)$. Then since $A(v)<+\infty$, we get $\wt{T}(t_0)=\T(t_0)$ and 
    \[w(\fk{q}')=\lim_{t\to +\infty}\frac{\wt{\T}(t)}{t}=\lim_{t\to +\infty} \frac{\T(t)}{t}=v(\fk{q}'),\]
    according to Proposition \ref{prop-lim.equals.vq'}.

    What remains to be proved for $w$ is $\lct(\fk{q}, t_0\cdot\fk{q}'; \fk{a}_{\bullet}^w)=A(w)+w(\fk{q})+t_0\cdot w(\fk{q}')$. Since $\tilde{w}$ is a Zhou valuation related to $\fk{q}\cdot {\fk{q}'}^N$, due to \Cref{thm-A(v)=1-v(q)} we have
    \[\wt{\T}(N)=\T(N)=A(w)+w(\fk{q}\cdot {\fk{q}'}^N)=A(w)+w(\fk{q})+N\cdot w(\fk{q}').\]
    Thus,
    \begin{flalign*}
        \begin{split}
            \wt{\T}(t_0)=\T(t_0)=&\T(N)-(N-t_0)\cdot v(\fk{q}')\\
            =&A(w)+w(\fk{q})+N\cdot w(\fk{q}')- (N-t_0)\cdot w(\fk{q}')\\
            =&A(w)+w(\fk{q})+t_0\cdot w(\fk{q}').
        \end{split}
    \end{flalign*}
    The proof is complete.
\end{proof}

Now we prove the following proposition, which is a converse of Proposition \ref{prop-compute.linear}. And this completes the proof of \Cref{thm-main thmB} (1).

\begin{proposition}\label{prop-linear.imply.compute}
    Let $\fk{q}$ be a nonzero ideal on $X$ and $v\in\Val_X^*$ with $A(v)<+\infty$. If the function
    \[t \longmapsto \lct(\fk{q}, t\cdot\fk{q}' ; \fk{a}_{\bullet}^v)\]
    is linear in $t\in [0,+\infty)$ for every nonzero ideal $\fk{q}'$ on $X$, then $v$ computes $\lct^{\fk{q}}(\fk{a}_{\bullet}^v)$.
\end{proposition}

\begin{proof}
Assume that $v$ does not compute $\lct^{\fk{q}}(\fk{a}_{\bullet}^v)\in (0,+\infty)$, i.e.,
    \[\lct^{\fk{q}}(\fk{a}_{\bullet}^v)< A(v)+v(\fk{q}),\]
    and we prove the desired result by contradiction. 
    
    Without loss of generality, we assume $\lct^{\fk{q}}(\fk{a}_{\bullet}^v)=1$. By the definition of log-discrepancy, there exists a log-smooth pair $(Y,E)$ over $X$ such that $v'\coloneqq r_{Y,E}(v)$ satisfies
    \[A(v')+v'(\fk{q})=A(v)+v(\fk{q})-\varepsilon>1.\] 

    Choose affine neighborhoods $U=\Spec R$ of $c_X(v)$ and $V\subseteq Y\times_X U$ of $c_Y(v)$ such that the components $E_1,\ldots,E_N$ of $E$ through $c_Y(v)$ have equations $y_1,\ldots,y_N$ on $V$. Write $\Gamma(V,\cal{O}_Y)=R[y_1,\ldots,y_s]$ with $y_j=a_j/b$ for a common denominator $0\neq b\in R$ and $a_j\in R$. Let $\fk{b}$ be the ideal of the scheme-theoretic closure in $X$ of the subscheme $(b=0)\subseteq U$. Giving $X\setminus U$ its reduced structure, set
    \[\fk{q}'=\fk{q}\cdot\fk{b}\cdot\mathcal I_{X\setminus U}.\]
    Then according to Lemma \ref{lem-f(t).linear.on.T.+infty}, we can find some $w\in\ZVal_X$ such that $w\ge v$, $w(\fk{q}')=v(\fk{q}')$, and
    \[1=\lct^{\fk{q}}(\fk{a}_{\bullet}^v)=\lct^{\fk{q}}(\fk{a}_{\bullet}^w)=A(w)+w(\fk{q}).\]
    Since $w\ge v$, equality on $\fk{q}'$ gives equality on each factor. Thus $w(\fk{q})=v(\fk{q})\ge v'(\fk{q})$ and $w(\mathcal I_{X\setminus U})=0$, so $c_X(w)\in U$ and $w(b)=v(b)$. By \cite[Lemma~4.4]{JM12}, $w\ge v$ on $R$. Hence $w(y_j)=w(a_j)-w(b)\ge v(y_j)\ge0$ for all $j$, so $\Gamma(V,\cal{O}_Y)\subseteq\cal{O}_w$ and $c_Y(w)\in V$. It follows that
    \[w(E_i)=w(y_i)\ge v(y_i)=v(E_i)=v'(E_i), \quad 1\le i\le N.\]
    Then $A(w)\ge A(r_{Y,E}(w))\ge A(v')$ due to (\ref{equ-A(v) for QM(Y,D)}). However,
    \[1=A(w)+w(\fk{q})\ge A(v')+v'(\fk{q})>1,\]
    which is a contradiction. The proof of that $v$ computes $\lct^{\fk{q}}(\fk{a}^v_{\bullet})$ is done. 
\end{proof}

\section{Enlarging a graded sequence}\label{section7}
This section provides a natural generalization of \cite[Section 7.4]{JM12} to the context of mixed jumping numbers. In addition, it serves as preparation for the proof of \Cref{thm-main thmB} (2).

Let $\fk{a}_{\bullet}$ be a graded sequence, and $\fk{q}'$ be a nonzero ideal on $X$. Note that for every $\beta>0$, if we set
\[\fk{c}_m\coloneqq \sum_{i=0}^m \fk{a}_i\cdot {\fk{q}'}^{\lceil \beta(m-i) \rceil}, \quad \forall\, m\in\bfZ_{\ge 0},\]
then $\fk{c}_{\bullet}=\{\fk{c}_m\}$ is a graded sequence such that $\fk{c}_m\supseteq \fk{a}_m\cup{\fk{q}'}^{\lceil \beta m \rceil}$ for each $m\in\bfZ_{\ge 0}$. This is a way to enlarge a graded sequence (see \cite[Section 7.4]{JM12}). 	We first establish a result for mixed jumping numbers, which is the key lemma to prove the differentiability part of \Cref{thm-main thmB} (2).
\begin{lemma}\label{lem-lct.c.equal.lct.a}
    Let $\fk{q}, \fk{q}'$ be two nonzero ideals, and $\fk{a}_{\bullet}$ a graded sequence of nonzero ideals. Let $\beta>0$ and denote by $\fk{c}_{\bullet}=\{\fk{c}_m\}$ the graded sequence given by
    \[\fk{c}_m=\sum_{i=0}^m\fk{a}_{i}\cdot{\fk{q}'}^{\lceil\beta (m-i)\rceil}, \quad m\in\bfZ_{\ge 0}.\]
     Given $\lambda\in\bfR$, if
    \[0<\lct(\fk{q}, (\lambda-\beta d)\cdot\fk{q}'; \fk{a}_{\bullet})+d\le \lct(\fk{q}, \lambda\cdot\fk{q}'; \fk{a}_{\bullet})<+\infty\]
    for every sufficiently small $d>0$, then $\lct(\fk{q}, \lambda\cdot\fk{q}'; \fk{c}_{\bullet})=\lct(\fk{q}, \lambda\cdot\fk{q}'; \fk{a}_{\bullet})$.
\end{lemma}

\begin{proof}
    Since $\fk{c}_m\supseteq\fk{a}_m$ for each $m$, we have $\lct(\fk{q}, \lambda\cdot\fk{q}'; \fk{c}_{\bullet})\ge\lct(\fk{q}, \lambda\cdot\fk{q}'; \fk{a}_{\bullet})$. For each $v\in\Val_X^*$,
    \begin{flalign*}
        \begin{split}
    v(\fk{c}_m)&=v\left(\sum_{i=0}^m\fk{a}_{i}\cdot{\fk{q}'}^{\lceil \beta(m-i)\rceil}\right)\\
            &=\min_i\big\{v(\fk{a}_i)+\left\lceil \beta(m-i) \right\rceil\cdot v(\fk{q}')\big\}\\
            &=\min\big\{\beta\cdot v(\fk{q}') m, v(\fk{a}_m)\big\}+O(1),
        \end{split}
    \end{flalign*}
    which shows
    \begin{equation}\label{eq-v.c.bullet}
        v(\fk{c}_{\bullet})=\min\big\{\beta\cdot v(\fk{q}'), v(\fk{a}_{\bullet})\big\}.
    \end{equation}
    
    For every sufficiently small $d>0$, pick any $w\in\Val_X^*$ such that $w(\fk{a}_{\bullet})=1$ and
    \[\lct(\fk{q}, (\lambda-\beta d)\cdot\fk{q}'; \fk{a}_{\bullet})=A(w)+w(\fk{q})+(\lambda-\beta d) \cdot w(\fk{q}')-\delta(d)\cdot d,\]
    where $\delta(d)\to 0^+$ when $d\to 0^+$. Then 
    \begin{flalign*}
        \begin{split}
           &A(w)+w(\fk{q})+(\lambda-\beta d)\cdot w(\fk{q}')-\delta(d)\cdot d\\
           =& \lct(\fk{q}, (\lambda-\beta d)\cdot\fk{q}'; \fk{a}_{\bullet})\\
           \le& \lct(\fk{q}, \lambda\cdot\fk{q}'; \fk{a}_{\bullet})-d\\
            \le& \frac{A(w)+w(\fk{q})+\lambda \cdot w(\fk{q}')}{w(\fk{a}_{\bullet})}-d\\
            =& A(w)+w(\fk{q})+\lambda \cdot w(\fk{q}')-d,
        \end{split}
    \end{flalign*}
    yielding $\beta\cdot w(\fk{q}')\ge 1-\delta(d)$. Combining with (\ref{eq-v.c.bullet}), we get $w(\fk{c}_{\bullet})\ge 1-\delta(d)$. It follows that
    \begin{flalign*}
        \begin{split}
            \lct(\fk{q}, (\lambda-\beta d)\cdot\fk{q}'; \fk{c}_{\bullet})&\le \frac{A(w)+w(\fk{q})+(\lambda-\beta d)\cdot w(\fk{q}')}{w(\fk{c}_{\bullet})}\\
            &\le\frac{A(w)+w(\fk{q})+(\lambda-\beta d) \cdot w(\fk{q}')}{1-\delta(d)}\\
            &=\frac{\lct(\fk{q}, (\lambda-\beta d)\cdot\fk{q}'; \fk{a}_{\bullet})+\delta(d)\cdot d}{1-\delta(d)}.
        \end{split}
    \end{flalign*}
    Let $d\to 0^+$, and by the continuity of Tian functions (see Lemma \ref{lem-lct.concave}), we obtain $\lct(\fk{q}, \lambda\cdot\fk{q}'; \fk{c}_{\bullet})\le\lct(\fk{q}, \lambda\cdot\fk{q}'; \fk{a}_{\bullet})$.
    
    Consequently, $\lct(\fk{q}, \lambda\cdot\fk{q}'; \fk{c}_{\bullet})=\lct(\fk{q}, \lambda\cdot\fk{q}'; \fk{a}_{\bullet})$.
\end{proof}

Moreover, for Zhou valuations, we have the following

\begin{lemma}\label{lem-beta.ge.vq'}
    Let $v$ be a Zhou valuation related to a nonzero ideal $\fk{q}$, and $\fk{q}'$ a nonzero ideal. Assume that a graded sequence $\fk{c}_{\bullet}=\{\fk{c}_m\}$ of ideals satisfying
\[\fk{c}_m=\sum_{i=0}^m\fk{a}^v_{i}\cdot{\fk{q}'}^{\lceil\beta (m-i)\rceil}, \quad m\in\bfZ_{\ge 0},\]
    where $\beta$ is a positive constant. Then we have $\lct^{\fk{q}}(\fk{c}_{\bullet})=1$ if $\beta\ge 1/v(\fk{q}')$, and $\lct^{\fk{q}}(\fk{c}_{\bullet})>1$ if $\beta<1/v(\fk{q}')$.
\end{lemma}

\begin{proof}
    Note that we have $\fk{c}_m\supseteq \fk{a}_m^v\cup{\fk{q}'}^{\lceil \beta m \rceil}$ for all $m\in\bfZ_{\ge 0}$. Since $\fk{a}^v_m\subseteq\fk{c}_m$ for each $m$, we have $\lct^{\fk{q}}(\fk{c}_{\bullet})\ge\lct^{\fk{q}}(\fk{a}^v_{\bullet})=1$.

    If $\beta\ge 1/v(\fk{q}')$, then $v\left({\fk{q}'}^{\lceil \beta m \rceil}\right)=\lceil \beta m \rceil\cdot v(\fk{q}')\ge m$, which implies ${\fk{q}'}^{\lceil \beta m \rceil}\subseteq \fk{a}_m^v$ and thus $\fk{a}_{m}^v=\fk{c}_m$ for each $m$. It is clear that $\lct^{\fk{q}}(\fk{c}_{\bullet})=1$ for this situation.

    Next we assume $\lct^{\fk{q}}(\fk{c}_{\bullet})=1$. Then Theorem \ref{thm-existence.gr.sys} shows that there exists a Zhou valuation $w$ related to $\fk{q}$ satisfying $\fk{c}_m\subseteq\fk{a}_{m}^w$ for every $m$. This implies $\fk{a}_m^v\subseteq\fk{a}_m^w$ for each $m$, and then $w\ge v$. The fact that $v$ is a Zhou valuation related to $\fk{q}$ indicates $w=v$. We get $\fk{a}_m^v=\fk{c}_m$ for each $m$. It follows that $v\left({\fk{q}'}^{\lceil \beta m \rceil}\right) \ge m$, which gives $\beta\ge 1/v(\fk{q}')$. 
    Consequently, we have $\lct^{\fk{q}}(\fk{c}_{\bullet})>1$ if $\beta< 1/v(\fk{q}')$.
\end{proof}


\section{Tian functions of Zhou valuations}\label{section8}
In this section, we prove \Cref{thm-main thmB} (2) by enlarging graded sequences, as in the previous section. First, we prove that a valuation being a Zhou valuation will imply the linearity on $[0,+\infty)$ and the differentiability at $t=0$ of Tian functions. 

\begin{theorem}\label{thm-ZV.lct.differential}
    Let $\fk{q}$ and $\fk{q}'$ be two nonzero ideals on $X$, and $v\in\Val_X^*$ a Zhou valuation related to $\fk{q}$. Then the function
    \[(-\varepsilon_0,+\infty)\ni t\longmapsto\lct(\fk{q}, t\cdot\fk{q}';\fk{a}_{\bullet}^v)\]
    is linear on $[0,+\infty)$, and differentiable at $t=0$.
\end{theorem}

\begin{proof}[Proof of Theorem \ref{thm-ZV.lct.differential}]
    Denote $\T(t)=\lct(\fk{q}, t\cdot\fk{q}';\fk{a}_{\bullet}^v)$ for $t\in (-\varepsilon_0,+\infty)$. By Theorem \ref{thm-A(v)=1-v(q)} and Proposition \ref{prop-compute.linear}, the linearity of $\T(t)$ on $[0,+\infty)$ holds true. In the following we prove that $\T(t)$ is differentiable at $t=0$.

    Due to the monotonicity and concavity of $\T(t)$, we have that $\T'_{+}(0)$ and $\T'_{-}(0)$ exist, and $0\le \T'_+(0)\le \T'_-(0)$. It suffices to prove $\T'_+(0)\ge \T'_-(0)$. We may assume $\T'_-(0)>0$. Denote $\gamma\coloneqq 1/\T'_-(0)$. Then for every sufficiently small $d>0$, we have
    \[\T\left(-\gamma d\right)+d\le \T(0).\]
    Set $\fk{c}_{\bullet}=\{\fk{c}_m\}$, where
    \[\fk{c}_m=\sum_{i=0}^m\fk{a}^v_{i}\cdot{\fk{q}'}^{\lceil \gamma (m-i)\rceil}, \quad m\in\bfZ_{\ge 0}.\]
    Now \Cref{prop-lct(Zhou val)=1} and \Cref{lem-lct.c.equal.lct.a}  yield that $\lct^{\fk{q}}(\fk{c}_{\bullet})=\lct^{\fk{q}}(\fk{a}_{\bullet}^v)=1$, and it follows from \Cref{lem-beta.ge.vq'} that
    \[\gamma\ge 1/v(\fk{q}') \Rightarrow \T'_{-}(0)\le v(\fk{q}').\]
    Proposition \ref{prop-compute.linear} shows $\T'_{+}(0)=v(\fk{q}')$. Thus, we get $\T'_+(0)\ge \T'_-(0)$, which completes the proof of that $\T(t)$ is differentiable at $t=0$.
\end{proof}

Conversely, we have the following result. 

\begin{proposition}\label{prop-converse.differential.ZVal}
    Let $\fk{q}$ be a nonzero ideal on $X$. If $v\in\Val_X^*$ with $A(v)<+\infty$ satisfying that the function
    \begin{flalign*}
        \begin{split}
            \T_{\fk{q}'}(t)\colon (-\varepsilon_{\fk{q}'},+\infty)&\longrightarrow \bfR_{\ge 0}\\
            t&\longmapsto \lct(\fk{q}, t\cdot\fk{q}';\fk{a}_{\bullet}^v)
        \end{split}
    \end{flalign*}
    is differentiable at $t=0$ for every nonzero ideal $\fk{q}'$ on $X$, then
    \[\fk{q}'\longmapsto \frac{\mathrm{d} \T_{\fk{q}'}}{\mathrm{d} t}\Big|_{t=0}\]
    is a valuation on $X$ belonging to $\ZVal_X$. If additionally assume that $\T_{\fk{q}'}(t)$ is linear on $[0,+\infty)$ for every ideal $\fk{q}'$ on $X$ and $\lct^{\fk{q}}(\fk{a}_{\bullet}^v)=1$, then $v$ is a Zhou valuation related to $\fk{q}$. 
\end{proposition}

\begin{proof}
    Let $w\in\ZVal_X$ with $\lct^{\fk{q}}(\fk{a}_{\bullet}^v)=\lct^{\fk{q}}(\fk{a}_{\bullet}^w)$ and $w\ge v$, given by \Cref{prop-lctqav'=lctqav}, and the associated Zhou valuation of $w$ is related to $\fk{q}$.
    Denote
    \[\wt{\T}_{\fk{q}'}(t)=\lct(\fk{q},t\cdot\fk{q}';\fk{a}_{\bullet}^w),\]
    for each nonzero ideal $\fk{q}'$ on $X$. Then $\wt{\T}_{\fk{q}'}\ge \T_{\fk{q}'}$, and by (a variation of) \Cref{thm-ZV.lct.differential}, $\wt{\T}_{\fk{q}'}$ is differentiable at $t=0$ with
    \[\frac{\mathrm{d} \wt{\T}_{\fk{q}'}}{\mathrm{d} t}\Big|_{t=0}=w(\fk{q}').\]
    Since $\T_{\fk{q}'}$ is concave and differentiable at $t=0$, it must hold that
    \[\frac{\mathrm{d} \T_{\fk{q}'}}{\mathrm{d} t}\Big|_{t=0}=\frac{\mathrm{d} \wt{\T}_{\fk{q}'}}{\mathrm{d} t}\Big|_{t=0}=w(\fk{q}').\]
    It follows that $\fk{q}'\mapsto\frac{\mathrm{d} \T_{\fk{q}'}}{\mathrm{d} t}\Big|_{t=0}$ is a Zhou valuation (equals to $w$). If additionally we have that $\T_{\fk{q}'}(t)$ is linear on $[0,+\infty)$ for every ideal $\fk{q}'$ on $X$ and $\lct^{\fk{q}}(\fk{a}_{\bullet}^v)=1$, then by \Cref{prop-lim.equals.vq'},
    \[\frac{\mathrm{d} \T_{\fk{q}'}}{\mathrm{d} t}\Big|_{t=0}=v(\fk{q}')\quad \text{and} \quad \lct^{\fk{q}}(\fk{a}_{\bullet}^w)=\lct^{\fk{q}}(\fk{a}_{\bullet}^v)=1,\]
    and thus $w$ can be chosen as a Zhou valuation related to $\fk{q}$. Consequently, $v=w$ is also a Zhou valuation related to $\fk{q}$.
\end{proof}

\section{Denseness of the cone of Zhou valuations}\label{section9}

We demonstrate that the cone of Zhou valuations is dense in the cone of valuations. A similar statement can be found in \cite[Section 1.4]{BGY23}, and our proof here is almost the same of there.

\begin{lemma}\label{lem-vq.=.inf.tildevq}
    Let $v\in\Val_X^{*}$ with $A(v)<+\infty$. Then for every nonzero ideal $\fk{q}$ on $X$,
    \[v(\fk{q})=\inf\big\{\tilde{v}(\fk{q})\colon \tilde{v}\in\ZVal_X, \ \tilde{v}\ge v\big\}.\] 
\end{lemma}

\begin{proof}
  It suffices to verify that 
  \[v(\fk{q})\geq \inf\big\{\tilde{v}(\fk{q})\colon \tilde{v}\in\ZVal_X, \ \tilde{v}\ge v\big\}.\]
  For every $k\in\mathbf{Z}_+$, we have $\gamma(k)\coloneqq\lct^{\fk{q}^k}(\fk{a}_{\bullet}^v)<+\infty$, since $A(v)<+\infty$. Then for every given $k$ we have $\lct^{\fk{q}^k}(\fk{a}_{\bullet}^{v/\gamma(k)})=1$, and there exists a Zhou valuation $\tilde{v}_k$ related to $\fk{q}^k$ such that 
  \begin{enumerate}
      \item[(1)] $\frac{v}{\gamma(k)}\leq \tilde{v}_k$ by \Cref{lem-Val(X;q)} and \Cref{cor-existence.valuation};
      \item[(2)] $\tilde{v}_k(\fk{q}^k)=1-A(\tilde{v}_k)\leq 1$ by \Cref{thm-A(v)=1-v(q)}.
  \end{enumerate}
  Let $\hat{v}_k\coloneqq \gamma(k)\cdot\tilde{v}_k$. Then $\hat{v}_k\in \ZVal_X$ with $\hat{v}_k\geq v$, and 
  \[v(\fk{q})\leq \hat{v}_k(\fk{q})=\frac{\gamma(k)}{k}\cdot\tilde{v}_k(\fk{q}^k)\leq \frac{\lct^{\fk{q}^k}(\fk{a}_{\bullet}^v)}{k}\leq v(\fk{q})+\frac{A(v)}{k},\]
  where the last inequality is given by \Cref{lem-JM12 section6.2}. Due to $A(v)<+\infty$, we obtain
  \[v(\fk{q})\geq \limsup_{k\to\infty} \hat{v}_k(\fk{q})\geq \inf\big\{\tilde{v}(\fk{q})\colon \tilde{v}\in\ZVal_X, \ \tilde{v}\ge v\big\},\]
  and finish the proof.
\end{proof}

\begin{theorem}\label{thm-cone of Zhou dense}
    The cone of the Zhou valuations $\ZVal_X$ on $X$ is dense in $\Val_X$.
\end{theorem}

\begin{proof}
   Note that for every quasi-monomial valuation $v$ we have $A(v)<+\infty$. Therefore, by \Cref{prop-structure thm of ValX}(1), it suffices to show that for every $v\in \Val_X^*$ with $A(v)<+\infty$,  for any finite fixed non-trivial and nonzero ideals $\fk{q}_1,\cdots,\fk{q}_r$ ($r\in \bfZ_+$), and for every given integer $n\in \mathbf{Z}_+$, there exists a valuation $\tilde{v}\in \ZVal_X$ such that 
  \[\max_{1\leq i\leq r}|\tilde{v}(\fk{q}_i)-v(\fk{q}_i)|<\frac{1}{n}.\]
  Set $\fk{q}\coloneqq \fk{q}_1\cdots \fk{q}_r$. Then according to Lemma \ref{lem-vq.=.inf.tildevq}, for any $\epsilon>0$, we can find a valuation $\tilde{v}\in \ZVal_X$ with ${v}\leq \tilde{v}$ and $\tilde{v}(\fk{q})< v(\fk{q})+\epsilon$. Then we can choose $\epsilon$ to be small enough so that $|\tilde{v}(\fk{q}_i)-v(\fk{q}_i)|<\frac{1}{n}$ for all $i\in \{1,\cdots,r\}$.
\end{proof}

\section{Singularities of graded sequences and Zhou valuations}\label{section10}

One important application of Zhou weights in \cite{BGMY23} is to characterize the singularities of plurisubharmonic functions, which recovers and generalizes a part of the main result of \cite{BFJ08}, while the Zhou valuations are also used to characterize the division relation of germs of holomorphic functions in \cite{BGMY23}. In this section, we show that the algebraic Zhou valuations can be used to characterize the singularities of graded sequences of ideals. Our result is as follows:

\begin{theorem}
    Let $\fk{a}_{\bullet}=\{\fk{a}_m\}_{m\ge 1}$, $\fk{a}'_{\bullet}=\{\fk{a}'_{m}\}_{m\ge 1}$ be two nonzero graded sequences of ideals on $X$. Then the following statements are equivalent:
\begin{enumerate}
    \item $\mathcal{J}(t\cdot\fk{a}_{\bullet})\subseteq \mathcal{J}(t\cdot\fk{a}'_{\bullet})$ for every $t>0$.
    \item $v(\fk{a}_{\bullet})\ge v(\fk{a}'_{\bullet})$ for every $v\in\Val_X$ with $A(v)<+\infty$.
    \item $v(\fk{a}_{\bullet})\ge v(\fk{a}'_{\bullet})$ for every $v\in\ZVal_X$.
\end{enumerate} 
\end{theorem}

\begin{proof}
    Obviously, we have (2) $\Rightarrow$ (3). Then it suffices to prove (1) $\Rightarrow$ (2) and (3) $\Rightarrow$ (1).

    (1) $\Rightarrow$ (2): Without loss of generality, we assume $v(\fk{a}'_{\bullet})>0$. Fix $k\in\bfZ_{\ge 1}$. Denote $\fk{q}_k\coloneqq \mathcal{J}(k\cdot\fk{a}_{\bullet})$. Since $\fk{q}_k\subseteq \mathcal{J}(k\cdot\fk{a}'_{\bullet})$, we have
    \begin{equation}\label{eq-lct.qk.>k}
        \frac{A(v)+v(\fk{q}_k)}{v(\fk{a}'_{\bullet})}\ge\lct^{\fk{q}_k}(\fk{a}'_{\bullet})>k.
    \end{equation}
     On the other hand, we can find an integer $p_k>0$ such that
    \[\mathcal{J}(k\cdot\fk{a}_{\bullet})=\mathcal{J}\Big(\frac{k}{pk}\cdot\fk{a}_{pk}\Big)=\mathcal{J}\Big(\frac{1}{p}\cdot\fk{a}_{pk}\Big)\]
    for every positive multiple $p$ of $p_k$ (cf. \cite[Prop. 11.1.18]{LarII04}). The subadditivity property of multiplier ideals (see \cite[Theorem A.2]{JM12} or \cite[Theorem 9.5.17]{LarII04}) implies
    \[\fk{q}_k^p=\mathcal{J}\Big(\frac{1}{p}\cdot\fk{a}_{pk}\Big)^p\supseteq \mathcal{J}\Big(\frac{p}{p}\cdot\fk{a}_{pk}\Big)=\mathcal{J}(\fk{a}_{pk})\supseteq \fk{a}_{pk},\]
    yielding that
    \[v(\fk{q}_k)\le \frac{1}{p}v(\fk{a}_{pk})\qquad (p_k\mid p).\]  
    Combining with (\ref{eq-lct.qk.>k}) we get
    \[v(\fk{a}'_{\bullet})<\frac{A(v)+v(\fk{q}_k)}{k}\le \frac{A(v)}{k}+\frac{1}{pk}v(\fk{a}_{pk}).\]
    Letting $p\to+\infty$ through multiples of $p_k$, we obtain
    \[v(\fk{a}'_{\bullet})\le \frac{A(v)}{k}+v(\fk{a}_{\bullet}).\]
    Finally, we let $k\to +\infty$, and thus getting $v(\fk{a}'_{\bullet})\le v(\fk{a}_{\bullet})$.

    (3) $\Rightarrow$ (1): If $(1)$ is not true, then there exist $t_0>0$ and a nonzero ideal $\fk{q}$ on $X$ such that $\fk{q}\subseteq\mathcal{J}(t_0\cdot\fk{a}_{\bullet})$ and $\fk{q}\not\subseteq\mathcal{J}(t_0\cdot\fk{a}'_{\bullet})$. Thus,
    \[t_1\coloneqq\lct^{\fk{q}}(\fk{a}'_{\bullet})\leq t_0<\lct^{\fk{q}}(\fk{a}_{\bullet}).\]
    
    Since $\fk{a}_{\bullet}'$ is nonzero, we have $t_1>0$. By Corollary \ref{cor-ZV.computes}, there exists a Zhou valuation related to $\fk{q}$ with $v(\fk{a}_{\bullet}')=1/t_1$, computing $\lct^{\fk{q}}(\fk{a}_{\bullet}')$. Then it follows from the assumption of (3) that
    \[\lct^{\fk{q}}(\fk{a}_{\bullet})\le\frac{A(v)+v(\fk{q})}{v(\fk{a}_{\bullet})}\le\frac{1}{v(\fk{a}_{\bullet}')}=t_1,\]
    which contradicts to $t_1<\lct^{\fk{q}}(\fk{a}_{\bullet})$. Thus, (3) $\Rightarrow$ (1).
    
    The proof is complete.
\end{proof}

The following proposition is a characterization of the asymptotic multiplier ideals via Zhou valuations which is an analogue of \cite[Corollary 1.17]{BGMY23}.

\begin{proposition}
    For every nonzero ideal $\fk{q}\subseteq \mathcal{O}_X$, $\fk{a}_{\bullet}$ a graded sequence of ideals on $X$, and $\lambda>0$, we have $\fk{q}\subseteq\mathcal{J}(\lambda\cdot\fk{a}_{\bullet})$ if and only if $v(\fk{a}_{\bullet})<\frac{1}{\lambda}$ for every Zhou valuation $v$ related to $\fk{q}$ on $X$.
\end{proposition}

\begin{proof}
    This proposition holds since
    \begin{flalign*}
        \begin{split}
            \fk{q}\subseteq\mathcal{J}(\lambda\cdot\fk{a}_{\bullet})\quad\Longleftrightarrow&\quad \lct^{\fk{q}}(\fk{a}_{\bullet})>\lambda\\
            \Longleftrightarrow &\quad \min_{v\in\ZVal_X(\fk{q})}\frac{A(v)+v(\fk{q})}{v(\fk{a}_{\bullet})}>\lambda\\
            \Longleftrightarrow&\quad\min_{v\in\ZVal_X(\fk{q})}\frac{1}{v(\fk{a}_{\bullet})}>\lambda\\    \Longleftrightarrow&\quad\max_{v\in\ZVal_X(\fk{q})}v(\fk{a}_{\bullet})<\frac{1}{\lambda},
        \end{split}
    \end{flalign*}
where we recall that $\ZVal_X(\fk{q})$ denotes the set of Zhou valuations related to $\fk{q}$ on~$X$.
\end{proof}

\section{Analytic correspondence}\label{section11}

Since the ideal of algebraic Zhou valuations actually comes from the \emph{analytic Zhou valuations}, it is natural to consider whether one can reduce the algebraic Zhou valuations to the original analytic Zhou valuations on some special schemes. 

In \cite{JM14}, the authors gave the ways to link the objects in algebraic and analytic settings. We recall some constructions and observations from \cite{JM14}. Let $U$ be a complex manifold, and let $V$ be the \emph{germ} of a complex submanifold at a point $x$ in $U$. Denote by $\cal{O}_{x,V}$ the localization of $\cal{O}_{x}$ along the ideal $I_V$, where $\cal{O}_x$ denotes the ring of germs of \emph{holomorphic} functions at $x$. Then $\cal{O}_{x,V}$ is an excellent regular domain with maximal ideal $\fk{m}_{x,V}$. Set
\[\lct_{x,V}(\fk{q},\lambda\cdot\fk{q}';\fk{a}_{\bullet})\coloneqq \lct(\fk{q}\cdot\cal{O}_{x,V},\lambda\cdot(\fk{q}'\cdot\cal{O}_{x,V});\fk{a}_{\bullet}\cdot\cal{O}_{x,V}),\]
and
\[\lct_{x,V}(\fk{q},\lambda\cdot\fk{q}';\fk{b}_{\bullet})\coloneqq \lct(\fk{q}\cdot\cal{O}_{x,V},\lambda\cdot(\fk{q}'\cdot\cal{O}_{x,V});\fk{b}_{\bullet}\cdot\cal{O}_{x,V}),\]
where $\fk{q},\fk{q}'$ are ideals of $\cal{O}_x$, $\fk{a}_{\bullet}$ is a graded sequence of ideals on $X$, and $\fk{b}_{\bullet}$ is a subadditive system of ideals on $X$. Then
\[\lct_{x,V}(\fk{q},\lambda\cdot\fk{q}';\fk{a}_{\bullet})=\inf_{v}\frac{A(v)+v(\fk{q})+\lambda\cdot v(\fk{q}')}{v(\fk{a}_{\bullet})}\]
and
\[\lct_{x,V}(\fk{q},\lambda\cdot\fk{q}';\fk{b}_{\bullet})=\inf_{v}\frac{A(v)+v(\fk{q})+\lambda\cdot v(\fk{q}')}{v(\fk{b}_{\bullet})},\]
where the infimums are over all (quasi-monomial, or Zhou) valuations of $\cal{O}_{x,V}$ by \Cref{thm-cone of Zhou dense}. In this section, we only care about the case $V=\{x\}$.

Let $\varphi$ be a plurisubharmonic (psh for short) function on a complex manifold $U$. Then the \emph{multiplier ideal sheaf} $\cal{J}(\varphi)$ is a sheaf on $U$ such that at any point $x\in U$ the stalk is defined by
\begin{equation}\label{equ-MIS}
    \cal{J}(\varphi)_x\coloneqq \big\{f\in\cal{O}_x\colon |f|^2e^{-2\varphi} \text{ is locally integrable at } x\big\}.
\end{equation}
The sheaf $\cal{J}(\varphi)$ is a coherent sheaf (\cite{Nad89,Nad90}). For non-zero ideals $\fk{q},\fk{q}'$ on $U$ and $\lambda\in\bfR$, define the \emph{mixed jumping number} of $\varphi$ at $x$ relative to $\fk{q},\fk{q}'$ and $\lambda$ by
\[c^{\fk{q},\lambda\cdot\fk{q}'}_{x}(\varphi)\coloneqq \sup\big\{c\ge 0\colon |\fk{q}|^2|\fk{q}'|^{2\lambda}e^{-2c\varphi} \text{ is locally integrable at } x\big\},\]
where it is the \emph{classical jumping number} if $\fk{q}'=\cal{O}_U$ or $\lambda=0$.

The following proposition is \cite[Proposition 3.12]{JM14} after some slight changes.

\begin{proposition}[\cite{JM14}]\label{prop-jm.equa.lct}
    Let $x$ be any point in $U$, and $V$ be the germ of a proper complex submanifold at $x$. Let $\varphi$ be a psh function on $U$, and set
    \[\fk{b}_t\coloneqq\cal{J}(t\varphi), \ t>0,\]
    which is a subadditive sequence of ideals on $U$. The following statements hold.
    \begin{enumerate}
    \item For nonzero ideals $\fk{q},\fk{q}'\subseteq\cal{O}_x$ and $\lambda\in(-\varepsilon_0,+\infty)$, where $0<2\varepsilon_0<\lct_x(\fk{q}')$, we have
    \[c_{x}^{\fk{q},\lambda\cdot\fk{q}'}(\varphi)=\lct_x(\fk{q},\lambda\cdot\fk{q}';\fk{b}_{\bullet}).\]

    \item The subadditive sequence $\fk{b}_{\bullet}\cdot\cal{O}_{x,V}$ has controlled growth.

    \item If $\varphi=\log|\fk{a}|+O(1)$ near $x$ for some nonzero ideal $\fk{a}\subseteq\cal{O}_x$, then $\fk{b}_t=\cal{J}(t\cdot\fk{a})$ for every $t>0$. Therefore, we have $c_{x}^{\fk{q},\lambda\cdot\fk{q}'}(\varphi)=\lct_x(\fk{q},\lambda\cdot\fk{q}';\fk{a})$ for $\fk{q},\fk{q}'$ and $\lambda$ as in (1).
    \end{enumerate}
\end{proposition}

\begin{proof}
Here we follow the arguments from \cite[Proposition 3.12]{JM14}, where the statement (2) has been proved.

To prove (1), we use Demailly's approximation theorem (see \cite{Dem92, AMAG}). Take a pseudoconvex domain $B\subseteq U$ containing $x$ and for every $p\geq 1$, define 
\begin{equation*}
    \varphi_p\coloneqq \frac{1}{p}\sup\left\{\log|f|\colon \int_B|f|^2e^{-2p\varphi}\leq 1\right\}.
\end{equation*}
Then $\varphi_p$ is psh on $B$. By the Ohsawa-Takegoshi $L^2$ extension theorem (\cite{OT87}), there exists a constant $C$ not depending on the choices of $\varphi$ and $p$ such that 
\[\varphi\leq \varphi_p+\frac{C}{p}\]
on $B$. Therefore, for any nonzero ideals $\fk{q}$ and $\fk{q}'$ of $\mathcal{O}_x$, we have 
\begin{equation*}
\begin{split}
    (p\cdot \lct_x(\fk{q},\lambda\cdot \fk{q}';\fk{b}_p))^{-1}&=c_x^{\fk{q},\lambda\cdot \fk{q}'}(\varphi_p)^{-1}\leq c_x^{\fk{q},\lambda\cdot \fk{q}'}(\varphi)^{-1}\\
    &\leq c_x^{\fk{q},\lambda\cdot \fk{q}'}(\varphi_p)^{-1}+\frac{2}{p}\\
    &=(p\cdot \lct_x(\fk{q},\lambda\cdot \fk{q}';\fk{b}_p))^{-1}+\frac{2}{p}
\end{split}
\end{equation*}
for every $p\geq 1$. The two equalities follow from $p\varphi_p=\log|\fk{b}_p|+O(1)$ by taking a common log resolution of $\fk{b}_p,\fk{q},\fk{q}'$.
For the second inequality, put $W=|\fk{q}|^2|\fk{q}'|^{2\lambda}$. The choice of $\varepsilon_0$ gives $W\in L^2_{\mathrm{loc}}$ near $x$. Since $\fk{b}_p=\cal{J}(p\varphi)$, we have $e^{2p(\varphi_p-\varphi)}\in L^1_{\mathrm{loc}}$, hence $We^{p(\varphi_p-\varphi)}\in L^1_{\mathrm{loc}}$ by Cauchy--Schwarz.
For $0<t<c_x^{\fk{q},\lambda\cdot\fk{q}'}(\varphi_p)$, set $1/s=1/t+2/p$. By H\"older's inequality,
\[
    We^{-2s\varphi}=(We^{-2t\varphi_p})^{s/t}(We^{p(\varphi_p-\varphi)})^{2s/p}\in L^1_{\mathrm{loc}}.
\]
Letting $t\uparrow c_x^{\fk{q},\lambda\cdot\fk{q}'}(\varphi_p)$ gives the desired bound. Letting $p$ go to infinity and applying the equation before \Cref{prop-lct.equals.inf.Val.mixed.verb}, we get
\[c_{x}^{\fk{q},\lambda\cdot\fk{q}'}(\varphi)=\lct_x(\fk{q},\lambda\cdot\fk{q}';\fk{b}_{\bullet}).\]

It remains to prove (3). This is given by the following 
\begin{equation*}
    \fk{b}_t\coloneqq\cal{J}(t\varphi)=\cal{J}(t\log|\fk{a}|)=\cal{J}(\log|\fk{a}^t|)=\cal{J}(t\cdot\fk{a}),
\end{equation*}
where the last equation follows from \cite[Proposition 1.7]{DK01}. Finally, we finish our proof by applying \Cref{prop-jm.equa.lct}(1) and \Cref{prop-lct(a.)=lct(b.)}.
\end{proof}

\subsection{The analytic correspondence of algebraic Zhou valuations}

Let $o$ be the origin in $\bfC^n$. Denote by $\cal{O}_o$ the ring of germs of holomorphic functions at $o$, and by $\fk{m}$ the maximal ideal of $\cal{O}_o$. As mentioned above, $\cal{O}_o$ is a Noetherian regular excellent domain of characteristic zero. For the scheme $X=\Spec\cal{O}_o$, the valuations on $X$ can be regarded as the valuations on the ring $\cal{O}_o$. Especially, the following theorem shows that the algebraic Zhou valuations on $\cal{O}_o$ and the analytic Zhou valuations defined on $\cal{O}_o$ in \cite{BGMY23} (see also \cite{BGY23} for a characterization) basically coincide.

\begin{theorem}\label{thm-Zhou.valuation.coincide}
    Let $\fk{q}\subseteq\cal{O}_o$ be a nonzero ideal. Then a valuation $v$ on $\cal{O}_o$ is an algebraic Zhou valuation related to $\fk{q}$ with $v(\fk{m})>0$ if and only if $v$ is an analytic Zhou valuation related to $|\fk{q}|^2$.
\end{theorem}

The algebraic Zhou valuations related to $\fk{q}$ are defined by \Cref{def-Zhou.val}, and the definition of analytic Zhou valuations will be recalled in the following.

\subsubsection{Zhou weights and Zhou valuations}

We first recall the Zhou weights and Zhou valuations introduced in \cite{BGMY23} (with some slight reformulations; see also \cite{BGY23}).

\begin{definition}[{\cite[Definition 1.2, and Definition 1.18]{BGMY23}}]\label{def-analytic Zhou weight}
    Let $\fk{q}\subseteq \cal{O}_o$ be a nonzero ideal. A psh germ $\Phi$ is called a \emph{Zhou weight} related to $|\fk{q}|^2$, if the following statements hold:
    \begin{enumerate}
        \item There exists a constant $N>0$ such that $\Phi\ge N\log|z|+O(1)$ near $o$;
        \item $|\fk{q}|^2e^{-2\Phi}$ is not integrable near $o$;
        \item For any psh germ $\Psi$ satisfying that $|\fk{q}|^2e^{-2\Psi}$ is not integrable near $o$ and $\Psi\ge \Phi+O(1)$ near $o$, we must have $\Psi=\Phi+O(1)$ near $o$.
    \end{enumerate}
\end{definition}

The \emph{cone of Zhou weights} consists of all psh germs $\Phi'$ at $o$ such that there exists some constant $c>0$ and a Zhou weight $\Phi$ related to $|\fk{q}|^2$ for some nonzero ideal $\fk{q}\subseteq\cal{O}_o$, such that $\Phi'=c\Phi+O(1)$ near $o$. In addition, one can also replace a Zhou weight by a \emph{maximal psh germ with the same singular type}, i.e., there exists a psh germ $\wt{\Phi}=\Phi+O(1)$, such that $(\ddc\wt{\Phi})^n\equiv 0$ on $U\setminus\{o\}$, where $U$ is an open neighborhood of $o$; see \cite[Proposition 1.12]{BGMY23}. For a maximal weight $\varphi$ near $o$, the \emph{relative type} introduced by Rashkovskii in \cite{Rash06} is defined as:
\[\sigma(u,\varphi)\coloneqq \sup\{c\ge 0\colon u\le c\varphi+O(1) \text{ near } o\},\]
where $u$ is any psh germ at $o$, and the maximality of $\varphi$ indicates (see \cite[Proposition 3.4]{Rash06})
\[u\le \sigma(u,\varphi)\varphi+O(1) \quad \text{near } o.\]
In particular, the relative types of Zhou weights are called \emph{Zhou numbers} in \cite{BGMY23}. 

Now let $v$ be a valuation on the ring $\cal{O}_o$.

\begin{definition}[\cite{BGMY23}]\label{def-analytic Zhou valuation}
    We call $v$ an \emph{analytic Zhou valuation} related to $|\fk{q}|^2$, where $\fk{q}\subseteq\cal{O}_o$ is a nonzero ideal, if there exists a Zhou weight $\Phi$ related to $\fk{q}$ such that
    \[v(f)=\sigma(\log|f|,\Phi), \quad \forall\,(f,o)\in\cal{O}_o.\]
\end{definition}

Non-trivially, it was shown in \cite[Corollary 1.9]{BGMY23} that every Zhou weight can be associated with a Zhou valuation on $\cal{O}_o$ by establishing an integral expression form of Zhou numbers, where the proof involves a lot of analysis techniques. 


\subsubsection{Algebraic and analytic Zhou valuations}

Now we turn back to prove Theorem \ref{thm-Zhou.valuation.coincide}. Before presenting the proof, let us briefly compare the logical structure of our proof with the proofs in \cite{BGMY23}. The authors of \cite{BGMY23} first introduced the notion of Zhou weight. They then proved that the Zhou numbers (the relative types to the Zhou weight) of the psh germs $u=\log|f|$ satisfy the definition of valuations (see \cite[Corollary 1.7]{BGMY23}), which leads to the notion of Zhou valuations. However, in our article, since we directly define what are Zhou valuations in the valuation space (see \Cref{def-intro Zhou val}), we must prove that, for each (algebraic) Zhou valuation, there exists a psh weight, which should be a Zhou weight, such that the relative type to the weight is compatible with the (algebraic) Zhou valuation. Although the construction of the weight in our proof is very similar with the constructions in \cite{BGMY23, BGY23}, but the proofs in \cite{BGMY23, BGY23} strongly depend on the fact that, \emph{Zhou weights admit Zhou valuations}, which will not be assumed in our proof.

In addition, we will avoid using analytical methods as much as possible in the proof. For example, the strong openness property of multiplier ideal sheaves (\cite{GZ15soc}; see also \cite{Hiep14,Lem17}) will only be used to deduce ``$|\fk{q}|^2e^{-2\Phi}$ is not integrable near $o$" from ``$c_o^{\fk{q}}(\Phi)= 1$". Meanwhile, a \emph{lower semi-continuity property of jumping numbers} (see \cite[Proposition 1.8]{GZ15eff}; see also \cite{Hiep14}) is needed, where a special case, that is the lower semi-continuity property of complex singularity exponents, was proved in \cite[Main Theorem 0.2]{DK01} by Demailly's approximation and the Ohsawa--Takegoshi $L^2$ extension theorem.

We start from a general valuation on $\Spec\cal{O}_o$. We will give a maximal psh weight associated with the valuation and establish some properties of the weight.

Denote the unit ball centered at $o$ in $\bfC^n$ by $\bfB(o,1)$.

\begin{lemma}\label{lem-valuation.corre.psh.general}
    Let $v$ be a valuation on $\cal{O}_o$ with $v(\fk{m})>0$ and $A(v)<+\infty$. For every $z \in \bfB(o,1)$, set
    \begin{equation}\label{eq-Phi.v.star}
        \Phi_{v}(z)\coloneqq\sup\Big\{\frac{\log|f(z)|}{v(f)}\colon f\in\cal{O}\big(\bfB(o,1)\big), \, f(o)= 0, \, f\not\equiv 0, \, \sup_{\bfB(o,1)}|f|\le 1\Big\}.
    \end{equation}
    Then the upper semi-continuous regularization $\Phi^{\star}_{v}$ of $\Phi_v$ is a negative psh function on $\bfB(o,1)$ satisfying the following statements:
    \begin{enumerate}
        \item There exist $C_1,C_2>0$ such that $C_1\log|z|\le \Phi_{v}^{\star}\le C_2\log|z|$ near $o$.
        \item The psh function $\Phi_v^{\star}$ is locally bounded and maximal on $\bfB(o,1)\setminus\{o\}$, and $e^{\Phi_v^{\star}}$ is continuous on $\bfB(o,1)$. 
        \item If $\fk{q},\fk{q}'\subseteq\cal{O}_o$ are nonzero ideals, and $\lambda\in(-\varepsilon_0,0]\subseteq\bfR$ with $\varepsilon_0>0$ sufficiently close to $0$, then
        \[\lct(\fk{q},\lambda\cdot\fk{q}';\fk{a}^v_{\bullet})=c_o^{\fk{q},\lambda\cdot\fk{q}'}(\Phi_v^{\star}).\]

        Especially, $\lct^{\fk{q}}(\fk{a}_{\bullet}^v)=c_o^{\fk{q}}(\Phi_v^{\star})$ if $\lambda=0$.
        \item If $\fk{q},\fk{q}'\subseteq\cal{O}_o$ are nonzero ideals, then the concave function (see Lemma \ref{lem-lct.concave})
        \[\T(t)=\lct(\fk{q}, t\cdot \fk{q}' ; \fk{a}^v_{\bullet}), \quad t\in (-\varepsilon_0, +\infty)\]
        satisfies that $\T'_-(0)\ge \sigma(\log|\fk{q}'|,\Phi_v^{\star})$, where $\T'_-(0)$ denotes the left derivative of $\T$ at $0$, and $\sigma(\,\cdot\,, \Phi_v^{\star})$ is the relative type to $\Phi_v^{\star}$.
    \end{enumerate}
\end{lemma}

\begin{proof}
    Clearly $\Phi_v^{\star}$ is a negative psh function on $\bfB(o,1)$, and $\Phi_v^{\star}$ a.e. equals to $\Phi_v$ (cf. \cite[Proposition 4.24]{CADG}). Denote 
    \[\cal{H}\coloneqq \Big\{f\in\cal{O}\big(\bfB(o,1)\big)\colon f(o)= 0, \, f\not\equiv 0, \, \sup_{\bfB(o,1)}|f|\le 1\Big\}.\]
    
        (1) Since $A(v)<+\infty$ and $v(\fk{m})>0$, according to the Izumi's inequality (ref. \cite{Izu85} or \cite[Proposition 5.10]{JM12}), we have
        \[v(\fk{m})\ord_o(f)\leq v(f)\leq A(v)\ord_o(f)\]
        for every $f\in \mathcal{H}$.
        On the one hand, we have
        \[\Phi^{\star}_v(z)\ge \Phi_v(z)\ge\max_{1\le i\le n}\frac{\log|z_i|}{v(z_i)}\ge\frac{1}{v(\fk{m})} \max_{1\le i\le n}\log|z_i| \ge C_1\log|z|\]
        for some constant $C_1>0$, where $z=(z_1,\cdots,z_n)$.
        On the other hand, for every $f\in\cal{H}$,
        \[\frac{\log|f(z)|}{v(f)}\le\frac{\ord_o(f)\log|z|}{v(f)}\le\frac{\ord_o(f)\log|z|}{A(v)\ord_o(f)}=\frac{1}{A(v)}\log|z|,\]
        which implies $\Phi_v\le C_2\log|z|$ with $C_2\coloneqq 1/A(v)$ and thus $\Phi_v^{\star}\le C_2\log|z|$. The first statement holds.
        
        (2) It follows from (1) that $\Phi_v^{\star}$ is locally bounded outside $o$ and $e^{\Phi_v^{\star}}$ is continuous. The standard arguments give the maximality of $\Phi_v^{\star}$.
        
         (3) Since $A(v)<+\infty$, we have $\lct(\fk{q},\lambda\cdot\fk{q}';\fk{a}_{\bullet}^v)<+\infty$. By the coherence of ideal sheaves on $\bfC^n$, the germs of functions in $\cal{H}$ at $o$ generate every nonzero ideal of $\cal{O}_o$. Thus, we have
        \begin{equation}\label{eq-Phi.v.star.ge.q/m.log.amv}
            \Phi_v^{\star}\ge \frac{1}{m}\log|\fk{a}^v_m|+O(1) \quad \text{near } o,
        \end{equation}
        for every $m\ge 1$. Since 
        \begin{equation}\label{eq-m.lct.equal.coqlambdaq'}
            m\cdot\lct(\fk{q},\lambda\cdot\fk{q}';\fk{a}_{m}^v)=m\cdot c_o^{\fk{q},\lambda\cdot\fk{q}'}\left(\log|\fk{a}_m^v|\right)=c_o^{\fk{q},\lambda\cdot\fk{q}'}\left(\frac{1}{m}\log|\fk{a}_m^v|\right)
        \end{equation}
        by \Cref{prop-jm.equa.lct}(3), letting $m\to +\infty$, it follows from (\ref{eq-Phi.v.star.ge.q/m.log.amv}) that 
        \[\lct(\fk{q},\lambda\cdot\fk{q}';\fk{a}_{\bullet}^v)\le c_o^{\fk{q},\lambda\cdot\fk{q}'}(\Phi_v^{\star}).\]
        
        On the other hand, Choquet's Lemma (cf. \cite[Proposition 4.23]{CADG}) shows that there exists a sequence $\phi_j=\log|f_j|/v(f_j)$ with $f_j\in\cal{H}$ such that the $\Phi_v^{\star}=(\sup_j \phi_j)^{\star}$. Since $f^k\in \cal{H}$ when $f\in \cal{H}$ for $k\in\bfZ_+$ and $v(f^k)=kv(f)$, after replacing $f_j$ with some power of $f_j$, we may assume that 
        \begin{enumerate}
            \item $v(f_j)>1$ for all $j\geq 1$;
            \item $v(f_{j+1})>v(f_j)$ for all $j$ and $v(f_j)\to +\infty$ when $j\to +\infty$;
            \item  $v(f_j)/v(f_{j-1})\to +\infty$ when $j\to +\infty$.
        \end{enumerate}
        Let $\varphi_j(z)\coloneqq\max\limits_{1\le k\le j}\phi_k(z)$. Then $\varphi_j$ increasingly a.e. converges to $\Phi_{v}^{\star}$, which implies $\sup\limits_{j\ge 1}c_o^{\fk{q},\lambda\fk{q}'}(\varphi_j)=c_o^{\fk{q},\lambda\cdot\fk{q}'}(\Phi_v^{\star})$ when $\lambda\le 0$, due to the lower semi-continuity property of jumping numbers (see \cite[Proposition 1.8]{GZ15eff}).
        
        Denote
        \[m_j=\lceil v(f_j)+v(f_{j-1})\rceil, \quad m'_j=\lceil v(f_j)\rceil-1, \quad \forall\, j\ge 1,\]
        and
        \[p_{k,j}=\left\lceil \frac{m'_j}{v(f_k)}\right\rceil, \quad k=1,\ldots, j,\]
        which is a positive integer such that $g_{k,j}\coloneqq f_k^{p_{k,j}}\in\fk{a}_{m'_j}^v$ for each $k\leq j$. Observe that when $1\le k\le j-1$, 
        \begin{equation*}
            \begin{aligned}
                v(g_{k,j})=p_{k,j}v(f_k)&\le\left(\frac{m'_j}{v(f_k)}+1\right)v(f_k)\\
                &=m'_j+v(f_k)\le v(f_j)+v(f_{j-1})\le m_j,
            \end{aligned}
        \end{equation*}
        and $v(g_{j,j})=v(f_j)\le m_j$. It follows that
        \begin{equation*}
            \begin{aligned}
                \varphi_j&=\max_{1\le k\le j}\frac{\log|f_k|}{v(f_k)}=\max_{1\le k\le j}\frac{\log|g_{k,j}|}{v(g_{k,j})}\\
                &\le\frac{1}{m_j}\max_{1\le k\le j}\log|g_{k,j}|\le \frac{1}{m_j}\log|\fk{a}_{m_j'}^v|+O(1)\\
                &=\frac{m'_j}{m_j}\cdot \left(\frac{1}{m'_j}\log|\fk{a}_{m'_j}^v|+O(1)\right) \quad \text{near } o,
            \end{aligned}
        \end{equation*}
        for every $j\ge 1$. Hence, since $\frac{m'_j}{m_{j}}\to 1$ as $j\to +\infty$, according to (\ref{eq-m.lct.equal.coqlambdaq'}), we obtain when $\lambda\in (-\varepsilon_0,0]$,
        \[\lct(\fk{q},\lambda\cdot\fk{q}';\fk{a}_{\bullet}^v)=\sup_{j\ge 1}c_o^{\fk{q},\lambda\cdot\fk{q}'}\left(\frac{1}{m'_j}\log|\fk{a}_{m'_j}^v|\right)\ge \sup_{j\ge 1}c_o^{\fk{q},\lambda\fk{q}'}(\varphi_j)=c_o^{\fk{q},\lambda\cdot\fk{q}'}(\Phi_v^{\star}).\]
        Eventually, we get that the third statement holds.

        (4) By Statement (1), the relative type
        \[\sigma\coloneqq \sigma(\log|\fk{q}'|, \Phi_v^{\star})\in (0,+\infty).\]
        Since $\Phi_v^{\star}$ is maximal, we have
        \[\log|\fk{q}'|\le \sigma\Phi_v^{\star}+O(1) \quad \text{near } o.\]
        For $\mu>0$ sufficiently small and any $c>0$, it follows that
        \[\int_{V} |\fk{q}|^2|\fk{q}'|^{-2\mu}e^{-2c\Phi_v^{\star}}=\int_V|\fk{q}|^2e^{-2\mu\log|\fk{q}'|-2c\Phi_v^{\star}}\ge C_0\int_V|\fk{q}|^2e^{-2(c+\mu\sigma)\Phi_v^{\star}},\]
        for any neighborhood $V$ of $o$ and $C_0>0$. Thus,
        \[c_o^{\fk{q}}(\Phi_v^{\star})-\mu\sigma\ge c_o^{\fk{q},-\mu\cdot\fk{q}'}(\Phi_v^{\star}),\]
        which implies
        \[\T(-\mu)=\lct(\fk{q},-\mu\cdot\fk{q}';\fk{a}_{\bullet}^v)=c_o^{\fk{q},-\mu\cdot\fk{q}'}(\Phi_v^{\star})\le c_o^{\fk{q}}(\Phi_v^{\star})-\mu\sigma= \T(0)-\mu\sigma\]
        by applying Statement (3). According to the concavity and monotonicity of the function $\T(t)$, we can verify that $\T'_-(0)$ exists and $\T'_-(0)\ge \sigma$.

    The proof is complete.
\end{proof}

For Zhou valuations on $\cal{O}_o$, we can prove that the function $\Phi_v^{\star}$ satisfies more properties as demonstrated in the following proposition.

\begin{proposition}\label{prop-alg.ZVal.to.Zhou.weight}
Let $\fk{q}\subseteq\cal{O}_o$ be a nonzero ideal, and $v$ be a valuation on $\Spec \cal{O}_o$. If $v$ is an algebraic Zhou valuation related to $\fk{q}$ with $v(\fk{m})>0$, then the maximal psh function $\Phi^{\star}_v$ associated to $v$ defined as which in Lemma \ref{lem-valuation.corre.psh.general} on the unit ball $\bfB(o,1)$ additionally satisfies the following statements:
\begin{enumerate}
    \item The jumping number $c_o^{\fk{q}}(\Phi_v^{\star})=1$.
    \item The relative type of $\Phi_v^{\star}$ is \emph{compatible} with $v$, i.e.,
    \[\sigma(\log|\fk{q}'|,\Phi_v^{\star})=v(\fk{q}'), \quad \text{for any ideal } \fk{q}'\subseteq\cal{O}_o.\]
    \item For every psh germ $\psi$ at $o$ with $c_o^{\fk{q}}(\psi)=1$ and $\psi\ge \Phi_v^{\star}+O(1)$ near $o$, it must hold that $\psi=\Phi_v^{\star}+O(1)$ near $o$.
\end{enumerate}
\end{proposition}

\begin{proof}
    The statement (3) in Lemma \ref{lem-valuation.corre.psh.general} and $\lct^{\fk{q}}(\fk{a}_{\bullet}^v)=1$ given by \Cref{prop-lct(Zhou val)=1} indicate that $c_o^{\fk{q}}(\Phi_v^{\star})=\lct^{\fk{q}}(\fk{a}_{\bullet}^v)=1$. 

    For the function $\T(t)=\lct(\fk{q},t\cdot\fk{q}';\fk{a}_{\bullet}^v)$:
    \begin{enumerate}
        \item[(a)] The Statement (4) in Lemma \ref{lem-valuation.corre.psh.general} shows $\T'_-(0)\ge\sigma(\log|\fk{q}',\Phi_v^{\star})$;
        \item[(b)] Since $v$ is a Zhou valuation related to $\fk{q}$, Theorem \ref{thm-ZV.lct.differential} implies $\T'_-(0)=\T'_+(0)=v(\fk{q}')$;
        \item[(c)] The definition of $\Phi_v^{\star}$ gives $\sigma(\log|\fk{q}'|,\Phi_v^{\star})\ge v(\fk{q}')$.
    \end{enumerate}
    Thus, $\sigma(\log|\fk{q}'|,\Phi_v^{\star})=v(\fk{q}')$ for every nonzero ideal $\fk{q}'\subseteq\cal{O}_o$.

    Now we prove the third statement. Suppose that the psh germ $\psi$ at $o$ satisfies $c_o^{\fk{q}}(\psi)=1$ and $\psi\ge \Phi_v^{\star}+O(1)$ near $o$. Denote
    \[\fk{b}_t=\cal{J}(t\psi)_o, \quad t\in (0,+\infty).\]
    Note that $\fk{b}_{\bullet}=\{\fk{b}_t\}_t$ is a subadditive system of ideals in $\cal{O}_o$ (cf. \cite{DEL00}). According to \Cref{prop-jm.equa.lct}, we have $\lct^{\fk{q}}(\fk{b}_{\bullet})=c_o^{\fk{q}}(\psi)=1$, and $\fk{b}_{\bullet}$ has controlled growth.

    Since $\psi\ge\Phi_v^{\star}+O(1)$ near $o$, we can verify that $\fk{b}_j\supseteq\fk{a}^v_j$ for every $j\in\bfZ_+$. In fact, by (\ref{eq-Phi.v.star.ge.q/m.log.amv}), for every $N\in \bfZ_+$,
    \[\psi\ge\frac{1}{N}\log|\fk{a}_N^v|+O(1) \quad \text{near } o,\]
    which implies
    \[\fk{b}_j=\cal{J}(j\psi)_o\supseteq\cal{J}\left(\frac{j}{N}\log|\fk{a}_N^v|\right)_o=\cal{J}\left(\frac{j}{N}\cdot\fk{a}_N^v\right)=\cal{J}(j\cdot\fk{a}_{\bullet}^v)\]
   by \Cref{prop-jm.equa.lct}(3), where $N\gg 1$ is divisible enough. Especially, we have $\fk{a}_j^v\subseteq\cal{J}(j\cdot\fk{a}_{\bullet}^v)\subseteq \fk{b}_j$.

    Since $v(\fk{m})>0$ and $\fk{a}_j^v\subseteq\fk{b}_j$ for every $j$, there exists some positive integer $p$ such that $\fk{m}^{pj}\subseteq \fk{b}_j$ for every $j\in\bfZ_+$. Now according to \cite[Proposition 2.1]{JM14}, we can find a valuation $w$ on $\Spec \cal{O}_o$ with $A(w)<+\infty$ which computes $\lct^{\fk{q}}(\fk{b}_{\bullet})$. After rescaling, we can assume that $w(\fk{b}_{\bullet})=1$. Then it follows from \Cref{lem-JM12 section6.2} that
    \[\lct^{\fk{q}}(\fk{a}_{\bullet}^{w})\le A(w)+w(\fk{q})=\lct^{\fk{q}}(\fk{b}_{\bullet})=1.\]
    Since $\fk{a}_j^v\subseteq\fk{b}_j$ for every $j$, we have $w(\fk{a}_{\bullet}^v)\ge w(\fk{b}_{\bullet})=1$, which implies $w\ge v$ by \Cref{lem-w(a.bullet.v)=inf}. Now, according to the assumption that $v$ is an algebraic Zhou valuation related to $\fk{q}$, we get $w=v$, which indicates that $v(\fk{b}_{\bullet})=w(\fk{b}_{\bullet})=1$. Then it follows from \cite[Proposition 6.5]{JM12} and \Cref{prop-jm.equa.lct}(2) that
    \[\frac{v(\fk{b}_{t})}{t}\ge v(\fk{b}_{\bullet})-\frac{A(v)}{t}\ge 1-\frac{1}{t}, \quad \forall t>0,\]
    where $A(v)=1-v(\fk{q})\le 1$ by \Cref{thm-A(v)=1-v(q)}. It implies that $$\sigma(\log|\fk{b}_{j}|,\Phi_v^{\star})=v(\fk{b}_j)\ge j-1$$ when $j\in\bfZ_{\ge 2}$ by the statement (2), i.e.,
    \[\frac{1}{j}\log\big|\cal{J}(j\psi)_o\big|=\frac{1}{j}\log|\fk{b}_j|\le\frac{j-1}{j}\Phi_v^{\star}+O(1) \quad \text{near } o,\]
    for every $j\gg 1$. Finally, we apply Demailly's approximation theorem (or the Ohsawa-Takegoshi $L^2$ extension theorem) to $\psi$ to get
    \[\psi\le \frac{1}{j}\log\big|\cal{J}(j\psi)_o\big|+O(1)\le\frac{j-1}{j}\Phi_v^{\star}+O(1) \quad \text{near } o,\]
    which implies
    \[\sigma(\psi,\Phi_v^{\star})\ge \frac{j-1}{j}\to 1, \quad j\to+\infty.\]
    Hence, $\psi\le\Phi_v^{\star}+O(1)$ near $o$, and thus $\psi=\Phi_v^{\star}+O(1)$ near $o$.

    The proof is complete.
\end{proof}

Now we prove Theorem \ref{thm-Zhou.valuation.coincide}.

\begin{proof}[Proof of Theorem \ref{thm-Zhou.valuation.coincide}]
    We first assume that $v$ is an algebraic Zhou valuation on $\Spec\cal{O}_o$ with $v(\fk{m})>0$. Then \Cref{lem-valuation.corre.psh.general} and Proposition \ref{prop-alg.ZVal.to.Zhou.weight} show that there exists a maximal psh germ $\Phi$ at $o$ such that the following statements hold:
    \begin{enumerate}
        \item[(a)] There exist positive constants $C_1, C_2$ such that $C_1\log|z|\le\Phi\le C_2\log|z|$ near $o$.
        \item[(b)] $c_o^{\fk{q}}(\Phi)=1$.
        \item[(c)] For every psh germ $\psi$ at $o$ with $c_o^{\fk{q}}(\psi)=1$ and $\psi\ge\Phi+O(1)$ near $o$, it must be $\psi=\Phi+O(1)$ near $o$.
        \item[(d)] $\sigma(\log|\fk{q}'|,\Phi)=v(\fk{q}')$ for every ideal $\fk{q}'\subseteq\cal{O}_o$.
    \end{enumerate} 
    The statement (b) and the strong openness property (\cite{GZ15soc}) show that $|\fk{q}|^2e^{-2\Phi}$ is not integrable near $o$. Now we can see that the statements (a) (b) and (c) actually tell that $\Phi$ is a Zhou weight related to $|\fk{q}|^2$ near $o$; see \Cref{def-analytic Zhou weight}. Then it follows from the statement (d) that, restricted to $\cal{O}_o$, $v$ is an analytic Zhou valuation related to $|\fk{q}|^2$ by \Cref{def-analytic Zhou valuation}.

    Next, we prove the converse side. In fact, the proof is quite similar with the proof of the statement (3) of \Cref{prop-alg.ZVal.to.Zhou.weight}.
    
    Let $v$ be a valuation on $\Spec \cal{O}_o$, such that $v$ is an analytic Zhou valuation related to $|\fk{q}|^2$. By \Cref{def-analytic Zhou weight} and \Cref{def-analytic Zhou valuation} again, we assume that $\Psi$ is a local Zhou weight (and maximal) associated to the valuation $v$, i.e.,
    \begin{equation}\label{eq-Psi.relative type}
        \sigma(\log|f|,\Psi)=v(f), \quad \forall (f,o)\in\cal{O}_o.
    \end{equation}

    Set
    \[\fk{b}_t=\cal{J}(t\Psi)_o, \quad t\in (0,+\infty).\]
    Then $\fk{b}_{\bullet}=\{\fk{b}_t\}_t$ is a subadditive system of ideals in $\cal{O}_o$ with controlled growth and $\lct^{\fk{q}}(\fk{b}_{\bullet})=c_o^{\fk{q}}(\Psi)=1$, according to \Cref{prop-jm.equa.lct}. Since $\Psi\ge C\log|z|+O(1)$ near $o$ for $C\coloneqq 1/\sigma(\log|z|,\Psi)=1/v(\fk{m})>0$, we can see that there exists a positive integer $p$ such that $\fk{m}^{pj}\subseteq \fk{b}_j$ for every $j>0$. Then it follows from \cite[Proposition~2.1]{JM14} again that there exists a valuation $\tilde{v}$ on $\Spec \cal{O}_o$ with $A(\tilde{v})<+\infty$ which computes $\lct^{\fk{q}}(\fk{b}_{\bullet})$. After rescaling, we let $\tilde{v}(\fk{b}_{\bullet})=1$. Then $\lct^{\fk{q}}(\fk{a}_{\bullet}^{\tilde{v}})\le A(\tilde{v})+\tilde{v}(\fk{q})=\lct^{\fk{q}}(\fk{b}_{\bullet})=1$ and thus $\tilde{v}\in \Val(X;\fk{q})$; see \Cref{lem-JM12 section6.2} and \Cref{lem-Val(X;q)}.
    
    Let $w\ge \tilde{v}$ be an algebraic Zhou valuation on $\Spec\cal{O}_o$ related to $\fk{q}$, and then it suffices to show $w=v$. Define the function $\Phi_w^{\star}$ associated with $w$ by (\ref{eq-Phi.v.star}), where
    \[\Phi_{w}(z)\coloneqq\sup\Big\{\frac{\log|f(z)|}{w(f)}\colon f\in\cal{O}\big(\bfB(o,1)\big), \, f(o)= 0, \, f\not\equiv 0, \, \sup_{\bfB(o,1)}|f|\le 1\Big\}.\]
    Then $c_o^{\fk{q}}(\Phi_w^{\star})=1$ by \Cref{prop-alg.ZVal.to.Zhou.weight}(1). According to \Cref{prop-alg.ZVal.to.Zhou.weight}(2) and (\ref{eq-Psi.relative type}), to prove $w=v$, we only need to prove $\Psi=\Phi_w^{\star}+O(1)$ near $o$. Because $\Psi$ is a Zhou weight related to $|\fk{q}|^2$, it suffices to show $\Phi_w^{\star}\ge \Psi+O(1)$ near $o$.

    Since $w\ge \tilde{v}$ and $\tilde{v}(\fk{b}_{\bullet})=1$, we have $w(\fk{b}_{\bullet})\ge 1$. Then it follows from \cite[Proposition 6.5]{JM12} (or \Cref{lem-relation between a and b} (1)) and \Cref{prop-jm.equa.lct}(2) that
    \[\frac{w(\fk{b}_{t})}{t}\ge w(\fk{b}_{\bullet})-\frac{A(w)}{t}\ge 1-\frac{1}{t}, \quad \forall t>0,\]
    as $A(w)=1-w(\fk{q})\le 1$. Using \Cref{prop-alg.ZVal.to.Zhou.weight} and Demailly's approximation theorem, we get
    \[\Phi^{\star}_w\ge \frac{1}{w(\fk{b}_j)}\log|\fk{b}_j|+O(1)\ge \frac{j}{j-1}\Psi+O(1) \quad \text{near } o,\]
    for every $j>1$. This implies that the relative type
    \[\sigma(\Psi,\Phi_w^{\star})\ge 1-\frac{1}{j} \to 1 \quad \text{as } j\to+\infty.\]
    Consequently, we get $\Psi\le\Phi_w^{\star}+O(1)$ near $o$, and the proof is complete.
\end{proof}

We explain in the following corollaries and examples that how we can prove some other results in \cite{BGMY23} which are not mentioned above by algebraic methods. 

Corollary \ref{cor-Zhou.weight.to.Zhou.val} is a very nontrivial result of the Zhou weights, established in \cite{BGMY23} by giving an integral expression for the Zhou numbers (cf. \cite[Theorem 1.6]{BGMY23}), which highly relies on the analytic methods, and we can reprove it using the same arguments in the proofs of \Cref{prop-alg.ZVal.to.Zhou.weight} and Theorem \ref{thm-Zhou.valuation.coincide}.

\begin{corollary}[{\cite[Corollary 1.9]{BGMY23}}]\label{cor-Zhou.weight.to.Zhou.val}
    {If $\Phi$ is a Zhou weight related to some nonzero ideal $\fk{q}\subseteq\cal{O}_o$, then there exists a valuation $v$ on $\cal{O}_o$ such that
\[v(\fk{q}')=\sigma(\log|\fk{q}'|, \Phi), \quad \text{$\forall$ nonzero ideal } \fk{q}'\subseteq\cal{O}_o.\]}
\end{corollary}
\begin{proof}
   Let $\fk{b}_{\bullet}=\{\fk{b}_{t}\}_{t>0}=\{\cal{J}(t\Phi)_o\}_{t>0}$ be the subadditive system associated with the Zhou weight $\Phi$. Note that there exists $N>0$ such that $\Phi\ge N\log|z|+O(1)$ near $o$. Then using \cite[Proposition~2.1]{JM14}, we can find a valuation $w$ on $\cal{O}_o$ which computes $\lct^{\fk{q}}(\fk{b}_{\bullet})$, and $w(\fk{b}_{\bullet})=1$. Let $v\ge w$ be a Zhou valuation on $\cal{O}_o$ related to $\fk{q}$, and $\Phi_v^{\star}$ be the Zhou weight associated with $v$ related to $|\fk{q}|^2$. With the same arguments used in the proofs of \Cref{prop-alg.ZVal.to.Zhou.weight} (3) and also Theorem \ref{thm-Zhou.valuation.coincide}, we can get $\Phi_v^{\star}\ge \Phi+O(1)$, and thus $\Phi_v^{\star}=\Phi+O(1)$ since both $\Phi_v^{\star}$ and $\Phi$ are Zhou weights related to $|\fk{q}|^2$. Therefore, we have that $v(\fk{q}')=\sigma(\log|\fk{q}'|, \Phi)$  for every nonzero ideal $\fk{q}'\subseteq \cal{O}_o$. 
\end{proof}

\begin{corollary}\label{cor-Zhou.weight.tame}
{All Zhou weights are tame.}
\end{corollary}

\begin{proof}
This is a consequence of \cite[Theorem 1.10]{BGMY23}, which was proved by analytic approach there. We give an algebraic proof in the following.

Recall that a psh weight $\Phi$ is said to be \emph{tame}, if there exists a constant $C>0$ such that for every $t>0$ and every $(f,o)\in\cal{I}(t\varphi)_o$, $\log|f|\le (t-C)\Phi+O(1)$ holds near $o$. In fact, observing that for a Zhou valuation $v$ related to a nonzero ideal $\fk{q}$ and the associated Zhou weight $\Phi_v^{\star}$,
\begin{equation*}
    \begin{aligned}
        c_o^f(\Phi_v^{\star})&=\lct^{(f)}(\fk{a}_{\bullet}^v)\le\frac{A(v)+v(f)}{v(\fk{a}_{\bullet}^v)}=1-v(\fk{q})+v(f)\\
        &=1-v(\fk{q})+\sigma(\log|f|,\Phi_v^{\star}),
    \end{aligned}
\end{equation*}
we can thus directly conclude that the Zhou weights are tame using \Cref{thm-Zhou.valuation.coincide}, \Cref{cor-Zhou.weight.to.Zhou.val} and the fact that $\Phi_v^{\star}$ is maximal.
\end{proof}

\begin{example}[cf. {\cite[Example 1.5 \& 1.20]{BGMY23}}]\label{ex-Zhou.weight}
Let $(k_1,\ldots,k_n)\in\bfZ_+^n$, and $\Phi(z)\coloneqq\max_{1\le i\le n}\frac{1}{\alpha_i}\log|z_i|$ be a psh function on $\Delta^n\subseteq \bfC^n$, where $\alpha=(\alpha_1,\ldots,\alpha_n)\in\bfR_{>0}^n$ satisfying $\sum_{i=1}^n{\alpha_i}k_i=1$. Then $\Phi$ is a Zhou weight related to $|\fk{q}|^2$ near $o$ for $\fk{q}=(z_1^{k_1-1}\cdots z_n^{k_n-1})$.

\begin{proof}
The proof of a special case of this example in \cite{BGMY23} relies on the \emph{concavity property of minimal $L^2$ integrals} (cf. \cite{Guan19,Guan20,GY24}), which somehow is a deep analytic result. We give a mainly algebraic proof here.

Let $v_{\alpha}$ be the monomial valuation defined on $\cal{O}_o$ by
\[v_{\alpha}(f)\coloneqq \min\{\langle \alpha, \beta\rangle \colon c_{\beta}\neq 0\}, \quad \forall\,f(z)=\sum\limits_{\beta\in\bfZ^n_{\ge 0}}c_{\beta}z^{\beta}\in \cal{O}_{\mathbf{C}^n, o}.\]
Then $v_{\alpha}$ is a quasi-monomial adapted to $(X,H)=(\Spec\cal{O}_o,H_1+\cdots+H_n)$ with $H_i=V(z_i)$, and $A(v_{\alpha})=\sum_{i=1}^n\alpha_i$ (cf. \cite[Section 8]{JM14}). 
Set
\[\fk{a}\coloneqq (z_1^{k_1}\cdots z_n^{k_n}).\]
We have $\lct^{\fk{q}}(\fk{a})=1$. In fact, on the one hand, we have 
\[\lct^{\fk{q}}(\fk{a})\leq \frac{A(v_{\alpha})+v_{\alpha}(\fk{q})}{v_{\alpha}(\fk{a})}=\frac{\sum_{i=1}^n\alpha_i+\sum_{i=1}^n\alpha_i(k_i-1)}{\sum_{i=1}^n\alpha_ik_i}=1.\]
On the other hand, we have \[\lct^{\fk{q}}(\fk{a})=\inf_{w\in \Val_X^*}\frac{A(w)+w(\fk{q})}{w(\fk{a})}\]
by \Cref{lem-JM12 section6.2}. Consider the retraction map $r_{X,H}\colon \Val_X\to \QM(X,H)$. Then $\bar{w}\coloneqq r_{X,H}(w)$ is a monomial valuation. We have $w(\fk{q})=\bar{w}(\fk{q})$, $w(\fk{a})=\bar{w}(\fk{a})$ and $A(w)\ge A(\bar{w})$, which implies
\[\frac{A(w)+w(\fk{q})}{w(\fk{a})}\ge\frac{A(\bar{w})+\bar{w}(\fk{q})}{\bar{w}(\fk{a})}=\frac{\sum_{i=1}^n\bar{w}_i+\sum_{i=1}^n\bar{w}_i(k_i-1)}{\sum_{i=1}^n\bar{w}_ik_i}=1,\]
and thus $\lct^{\fk{q}}(\fk{a})\ge 1$, where $\bar{w}_i=\bar{w}(z_i)$.

Now we can conclude that $\lct^{\fk{q}}(\fk{a})$ is computed by the valuation $v_{\alpha}$, and $v_{\alpha}(\fk{a})=\sum_{i=1}^n\alpha_ik_i=1$. Therefore, by \Cref{ex-Zhou.val}, $v_{\alpha}$ is a Zhou valuation related to $\fk{q}$, and it is easy to check that $\Phi$ is a psh weight whose relative type (which actually is the Kiselman number; see \cite{Kis94}) is compatible with $v_{\alpha}$. Hence, $\Phi$ is a Zhou weight related to $|\fk{q}|^2$ near $o$.
\end{proof}
\end{example}

There are also several results obtained in \cite{BGY23} based on the notion of ``$\scr{A}$'' introduced there. According to \Cref{thm-A(v)=1-v(q)} and \Cref{thm-Zhou.valuation.coincide}, it is easy to see that the notion ``$\scr{A}(v)$'' defined in \cite{BGY23} in fact coincides with $A(v)$, the log-discrepancy, for all (analytic) Zhou valuations $v$ on $\cal{O}_o$. Consequently, we can replace all the ``$\scr{A}$'' by ``$A$'' in \cite{BGY23}, which will provide some corollaries, but we do not demonstrate for detail in the present paper.

\subsection*{Acknowledgments} 
The authors would like to thank Mattias Jonsson, Zhitong Mi, Xun Sun, Zheng Yuan and Xiangyu Zhou for helpful discussions. The second named author was supported by National Key R$\&$D Program of China 2021YFA1003100, NSFC-11825101 and NSFC-12425101. The third named author was supported by the European Union’s Horizon 2020 research and innovation programme under grant agreement No. 948066 (ERC-StG RationAlgic).

\appendix \section{The two-dimensional case}

After an early preprint version of the present paper was released, Jonsson asked the authors the following question in a private communication:

\begin{question}
    Can one prove that every quasi-monomial valuation is a Zhou valuation related to some ideal $\fk{q}$?
\end{question}

Until now, we have not been able to answer this question in general. However, by using the valuative tree theory developed by Favre and Jonsson (cf. \cite{FJ04,FJ05a,FJ05b}), we can show that every nontrivial quasi-monomial valuation on a \emph{$2$-dim} regular excellent scheme $X$ over $\bfQ$ is a Zhou valuation related to some ideal $\fk{q}$ (up to rescaling), i.e., $\ZVal_X=\QM_X\setminus\{\text{trivial valuation}\}$ for $\dim X\le 2$ (see \Cref{cor-dim.le.2}).

Let $v$ be a nontrivial quasi-monomial valuation on $X$. If $\xi=c_X(v)$ has codimension one, then $v=\lambda\ord_D$ for some $\lambda>0$, where $D=\overline{\{\xi\}}_{\mathrm{red}}$. Let $Z$ be the reduced singular locus of $D$, and take a log resolution of $\mathcal I_D\mathcal I_Z$ that is an isomorphism over $X\setminus Z$. Every exceptional divisor $F$ on this resolution satisfies $\ord_F(\mathcal I_Z)>0$. Thus, for $N\gg1$ and $\fk{q}=\mathcal I_Z^N$, we have
\[A(\ord_F)+\ord_F(\fk{q})>\ord_F(\mathcal I_D)\]
for every such $F$, while $A(\ord_D)+\ord_D(\fk{q})=\ord_D(\mathcal I_D)=1$. By \cite[Lemma 6.7]{JM12}, only positive multiples of $\ord_D$ compute $\lct^{\fk{q}}(\mathcal I_D)=1$. Hence $v\in\ZVal_X$ by \Cref{cor-ZV.computes} applied to the powers of $\mathcal I_D$.

It remains to consider a center $\xi$ of codimension two, which is a closed point. According to \cite[Proposition 7.12 \& Lemma 3.10]{JM12} and Cohen's structure theorem, we may assume $X=\Spec \widehat{\mathcal{O}_{X,\xi}}=\Spec k[[x,y]]$ for $\xi=(0,0)$ and a field $k$ of characteristic $0$. Furthermore, by \cite[Proposition 7.13 \& Lemma~3.11]{JM12}, it suffices to consider the case that $k$ is algebraically closed. After rescaling, we can normalize $v$ by $v(\fk{m})=1$, where $\fk{m}=(x,y)$ is the maximal ideal defining the closed point $(0,0)$.

Recall that the set consisting of such normalized centered quasi-monomial valuations on $R=k[[x,y]]$, denoted by $\V_{\qm}$, is a subtree of the complete tree $\V$ of normalized centered valuations (may take values in $\mathbf{R}_+\cup\{\infty\}$), where $\V$ is the so-called \emph{valuative tree} (see \cite{FJ04}). Therefore, it suffices to prove the following proposition, and in the proof, we will adopt some notations and conventions which can be referred to \cite{FJ04,FJ05a}.

\begin{proposition}\label{prop-qm.are.Zhou}
    For every quasi-monomial valuation $v\in\V_{\qm}$, there exists a large enough positive integer $N\gg 1$, such that $v$ is a Zhou valuation related to $\fk{m}^N$ up to rescaling.
\end{proposition}

\begin{proof}
    Due to \Cref{thm-A(v)=1-v(q)} and \Cref{cor-ZV.computes}, (as explained by the paragraph after \Cref{cor-ZV.computes}) we only need to prove that there exists $N\gg 1$ such that 
    \begin{equation}\label{eq-w.Vqm.only.v}
        \big\{w\in\V \colon w \ \text{computes} \ \lct^{\fk{m}^N}(\fk{a}_{\bullet}^v)\big\}=\{v\}.
    \end{equation} 

    Fix $N\in\bfZ_{\ge 0}$, which will be determined at last. We assume $v\neq w\in \V$ and $w$ computes $\lct^{\fk{m}^N}(\fk{a}_{\bullet}^v)$, i.e.,
    \[\lct^{\fk{m}^N}(\fk{a}_{\bullet}^v)=\inf_{\wt{v}\in\V_{\qm}}\frac{A(\wt{v})+N}{\wt{v}(\fk{a}_{\bullet}^{v})}=\frac{A(w)+N}{w(\fk{a}_{\bullet}^{v})}.\]

    We denote the last term in the above equality by $\sigma(w)$. If $w>v$, then $A(w)>A(v)$ by \cite[Section 3.6]{FJ04}. In addition, $w(\fk{a}_{\bullet}^{v})=1$ since $w>v$ and $w(\fk{m})=v(\fk{m})=1$. It follows that $\sigma(w)>\sigma(v)\ge \lct^{\fk{m}^N}(\fk{a}_{\bullet}^v)$, which shows $w$ can not compute $\lct^{\fk{m}^N}(\fk{a}_{\bullet}^v)$.

    Now we assume $w>v$ does not hold. Then the infimum $w\wedge v<v$, and either $w\wedge v=w$ or $w\wedge v<w$ holds. Due to \cite[Proposition 3.25]{FJ04} and the tree structure of $\V$, we have
    \begin{equation*}
        w(\fk{a}_{\bullet}^v)=\inf_{\phi}\frac{w(\phi)}{v(\phi)}=\inf_{\phi}\frac{m(\phi)\alpha(w\wedge v_{\phi})}{m(\phi)\alpha(v\wedge v_{\phi})}=\inf_{\phi}\frac{\alpha(w\wedge v_{\phi})}{\alpha(v\wedge v_{\phi})}=\frac{\alpha(w\wedge v)}{\alpha(v)},
    \end{equation*}
    where the infimum are over all irreducible $\phi\in\fk{m}$, $m(\phi)$ is the multiplicity of $\phi$, $\alpha$ is the \emph{skewness} (cf. \cite[Definition 3.23]{FJ04}), and $v_{\phi}$ is the \emph{curve valuation} (cf. \cite[Section 1.5.5]{FJ04}) defined by the curve $C=\{\phi=0\}$. Therefore, we have
    \[\sigma(w)=\frac{A(w)+N}{\alpha(w\wedge v)}\alpha(v)>\frac{A(w\wedge v)+N}{\alpha(w\wedge v)}\alpha(v)=\sigma(w\wedge v)\]
    whenever $w>w\wedge v$. Hence such a $w$ cannot compute the threshold.

    According to the above discussions, we can assume $w<v$, that is, $w\in\V_{\qm}$ lying in the segment $[v_{\fk{m}}, v)$, where the multiplicity valuation $v_{\fk{m}}=m(\,\cdot\,)$ is the root of the tree. Let
    \[v_{\fk{m}}=v_0<v_1<\ldots<v_g<v_{g+1}=v\]
    be the \emph{approximation sequence} (cf. \cite[Definition 3.45]{FJ04}) associated to $v$, which is finite since $v\in\V_{\qm}$. Then (cf. \cite[Section 3.6]{FJ04})
    \[A(v)=2+\sum_{j=1}^{g+1}m_j(\alpha_j-\alpha_{j-1}),\]
    where the \emph{multiplicity} (cf. \cite[Section 3.4]{FJ04}) $m(\wt{v})=m_j$ for $\wt{v}\in (v_{j-1}, v_j]$, and $\alpha_j=\alpha(v_j)$. We assume $w\in (v_g, v)$ first. Consider the parametrization of the segment $(v_g, v]$ given by the skewness: for $\wt{v}\in (v_g,v]$ with $\alpha(\wt{v})=t\in (\alpha_g, \alpha_{g+1}]$, denote $A(t)=A(\wt{v})$ and $\sigma(t)=\sigma(\wt{v})$. Then 
    \[A(t)=2+m_{g+1}(t-\alpha_g)+\sum_{j=1}^g m_j(\alpha_j-\alpha_{j-1})\]
    is a linear function in $t\in (\alpha_g, \alpha_{g+1}]$, and hence,
    \begin{equation}\label{eq-diff.of.sigma}
        \frac{\mathrm{d}\sigma(t)}{\mathrm{d}t}=\frac{\mathrm{d}}{\mathrm{d}t}\left(\frac{A(t)+N}{t}\cdot\alpha(v)\right)=\alpha(v)\frac{m_{g+1}\cdot t-A(t)-N}{t^2},
    \end{equation}
    for $t\in (\alpha_{g}, \alpha_{g+1}]$. Thus, if we take
    \begin{equation*}
        N>\max_{\alpha_g<t\le\alpha_{g+1}}\big|m_{g+1}\cdot t-A(t)\big|=\max_{v_g<\wt{v}\le v_{g+1}}\big|m(\wt{v})\alpha(\wt{v})-A(\wt{v})\big|,
    \end{equation*}
    then $\sigma(t)$ is strictly decreasing in $t\in (\alpha_g, \alpha_{g+1}]$, i.e., any $w\in (v_g, v_{g+1})$ can not compute $\lct^{\fk{m}^N}(\fk{a}_{\bullet}^v)$. Continuously, we can see that if we take
    \begin{equation}\label{eq-N.large.enough}
        N>\max_{v_{\fk{m}}\le \wt{v}\le v_{g+1}}\big|m(\wt{v})\alpha(\wt{v})-A(\wt{v})\big|,
    \end{equation}
    then $\sigma(t)$ is strictly decreasing in $t\in [1,v_{g+1}]$, that is, the valuation $w$ computing $\lct^{\fk{m}^N}(\fk{a}_{\bullet}^v)$ also can not lie in the segment $[v_{\fk{m}}, v)$.

    In conclusion, for every positive integer $N$ satisfying (\ref{eq-N.large.enough}), we have verified that (\ref{eq-w.Vqm.only.v}) holds, which implies that $v$ is a Zhou valuation related to the ideal $\fk{m}^N$ up to rescaling.
\end{proof}

In particular, for the trivial ideal $\fk{q}=(1)$ we have the following:

\begin{proposition}
    Let $\ZV(1)$ consist of those $v\in\V_{\qm}$ for which $v/A(v)$ is a Zhou valuation related to $(1)$. Then
    \[\ZV(1)=\big\{v\in\V_{\qm}\colon m(v)=1\big\}.\]
\end{proposition}

\begin{proof}
    The argument used in the first part of the proof of \Cref{prop-qm.are.Zhou} also shows that
    \[\ZV(1)=\big\{v\in\V_{\qm}\colon \sigma_0(w)>\sigma_0(v) \ \text{for all} \ w\in \V_{\qm} \ \text{with} \ w<v\big\},\]
    where
    \[\sigma_0(w)\coloneqq\frac{A(w)}{w(\fk{a}_{\bullet}^v)}=\frac{A(w)}{\alpha(w)}\alpha(v),\]
    for $w<v$.

    If $m(v)=1$, then for any $w<v$, we have $m(w)=1$, so
    \[\sigma_0(w)=\frac{1+\alpha(w)}{\alpha(w)}\alpha(v)>1+\alpha(v)=\sigma_0(v),\]
    by \cite[Proposition 3.48]{FJ04}. Hence, any $v\in \V_{\qm}$ with $m(v)=1$ belongs to $\ZV(1)$.

    Conversely, if $m(v)>1$, we can take
    \[v_g\coloneqq \max\big\{w\in [v_{\fk{m}}, v)\colon m(w)<m(v)\big\},\]
    which appears in the approximation sequence of $v$ satisfying $v_g<v$. Set $\alpha_g=\alpha(v_g)<\alpha(v)=\alpha_{g+1}$. Due to (\ref{eq-diff.of.sigma}), for $t\in (\alpha_g, \alpha_{g+1}]$, we have
    \[\frac{\mathrm{d}\sigma_0(t)}{\mathrm{d}t}=\alpha(v)\frac{m(w_t)\alpha(w_t)-A(w_t)}{\alpha(w_t)^2},\]
    where $t=\alpha(w_t)$ for $w_t\in (v_g, v]$. However, it follows from \cite[Proposition 3.48]{FJ04} that $A(w_t)<m(w_t)\alpha(w_t)$ as $m(w_t)=m(v)>1$ for all $w_t\in (v_g, v]$. Thus, the left derivative of $\sigma_0$ at $t=\alpha(v)$ is strictly positive, which shows that $v$ does not attain the infimum of $\sigma_0(\,\cdot\,)$ in $\V_{\qm}$ (in fact, one can show that the infimum is attained at $v_1=\max\{w\colon w\le v, \ m(w)=1\}$). Therefore, $v\notin \ZV(1)$ when $m(v)>1$.
\end{proof}

\end{document}